\documentclass[12pt]{amsart}

\usepackage{amssymb}

\calclayout

\usepackage{color,ulem}

\numberwithin{equation}{section}
\theoremstyle{plain}
\newtheorem{theorem}[equation]{Theorem}

\newtheorem{lemma}[equation]{Lemma}
\newtheorem{corollary}[equation]{Corollary}

\theoremstyle{remark}
\newtheorem{remark}[equation]{Remark}

\theoremstyle{definition}
\newtheorem{definition}[equation]{Definition}
\newtheorem{defn}[equation]{Definition}

\newtheorem{question}[equation]{Question}
\newtheorem*{question*}{Question}
\newtheorem{example}[equation]{Example}

\def\vint_#1{\mathchoice%
        {\mathop{\kern 0.2em\vrule width 0.6em height 0.69678ex depth -0.58065ex
                \kern -0.8em \intop}\nolimits_{\kern -0.4em#1}}%
        {\mathop{\kern 0.1em\vrule width 0.5em height 0.69678ex depth -0.60387ex
                \kern -0.6em \intop}\nolimits_{#1}}%
        {\mathop{\kern 0.1em\vrule width 0.5em height 0.69678ex
            depth -0.60387ex
                \kern -0.6em \intop}\nolimits_{#1}}%
        {\mathop{\kern 0.1em\vrule width 0.5em height 0.69678ex depth -0.60387ex
                \kern -0.6em \intop}\nolimits_{#1}}}
\def\vintslides_#1{\mathchoice%
        {\mathop{\kern 0.1em\vrule width 0.5em height 0.697ex depth -0.581ex
                \kern -0.6em \intop}\nolimits_{\kern -0.4em#1}}%
        {\mathop{\kern 0.1em\vrule width 0.3em height 0.697ex depth -0.604ex
                \kern -0.4em \intop}\nolimits_{#1}}%
        {\mathop{\kern 0.1em\vrule width 0.3em height 0.697ex depth -0.604ex
                \kern -0.4em \intop}\nolimits_{#1}}%
        {\mathop{\kern 0.1em\vrule width 0.3em height 0.697ex depth -0.604ex
                \kern -0.4em \intop}\nolimits_{#1}}}

\usepackage{graphicx}

\usepackage{url}

\newcommand{\fint}{\avint}

\title[Almost uniform domains]{Almost uniform domains and Poincar\'e inequalities}
\date{\today}

\author{Sylvester Eriksson-Bique}
\thanks{S.E.-B.~was partially supported by
NSF grant DMS\#--1704215 (U.S.).}
\address{S.E--B.: Research Unit of Mathematical Sciences,
P.O.Box 3000,
FI-90014, Oulu, Finland}
\email{\tt sylvester.eriksson-bique@oulu.fi}

\author{Jasun Gong}
\address{J.G.: Mathematics Department, Fordham University, John Mulcahy Hall Bronx, NY 10458-5165}
\email{jgong7@fordham.edu}


\newcounter{prob}
\newcommand{\Z}{\ensuremath{\mathbb{Z}}}
\newcommand{\E}{\ensuremath{\mathbb{E}}}
\newcommand{\N}{\ensuremath{\mathbb{N}}}
\newcommand{\R}{\ensuremath{\mathbb{R}}}

\newcommand{\nseq}{\ensuremath{{\bf n}}}

\newcommand{\diam}{\operatornamewithlimits{ \rm diam}}
\newcommand{\Lip}{\ensuremath{\mathrm{Lip\ }}}

\newcommand{\defeq}{\mathrel{\mathop:}=}

\newcommand{\nC}{\ensuremath{\overline{C}}}

\newcommand{\nD}{\ensuremath{\overline{D}}}
\newcommand{\mT}{\ensuremath{\mathcal{T}}}
\newcommand{\mR}{\ensuremath{\mathcal{R}}}
\newcommand{\nmR}{\ensuremath{\overline{\mathcal{R}}}}
\newcommand{\ngam}{\ensuremath{\overline{\gamma}}}

\newcommand{\len}{\ensuremath{\mathrm{Len}}}
\newcommand{\undf}{\ensuremath{\mathrm{Undef}}}
\newcommand{\gap}{\ensuremath{\mathrm{Gap}}}

\def\Xint#1{\mathchoice
{\XXint\displaystyle\textstyle{#1}}%
{\XXint\textstyle\scriptstyle{#1}}%
{\XXint\scriptstyle\scriptscriptstyle{#1}}%
{\XXint\scriptscriptstyle\scriptscriptstyle{#1}}%
\!\int}
\def\XXint#1#2#3{{\setbox0=\hbox{$#1{#2#3}{\int}$ }
\vcenter{\hbox{$#2#3$ }}\kern-.58\wd0}}
\def\avint{\Xint-}

\newcommand{\co}{\mskip0.5mu\colon\thinspace}   

\begin{document}

\maketitle

%


\begin{abstract}
Here we show existence of 
{
many
subsets of Euclidean
}
spaces that, despite having empty interior, still support Poincar\'e inequalities
{
with respect to the restricted Lebesgue measure.%
}  
Most importantly, 
{
despite the explicit constructions in our proofs,%
}
our methods do not depend on any rectilinear or self-similar structure of the underlying space.  We instead employ the 
uniform domain condition of
Martio and Sarvas. 
This
condition relies on the measure density of such subsets, as well as the regularity and relative separation of their boundary components.  

In doing so, our results hold true for metric spaces equipped with doubling measures and Poincar\'e inequalities in general, and for the Heisenberg groups in particular. To our knowledge, these are the first examples of such subsets on any 
(non-abelian)
Carnot group.
Such subsets also give 
new examples of Sobolev extension domains, 
also in the general setting of 
doubling metric measure spaces.

{
In the Euclidean case, our construction also 
includes
the non-self-similar Sierpi\'nski carpets of Mackay, Tyson, and Wildrick, as well as higher dimensional analogues not treated in the literature.
}
When specialized to the plane, 
our
results lead to 
{
new,%
}
general sufficient conditions for a planar subset to be 2-Ahlfors regular and to satisfy the Loewner condition. Two of these conditions, uniform separation and regularity of boundary components, are also necessary. The sufficiency is obtained with an additional measure density condition.

\end{abstract}

\tableofcontents

\section{Introduction}

\subsection{Poincar\'e inequalities and Sierpi\'nski sponges}

{
Let 
$(X,d)$
be a complete metric space that supports a 
{
doubling%
}
measure 
$\mu$.
We wish to understand the following question:\ 

\begin{itemize}
\item[]
\textit{
If $X$ supports 
a $(1,p)$-Poincar\'e inequality, 
then
when does a subset $Y$ of $X$, equipped with its restricted measure and metric, 
support a $(1,q)$-Poincar\'e inequality, and for which exponents $q\in [1,\infty)$?
}
\end{itemize}

\noindent
This question is motivated by the desire to construct 
a new, general class of examples that include
so-called uniform domains 
and more generally,
Sobolev extension domains. Below, our main results will give 
criteria to guarantee such examples, in
both the Euclidean and the general metric space setting. 
To this end, we begin with 
some definitions.

\begin{definition} \label{def:measuredoubling}
Let $r_0 > 0$. A proper metric measure space $(X,d,\mu)$ with a Radon measure $\mu$ is said to be
{\sc $D$-doubling at scale $r_0$} --- or {\sc $(D,r_0)$-doubling} for short ---
if for
all $r \in (0,r_0)$ and any $x \in X$ we have
$$
0 < \mu(B(x,2r)) \leq D \mu(B(x,r)).
$$
If $(X,d,\mu)$ is $D$-doubling at scale $r_0$ for all $r_0 > 0$, then $X$ is said to be {\sc $D$-doubling}.
\end{definition}

We will assume that the support of the measure equals the space, ${\rm supp}(\mu)=X$.

\begin{definition} \label{def:PI}
Let $r_0>0$ and $1 \leq p  < \infty$. A proper metric measure space $(X,d,\mu)$ with a Radon measure $\mu$ 
is said to satisfy a 
{\sc $(1,p)$-Poincar\'e inequality at scale $r_0$}
(with constant $C \geq 1$) 
 if for all Lipschitz functions $f:\ X \to \mathbb{R}$ and all $x \in X$ and $r \in (0,r_0)$ we have for $B \defeq B(x,r)$ and $CB\defeq B(x,Cr)$
\begin{equation}\label{eq:defPI}
\fint_B |f-f_B| ~d\mu \leq C r  \left(\fint_{C B} \Lip [f]^p ~d\mu \right)^{\frac{1}{p}}.
\end{equation}
If $r_0 = \infty$, then say that $X$ 
satisfies a {\sc (global) $(1,p)$-Poincar\'e inequality} (with the same constants).
\end{definition}

A space satisfying a Poincar\'e inequality and the doubling property is called a \emph{PI-space}. 

Here, for any measurable and locally integrable $f \co X \to \R$ its average value on a ball is
$$
f_B \defeq \fint_B f ~ d\mu \defeq \frac{1}{\mu(B)} \int_B f ~d\mu,$$
and its pointwise Lipschitz constant is
$$
\Lip [f](x) \defeq \limsup_{y \to x} \frac{|f(x)-f(y)|}{d(x,y)}.
$$
In the literature, there are different definitions of Poincar\'e inequalities, all of which coincide with our definition in the case of complete metric spaces. For a detailed discussion of these issues we refer to \cite{keith2003modulus, hajlasz1995sobolev, Heinonen2000}.

Poincar\'e inequalities play a profound role in analysis and the regularity of functions.  In the general setting of metric measure spaces,
they are crucial hypotheses for
nontrivial definitions of generalized Sobolev spaces \cite{hajlaszmetricsobolev, ChDiff99, shanmugalingamsobolev} and differentiability of Lipschitz functions \cite{ChDiff99}. 
Moreover, open subsets $\Omega \subset X$ supporting a $(1,p)$-Poincar\'e inequality and with a 
lower bound 
on their measure density
are important examples of sets admitting extensions of Sobolev spaces. See \cite{jonesuniform,koskelaextension,bjornuniform} 
and below
for more related historical discussion and references. We remark, that applying our work there requires some care, as our constructions lead to closed sets. However, one can also consider Sobolev extension problems with other gradients which make sense also for closed sets. 

Poincar\'e inequalities also play a profound role in the study of geometry of metric spaces, specifically in regards to quasiconformal mappings between them \cite{heinonen1998quasiconformal}.
Planar metric spaces that are Ahlfors $2$-regular and that support a $(1,2)$-Poincar\'e inequality are examples of sets which admit uniformization by slit carpets, see \cite[Section \S 7]{mackaytysonwildrick}. Such inequalities are also important in determining the so-called \textit{conformal dimension} of a space \cite{mackaytysonbook}.  In general, conformal dimension measures the extent to which Hausdorff dimension can be lowered by quasisymmetric maps, and 
it is known that
any Ahlfors regular space satisfying a Poincar\'e inequality has conformal dimension equal to its Hausdorff dimension.

However, a 
good understanding of 
the geometric conditions that guarantee such inequalities, in particular for subsets, has remained a challenge.
Particular examples of subsets in the plane satisfying Poincar\'e inequalities were given by Mackay, Tyson and Wildrick \cite{mackaytysonwildrick}. We briefly 
discuss a construction 
here that includes theirs.

Let ${\bf n} = (n_i)_{i=1}^\infty$ be a 
sequence of odd positive integers with $n_i \geq 3$. 
As
a convention, put
$$
{\bf n}^{-1} ~=~ \bigg( \frac{1}{n_i} \bigg)_{i=1}^\infty.
$$
Fix a dimension $d \geq 2$.  We define the {\sc Sierpi\'nski sponge associated to $\nseq$ in $\R^d$} as follows:\

\begin{enumerate}
\item
At the first stage,
put
$S_{0,\nseq}=[0,1]^d$ and 
$T_{0,\nseq}^1 = [0,1]^d$ and 
$\mT_{0,\nseq}=\{T_{0,\nseq}^1\}$.

\item
Assuming that we have defined
sets $S_{k,\nseq}$ and $T_{k,\nseq}^j$
and collections of cubes
$\mT_{k,\nseq}$ at the $k$th stage, 
for $k \in \N$,

\begin{itemize}
\item
Subdivide each $T \in \mT_{k,\nseq}$
into $(n_{k+1})^d$ equal subcubes; 
\item
Excluding the central subcube in $T$, 
index the remaining subcubes in any fashion as $T_{k+1, \nseq}^j$
and let $\mT_{k+1, \nseq} = \{T_{k+1,\nseq}^j\}$ be the 
collection of all 
such subcubes.
{
We note that for $k \in \N$,
the side length of each 
subcube $T_{k,\nseq}^j$ is
therefore
$$
s_k = \prod_{i=1}^k \frac{1}{n_i}.
$$
(For consistency, let $s_0=1$.)
}\item 
Define
the {\sc $k+1$'th order pre-sponge} 
as the set 
$$
S_{k+1, \nseq} = \bigcup_{T \in \mT_{k+1, \nseq}} T.
$$
\end{itemize}

\item
For technical purposes later, let $k \geq 1$ and 
define $\mR_{\nseq, k}$ to be the 
sub-collection
of central cubes removed from cubes $T \in \mT_{k-1,\nseq}$ at the $k$'th stage 
and put
$$
\nmR_{\nseq, k}=\bigcup_{l=1}^k \mR_{\nseq,l}.
$$
\end{enumerate}

The {\sc Sierpi\'nski sponge associated to the sequence $\nseq$} is then defined as

\begin{equation}
\label{eq:sierdef}
S_{\nseq} = \bigcap_{k \geq 0} S_{k,\nseq}.
\end{equation}

When $d=2$ we also refer to these sets as Sierpi\'nski carpets,  and the constant sequence $\nseq=(3,3,3\ldots)$ yields the usual ``middle-thirds'' Sierpi\'nski carpet, which is denoted by $S_3$.

The main result by Mackay, Tyson and Wildrick \cite{mackaytysonwildrick} states that Sierpi\'nski carpets with  positive 
Lebesgue
measure satisfy Poincar\'e inequalities. 
Their proof
was a \textit{tour de force} in 
constructing
so-called Semmes families of (rectifiable) curves and 
then applying
a characterization of Poincar\'e inequalities from Keith \cite{keith2003modulus}. (For precise definitions and a further discussion, see \cite{semmescurves}.) 

However, even slight variations of 
their
construction, such as removing a ``nearly central'' square instead of a central one, would require a new construction of a curve family
{
with new, equally technical details to check.%
}
Our motivation was therefore to find 
{
more general and robust%
}
methods that apply to \textit{all dimensions}, as well as to \textit{non-euclidean geometries} too. 

First of all, our methods lead to the following higher dimensional analogue of their result.


\begin{theorem}\label{thm:ndimSiercarpet} Let ${\nseq} = (n_i)$ be a sequence of odd integers with $n_i \geq 3$, and let $d \geq 2$. Then
 the following conditions are equivalent for
 the Sierpi\'nski sponge $S_{\nseq}$ in $\R^d$:\

\begin{enumerate}
\item ${\bf n}^{-1} \in \ell^d(\N)$;
\item The space $(S_{\nseq}, |\cdot|, \lambda)$ satisfies a $(1,p)$-Poincar\'e inequality for all $p>1$;
\item The space $(S_{\nseq}, |\cdot|, \lambda)$ satisfies a $(1,p)$-Poincar\'e inequality for some $p>1$.
\end{enumerate}
In addition, we have the following 
{
complementary%
}
case:
\begin{enumerate}
\item[(4)] If $S_{\nseq}$ has zero $\lambda$-measure, then there is no $D$-doubling measure $\mu$, for any $D \in [1,\infty)$, such that $(S_{\nseq}, |\cdot|,\mu)$ satisfies a $(1,p)$-Poincar\'e inequality with $p \in [1,\infty)$.
\end{enumerate}
\end{theorem}

{
The 
borderline
case of $p=1$ can also be fully characterized in terms of $\nseq$. The case of $d=2$ appeared before in \cite{mackaytysonwildrick}. 
The general borderline case for all $d\geq 2$ is presented in a separate paper by the authors \cite{sylvesterjasun},
and the approach involves substantially different 
methods.}

A crucial aspect of 
our
theorem is the sharp characterization of the exponents $p$. In what follows, we also obtain  essentially sharp characterizations for the given ranges of exponents
{
in more general Euclidean constructions,%
}
and even in the general metric space context!

\subsection{The planar Loewner problem} \label{subsect:introplanarloewner}

Motivated by this result, we consider general sets of the form $Y = 
\R^d
\setminus \bigcup_{R \in \mathcal{R}} R$,
{
for some countable collection $\mathcal{R}$ of 
open subsets
and study when $Y$ inherits 
a
Poincar\'e inequality.
(Bear in mind that the elements of $\mathcal{R}$ 
will still have good geometric properties, but
are not necessarily polyhedral, or even Lipschitz.)  

The case of $d=2$ is particularly interesting,
due to the connections with quasiconformal geometry.
In particular, for $d=2$ the conditions  given for $\mathcal{R}$ 
are not only sufficient, but also close to necessary.  This also gives a partial answer to the following question:

\begin{question}[``Planar Loewner problem''] \label{ques:planarLoewner}
Classify 
all
closed subsets of the plane which are Ahlfors $2$-regular and $2$-Loewner.
\end{question}

Though we will not explicitly define the Loewner condition here, we 
recall
that a closed Ahlfors $2$-regular subset is $2$-Loewner if and only if it satisfies a $(1,2)$-Poincar\'e inequality; for a more general definition and further discussion, see \cite{heinonen1998quasiconformal}.


Though natural to
pose,
this question hasn't been extensively studied in the literature.  Prior results exist only for some 
specific cases.
We now give a 
new, general, and sufficient 
condition for an affirmative answer to 
this
problem. To formulate it, consider collections of removed sets $\mathcal{R}$ and subcollections of sets that meet a given ball
$B(x,r)$, 
$$
\mathcal{R}(x,r) ~=~
\{ R \in \mathcal{R} : R \cap B(x,r) \neq \emptyset \},
$$
and 
consider further,
for $N \in \N$, an `$N$-fold density function' for $\mathcal{R}$ relative to balls:
\begin{align} 
\label{eq:densityfunction}
s_N(x,r) ~=&~
\inf\Big\{
\sum_{ R \in \mathcal{R}(x,r) \setminus I } \frac{\lambda(R)}{r^2} :
I \subset \mathcal{R}, |I| \leq N
\Big\},
\end{align}
where $\lambda(R)$ denotes the usual area, or Lebesgue measure, of $R$. 

\begin{theorem}\label{thm:planar} A closed subset $Y$ of $\R^2$ satisfies a $(1,p)$-Poincar\'e inequality for every $p \in (1,\infty)$ if
it is of the form
$$
Y = \Omega \setminus \bigcup_{R \in \mathcal{R}} R,
$$
where the following conditions hold for $\Omega$ and for $\mathcal{R}$, 
{
for some constants $K \geq 1$ and $s > 0$:
}
\begin{enumerate}
\item The set $\Omega$ is closed, each $R \in \mathcal{R}$ is open, and each boundary $\partial \Omega$ and $\partial R$ for $R$ is 
a $K$-quasicircle;
\item $\mathcal{R}$ is uniformly relatively $s$-separated, that is:\ for all $R,R' \in \mathcal{R} \cup \{ \partial\Omega \}$,
$$
\Delta(R,R') \defeq 
\frac{d(R,R')}{\min(\diam(R),\diam(R'))} \geq s;
$$
\item 
There exists $N \in \N$ such that
$$
\limsup_{r \to 0} \sup_{x \in Y} s_N(x,r) = 0.
$$
\end{enumerate}
\end{theorem}


Indeed, 
Condition (3) requires the density of $\mathcal{R}$ at any $x \in Y$ to vanish, 
but allowing 
at each scale $r$
for the $N$ largest ``obstacles''
in $\mathcal{R}$
to be excluded. 
A slightly stronger statement, which allows for the density only becoming sufficiently small, is given in Theorem \ref{thm:densitycarp}.

Theorem \ref{thm:planar} is new even when the collection of obstacles $R$ and $\Omega$ have simple geometry, 
such as when $\Omega$ and every $R$ are disks. 
It is known from
\cite[Corollary 1.9]{mackaytysonwildrick} 
that \textit{there exist} subsets of this form with empty interior and which satisfy a $(1,2)$-Poincar\'e inequality.  Such sets, called {\sc circle carpets}, are constructed 
implicitly
via uniformization and can therefore only be approximated numerically. In contrast, here we give a procedure that yields \textit{explicit} circle carpets satisfying Poincar\'e inequalities, with a sharp characterization of the range of exponents. This flexibility extends to other shapes and higher dimensions, as described in Corollary \ref{thm:euclideansponges} below. 

To reiterate, the conditions for the sets $R \in \mathcal{R}$ come in three forms: the regularity of their boundaries, their separation, and their density. The first two conditions in the statement are necessary for a subset to be Loewner, 
as given in
Theorem \ref{thm:necessary} below. 
These conditions 
also appear elsewhere in the literature;
for instance, 
they are
the relevant conditions in 
Bonk's work on uniformization of planar subsets \cite{bonkuniform}. 
Moreover,
the conditions on summability also bear close resemblance to the summability conditions arising in other work on uniformizing planar metric spaces \cite{merkenov,hakobyan}.

\subsection{Metric spaces and Carnot groups} \label{subsect:intrometriccarnot}

In the proof of Theorem \ref{thm:planar}, the 
most
crucial feature  
about the collection $\mathcal{R}$ 
is that 
$\R^2 \setminus R$ is a {\sc uniform} domain, 
for each $R \in \mathcal{R}$.
Such sets were first studied in \cite{martiosarvas,vaisala1988}; see Definition \ref{def:unifdom}. Roughly speaking, these correspond to domains $\Omega$ without ``outer cusps.'' Domains in Euclidean space with Lipschitz boundaries are uniform domains, for example, 
in all dimensions.

In fact, 
uniformity
is a purely metric  property. 
A crucial result of Bj\"orn and Shanmugalingam asserts that uniform domains $\Omega$ in a doubling metric measure space $X$ inherit a Poincar\'e inequality from $X$; see \cite{bjornuniform}. 
Motivated by this, 
we therefore formulate a more general 
theorem for metric spaces. 

To this end, 
call a domain
{\sc co-uniform} if its complement is uniform and its boundary is connected. 
The uniform sparseness condition, 
mentioned 
below, combines 
Conditions (2) and (3) in Theorem \ref{thm:planar} above;
for precise statements, see Definitions \ref{def:sparsecoll} and \ref{def:exteriorunif}. Note that 
the sequence 
$\mathbf{n}$ plays an analogous role as the 
one
in Theorem \ref{thm:ndimSiercarpet}, in that it handles the density of the omitted subsets.

\begin{theorem} \label{thm:metrisponge}
Let $X$ be an Ahlfors $Q$-regular complete metric measure space 
admitting a $(1,p)$-Poincar\'e inequality, and let $\nseq$ be
a sequence of positive integers with ${\nseq}^{-1} \in \ell^Q(\N)$.

If $\Omega$ is a bounded, $A$-uniform subset of $X$ and if $\{\mR_{\nseq,k}\}_{k=1}^\infty$ is a 
uniformly $\nseq-$sparse collection of 
co-uniform
subsets of $\Omega$, then 
the set
$$
S_{\nseq} = \Omega \setminus \bigcup_{k} \bigcup_{R \in \mR_{\nseq,k}} R,
$$
with its restricted measure and metric, is Ahlfors $Q$-regular and satisfies a $(1,q)$-Poincar\'e inequality for each $q>p$. 
Moreover, 
\begin{itemize}
\item 
if $p>1$, then it also satisfies a $(1,p)$-Poincar\'e inequality; \item 
if the union of all sets from $\mR_{ \nseq,k}$, over all $k\in \N$, is dense in $\Omega$, then $S_{\nseq}$ has empty interior.
\end{itemize}
\end{theorem}


The ranges of the exponents in 
Theorem \ref{thm:metrisponge}
are sharp. In particular, 
only for $p=1$ do such removals of sets lead to a loss in range, namely the loss of the $(1,1)$-Poincar\'e inequality; see \cite{mackaytysonwildrick} for an example.
For $p>1$ no such loss occurs, 
due to the seminal self-improvement result of Keith and Zhong \cite{keith2008poincare}.

For some spaces, such as the Heisenberg group in particular and step-2 Carnot groups in general, the existence of uniform domains is well-known, at all scales and locations within these spaces.
In such cases, Theorem \ref{thm:metrisponge} can be used to give new examples of subsets with Poincar\'e inequalities and empty interior; see Subsection \S\ref{sect:heisenberg} for these 
examples, 
as well as some of the definitions relevant to these geometries. Due to a recent result by T.\ Rajala \cite{rajalaunif}, it is likely that the result applies to any Carnot group.

\subsection{Sobolev extension domains}

As a corollary of our theorems, we obtain many new examples of Sobolev extension domains, both in Euclidean and non-Euclidean spaces. To wit, an open subset $\Omega \subset X$ is called a 
{
(Sobolev)%
}
extension domain if there exists a bounded extension operator $E : N^{1,p}(\Omega) \to N^{1,p}(X)$; 
in the case where $\Omega$ is open in $X=\R^d$
the Newtonian Sobolev space $N^{1,p}(\Omega)$,
as introduced in 
\cite{shanmugalingamsobolev}, coincides with the classical Sobolev space $W^{1,p}(\Omega)$.  This definition, when employing $N^{1,p}(\Omega)$, 
makes sense even for closed subsets $\Omega$, while classically the interest has been mostly for open domains. However, the case of closed sets, as well as the relationship between open and closed extension domains is subtle.

The first examples of 
extension domains
were given by Jones \cite{jonesuniform}. In general, a sufficient condition for $\Omega$ to be an extension domain is if $\Omega$ supports a $(1,q)$-Poincar\'e inequality for $q<p$.
This condition, however, is not necessary unless $p$ is sufficiently large, as discussed in \cite{bjornuniform}. 

It remains a difficult problem to give both necessary and sufficient conditions for a domain to be an 
extension domain. In fact, this has essentially been solved only for simply connected domains in the plane \cite{zhangthesis}. Our examples give flexible constructions of infinitely connected domains in $\R^d$ for $d \geq 2$, 
as well as in
step-2 Carnot groups and in general metric spaces, that are Sobolev extension domains. These examples are new even in the planar setting. 
See \cite{bjornuniform,koskelaextension} for more related discussion and references, as well as the Ph.D.\ thesis \cite{zhangthesis}.

\subsection{Methodology:\ removing subsets vs.\ ``fillings'' of spaces}

Thus far, the results in this article apply to subsets $Y$ obtained by removing, from an initial set, infinite collections 
$\mR$
of well-behaved subsets at all locations and scales. 
As we will see 
later, these results are special cases of Theorem \ref{thm:PIthm} and Corollary \ref{cor:almostunifcon}, where such sets $Y$ are viewed from a different perspective.  In particular, we view the intermediate sets $\Omega_r$, 
each
obtained by removing 
a
finite sub-collection of 
subsets in $\mR$ up 
to a given scale $r>0$, as good approximations (or ``fillings'') of $Y$; 
in particular, each $\Omega_r$ is doubling and supports a Poincar\'e inequality, both at scale $r$, and $\Omega_r$ also contains $Y$ with small complement.

In fact, these three properties \textit{alone} are sufficient for $Y$ to support a Poincar\'e inequality, provided that the associated constants 
are uniform in $r$.  No explicit removals of sets are actually needed for our proofs; the fillings $\Omega_r$ only need to satisfy these properties axiomatically, and they need not be defined, a priori, in terms of any removed set.  Similarly as for Sobolev extension domains \cite{koskelaextension}, it is the \textit{measure density} of the sets $\Omega_r$ that is crucial. (In fact, the small-ness of $\Omega_r \setminus Y$ is given in terms of measure density; see Definition \ref{def:PIfilling}.)

The sufficiency of these properties in turn relies crucially on a new characterization of Poincar\'e inequalities, as studied by the first author \cite{sylvester:poincare,sylvester:keithzhong}.
Roughly speaking, spaces supporting a
Poincar\'e inequality 
cannot ``see'' sets of small density: 
points that have small measure density, relative to a given set, can be connected by a 
quasi-geodesic
that meets that set in correspondingly small length. 
This correspondence, moreover, depends quantitatively but nontrivially on the exponent $p$.  Since we formulate density in terms of 
maximal functions, we refer to this characterization as ``maximal $p$-connectivity.''

Intuitively, $\Omega_r$ provides improved behavior for $Y$ without adding much density. 
Once such fillings are available, 
pairs of points in $Y$ that are at most a distance $r$ apart can be joined by 
rectifiable
curves inside $\Omega_r$.  Such curves
may not lie entirely in $Y$, but as the measure density of $\Omega_r \setminus Y$ is small, by maximal connectivity there must be curves which spend little time in this set. The ``bad'' portions 
of these curves can then be removed and replaced by ``good'' portions, via %
a delicate iteration argument. 


This filling process is subtle, and the dependence of
the exponent $p$ 
on the quality of the filling is nontrivial. This will be illustrated in the examples below in Subsection \S \ref{subsec:PIfilling}.
 
 Interestingly, we avoid throughout this paper any discussion about the modulus of curve families, and we do not construct any curve families to estimate such moduli. However, in recent work it is shown that such curve families always exist on spaces satisfying Poincar\'e inequalities. Thus, our tools can be considered to implicitly construct Semmes families of curves. See \cite{dimarinosavare,nagesetalambv}.

\subsection{General structure of paper}

In Section \S\ref{sec:intermediate} we first recall 
basic notions and relevant notation, 
and then give precise definitions for fillings of subsets.  The section  
concludes with the statement of our main result, Theorem \ref{thm:PIthm}, as well as auxiliary results and the strategy of the proof.


In Section \S \ref{sec:proof} we 
prove Theorem \ref{thm:PIthm}; 
it states that 
subsets admitting such fillings, or ``fillable subsets,'' must also 
satisfy Poincar\'e inequalities.  The proof requires Theorems \ref{thm:classification1} and \ref{thm:classification2}, which are characterizations of $(1,p)$-Poincar\'e inequalities and will be proven later.  

In Section \S\ref{sec:almostunif} we 
apply Theorem \ref{thm:PIthm}
first to Sierpi\'nski sponges, and then to 
general metric measure spaces 
with co-uniform domains removed.
We conclude this section with 
new examples of subsets of the Heisenberg group that satisfy Poincar\'e inequalities, as well as
a discussion of
our sufficient condition for planar Loewner subsets. 
All of these applications 
use the results in \S\ref{sec:intermediate}, but readers may choose to see how these results are applied first, before reading 
those technical proofs.
(%
To preserve the flow of discussion the proofs of certain technical results, such as Theorem \ref{thm:cutout}, are postponed to Appendix \ref{a:cutout}.
)

Lastly, in Section \S\ref{sect:classification} we 
prove Theorems \ref{thm:classification1} and \ref{thm:classification2} by introducing a certain ``path-connectivity'' function associated to metric measure spaces.  
(Readers who are primarily interested in the classification of Poincar\'e inequalities may opt to read Section \S\ref{sect:classification} independently of the 
other sections.)
In Appendix \ref{a:cutout}, we prove Theorem \ref{thm:cutout}, as well as other auxiliary results about uniform domains.

\section{Intermediate results}\label{sec:intermediate}

\subsection{Notation and Basic Notions}

Throughout the paper, we will work on complete and proper metric measure spaces $X$ equipped with some Radon measure $\mu$.  
{
Consistently, $Y$ refers to 
a closed subset of $X$ 
which will be shown to support
Poincar\'e inequalities. In the Euclidean case 
$X=\R^n$
we will also
denote such subsets by $S$, suggestively for
``sponge''.}

\begin{remark}[Types of constants] \label{rmk:local}
As a 
convention, we refer to certain constants as {\sc structural constants} if they describe fixed parameters for standard hypotheses or conditions. These include the doubling constant $D \geq 1$, the constant $C\geq 1$ in the Poincar\'e inequality (as well as uniformity constants $A>0$ that imply such inequalities), the choice of exponent $p \geq 1$, and the scale parameter $r_0 > 0$. 

Moreover, conditions on a metric space $X$ that depend on the scale parameter --- i.e.\ an 
upper (distance) bound between points on $X$ --- are referred to as {\sc local} conditions.  In particular, a {\sc locally $D$-doubling metric measure space} refers to a $(D,r_0)$-doubling metric measure space for some $r_0 > 0$ and a {\sc local $(1,p)$-Poincar\'e inequality} refers to a $(1,p)$-Poincar\'e inequality that is valid at scale $r_0$, for some $r_0>0$.  

The same convention will apply to other conditions in the sequel. Note, in this convention, the scale $r_0$ is 
 assumed to be 
uniform throughout the space.  Our convention is therefore slightly different from others, such as in \cite{bjornlocal}, where the scale can vary 
with the point.
\end{remark}

Open balls in a metric space are denoted by  $B=B(x,r)$, and their inflations by $CB=B(x,Cr)$, despite the ambiguity that balls may not be uniquely defined by their radii. 
If multiple metrics are used, we indicate the one used with a subscript, e.g. $B_d(x,r)$ to mean the ball with respect to the metric $d$. 

By a curve $\gamma$ in a metric space $X$ we mean a Lipschitz
map $\gamma \co I \to X$, where $I \subset \R$ is a bounded closed interval. 
As a convention, we assume that all rectifiable curves are {\em parametrised by arc-length} unless otherwise specified, in which case 
it
satisfies $\Lip(\gamma)(t) \defeq \limsup_{s \to t}\frac{d(\gamma(t),\gamma(s))}{|s-t|} \leq 1$, $t\in I$.

A metric space $X$ is called {\sc $\Lambda$-quasiconvex} if for every $x,y \in X$ there exists a curve $\gamma$ connecting $x$ to $y$ with $\len(\gamma) \leq \Lambda d(x,y)$. Such a curve $\gamma$, when it exists, is called a {\sc $\Lambda$-quasi-geodesic}. A space $X$ is called {\sc $\Lambda$-quasiconvex at scale $r_0>0$}, if the same holds for every $x,y \in X$ with $d(x,y) \leq r_0$. 

Frequently, we restrict the metric and measure onto some subset $A \subset X$. On $A$ the measure is denoted $\mu|_A$, and $d|_{A \times A}$, but we will often avoid this cumbersome notation. Also, metric balls in $A$ are simply intersections $B_{d|_{A\times A}}(x,r) = B(x,r) \cap A$, and they are denoted occasionally by $B_A(x,r)$.

Related to Definition \ref{def:measuredoubling}, a metric space $X$ is said to be {\sc $N$-metric doubling}, for some $N \in \N$, if for every ball $B(x,r)$ there exist $x_1, \dots, x_m \in X$ for some $m \leq N$ such that
\[B(x,r) \subset \bigcup_{i=1}^m B(x_i, r/2).\]
Clearly, every metric space equipped with a $D$-doubling measure is $D^4$-metric doubling. 
Later we will specialize to doubling measures with certain quantitative growth, as below.

\begin{definition} \label{def:arreg}
A proper metric measure space $(X,d,\mu)$ 
is said to be
{\sc Ahlfors $Q$-regular} with constant $C>0$ if for
all $0<r<\diam(X)$ and any $x \in X$ we have
$$
\frac{1}{C} r^Q \leq \mu(B(x,r)) \leq Cr^Q.
$$
The space is said to be {\sc Ahlfors $Q$-regular} up to scale $r_0$ if the same holds for $r\in (0,r_0)$.
\end{definition}


%
%


We define the centered Hardy-Littlewood maximal functions as
\begin{equation}\label{eq:maxfunc}
M f(x) \defeq \sup_{0<r} \vint_{B(x,r)} f ~d\mu.
\end{equation}
$$
M_s f(x) \defeq \sup_{r \in (0,s)} \vint_{B(x,r)} f ~d\mu.
$$

Here and in what follows,
we will use a localized version of the Maximal Function Theorem, 
see \cite[Theorem 2.19]{Mattila1999}.
The proof below, given for completeness, is a slight modification of the classical argument.

\begin{lemma} \label{lem:localmaximal}
If $X=(X,d,\mu)$ is a $D$-doubling metric measure space at scale $8R$, then
$$
\mu\left( \{ M_Rf > \lambda \} \cap B(x,r) \right) \leq \frac{D^3 \|f|_{B(x,r+R)}\|_{L^1}}{\lambda}
$$
for all $x \in X$, all $f \in L^1(X)$, and all $r,R,\lambda > 0$.
\end{lemma}

\begin{proof}
Put $E_\lambda := \{ M_Rf > \lambda \} \cap B(x,r)$.  For each $y \in E_\lambda$ there exists $r_y \in (0,R)$ so that
\begin{equation} \label{eq:level}
\int_{B(y,r_y)} |f| ~d\mu > \lambda \mu(B(y,r_y)),
\end{equation}
so $\{B(y,r_y)\}_{y \in E_\lambda}$ clearly covers $E_\lambda$.  A
standard $5$-covering theorem \cite[Theorem 2.1]{Mattila1999} (or alternatively \cite[Theorems 2.8.4--2.8.6]{federer}) 
then asserts that there is a countable, 
pairwise-disjoint subcollection of balls $B_i := B(y_i,r_{y_i})$ for $i \in I$ with each $y_i \in E_\lambda$ and so that
$$
\bigcup_{y\in E_\lambda} B(y,r_y) \subset \bigcup_{i \in I} B(y_i,5r_{y_i}).
$$
Using the fact that
$
\bigcup_{i \in I} B_i \subset B(x,r+R)
$
we then obtain
$$
\mu(E_\lambda) \leq
\sum_i \mu(B(y_i,5r_{y_i})) \leq
D^3\sum_i \mu(B_i) \leq
\frac{D^3}{\lambda} \sum_i \int_{B_i} |f| ~d\mu \leq
\frac{D^3}{\lambda} \int_{B(x,r+R)} |f| ~d\mu
$$
as desired.
\end{proof}

\subsection{Poincar\'e inequalities via fillings} \label{subsec:PIfilling}

In this subsection, we 
make precise 
the notion of filling and ``fillable set,'' 
the main tools in proving our results. One useful property of fillings $\Omega_r$ is that they satisfy a Poincar\'e inequality \emph{a priori} only at scales comparable to $r$. For our applications, this property will be easy to check, in that the geometry of the filling at scale $r$ will be kept simple.

\begin{definition}\label{def:PIfilling}
Let $\epsilon \in (0,1)$, $p \in [1,\infty)$, and $C,D
\geq 1$.
Given a closed subset $Y$ of a complete space 
$X$,
a closed subset $\Omega_r \subseteq X$ is called an {\sc $\epsilon$-filling of $Y$} at scale $r>0$ with constants $(D,C,p)$  
if 
the following conditions hold:%
\begin{enumerate}
\item 
$Y \subset \Omega_r$,
\item for every $x \in Y$, the density condition
$\displaystyle
\frac{ \mu( \Omega_r \cap B(x,r) \setminus Y ) }{ \mu( \Omega_r \cap B(x,r) ) } < \epsilon
$
holds,
\item the restricted space $(\Omega_r,d|_{\Omega_r \times \Omega_r}, \mu|_{\Omega_r})$ is 
$D$-doubling 
and satisfies a $(1,p)$-Poincar\'e inequality
at scale $2r$,
$$
\avint_B |f-f_B| d\mu \leq
C s \left(
\avint_{C B} \Lip(f)^p d\mu
\right)^{1/p}
$$
where $B=B_{\Omega_r}(x,s)$ is any ball in $\Omega_r$ with $s \leq 2r$.
\end{enumerate}
Then, $Y$ is called {\sc $p$-Poincar\'e $\epsilon$-fillable up to scale $r_0$}, 
with constants $(D,C)$ --- or {\sc $(\epsilon,D,C,p)$-PI fillable up to 
scale $r_0$}, for short --- if 
there exists an $\epsilon$-filling at scale $r$ of $Y$ with constants $(D,C,p)$ and any $r\in (0,r_0)$.

We say that $Y$ is {\sc asymptotically $p$-Poincar\'e fillable} if for some fixed constants $(D,C)$ and for any $\epsilon>0$ there exists $r_0>0$ such that $Y$ is $(\epsilon,D,C,p)$-PI fillable up to scale $r_0$.
\end{definition}

In terms of these sets, we can now give sufficient conditions for a subset to satisfy a Poincar\'e inequality.

\begin{theorem}\label{thm:PIthm}
Fix 
structural constants $(p,D,C,r_0)$ 
and let $X$ be a $D$-doubling
metric measure space.
Then, for every $q>p$ there exist $\epsilon_q, 
C_q,C_r>0$
with the following properties:
\begin{itemize}
\item[(a)] \label{item:gap}
If $Y$ is a
$p$-Poincar\'e, $\epsilon_q$-fillable subset of $X$ up to scale $r_0$ with constants $(D,C)$, 
then it satisfies a
$(1,q)$-Poincar\'e inequality with constant $C_q$ at scale $r_0/C_r$. 
\item[(b)] \label{item:nogap}
Further, if $Y$ is an asymptotically $p$-Poincar\'e fillable subset of $X$, then it satisfies a local $(1,q)$-Poincar\'e inequality 
for every $q > p$.
\end{itemize}
Here the constants $\epsilon_q$ and $
C_q, C'
$ are independent of the original scale $r_0$, but depend on the other structural constants and on the exponent $q$.
\end{theorem}

\begin{remark}
Note that $X$ is not assumed, a priori, to support a Poincar\'e inequality; only the fillings $\Omega_r$ from Definition \ref{def:PIfilling} do.  In many cases, including our applications in Section \S\ref{sec:almostunif}, we will assume that $X$ is a $p$-PI space, in which case good choices of $\Omega_r$ will inherit Poincar\'e inequalities from $X$.
\end{remark}

Note that the local Poincar\'e inequality could be improved to a semi-local one \cite{bjornlocal} (that is, \eqref{eq:defPI} holds
at every scale, with constant depending on the scale and location only), if the space is proper and connected.
In the case of bounded metric spaces, like non-self-similar Sierpi\'nski carpets, this semi-local property further improves to the usual global type.

\begin{remark}
It is crucial in Part (a) of the previous theorem that the density parameter $\epsilon_q$ be allowed to depend on the structural constants $D,C,p$.  
\end{remark}

Here we give some examples involving fillings of subsets and how the exponent of the Poincar\'e inequality can depend subtly on how the set is filled. In each case we construct a filling with arbitrarily good Poincar\'e inequalities, namely local $(1,1)$-Poincar\'e inequalities. The subset, however, only inherits the Poincar\'e inequality if the density parameter is sufficiently small, relative to a controlled constant in the Poincar\'e inequality of the filling.


\begin{example} \label{ex_cornercutpt}
Let $X = [-1,1]^2$, which is a $(1,1)$-PI-space, while the subset
$$
Y = [-1,0] \times [0,1] \cup [0,1] \times [-1,0] 
$$
is a $(1,p)$-PI-space only for $p>2$. However, if we ``thicken'' $Y$ at the origin, then the filling 
$$\overline{Y}^h_r=Y \cup \overline{B(0,hr)},$$
satisfies a $(1,q)$-Poincar\'e inequality at scale $r$ with constant $C_q^h$, where
$$
C_q^h 
~\approx_q~
\begin{cases}
h^{\frac{q-2}{q}}, & \text{if } 1 \leq q < 2,
\\
\log(1/h), & \text{if } q=2. 
\end{cases}
$$
and where $C_q^h$ can be bounded independent of $h$ for $q > 2$.
Here, the ratio implied in $\approx_q$ depends on $q$, but not on $h$, and could be made explicit. 

For every $r>0$, we can set $\Omega_r=\overline{Y}^h_r$, 
and see that $Y$ is $q$-Poincar\'e $h^2$-fillable up to scale $1$ with constants $(D,C_q^h)$, for some uniform doubling constant $D$. 
By Theorem \ref{thm:PIthm} then $Y$ satisfies a $(1,q)$-Poincar\'e inequality for $q>2$, as expected. However, for $q \in [1,2]$, the Poincar\'e constant $C_q^h$ blows up as $h \to 0$, so the subset $Y$ need not, and does not, satisfy a $(1,q)$-Poincar\'e inequality for $q \in [1,2]$.
\end{example}




The following example is closely related to the discussion of fat Sierpi\'nski carpets and sponges in Section \S\ref{sec:sierpinskisponge}.

\begin{example}
 Let $X = [0,1]^2$, and let $C_{1/3}$ be the usual ``middle thirds'' Cantor set in $[0,1]$ and denote by $\mathcal{I}_k$ the open removed intervals of length $3^{-k}$ in the construction of $C_{1/3}$. Now define the set of squares
\[
\mR = \left\{ \left.I \times \left(\frac{1-3^{-k}}{2}, \frac{1+3^{-k}}{2} \right) \right| I \in \mathcal{I}_k, k \in \N\right\}
\]
and denote
the complement of their union as
\[
Y = [0,1]^2 \setminus \bigcup_{R \in \mR} R.
\]
Unlike the standard ``middle-ninths''
Sierpi\'nski carpet, 
only the squares intersecting the line $y = \frac{1}{2}$ are removed. (See Figure \ref{fig:exampcarp}.)

Putting $\alpha = \frac{\log(2)}{\log(3)}$ for the Hausdorff dimension of $C_{1/3}$, we now claim that $Y$ with the restricted Lebesgue measure and Euclidean distance satisfies a $(1,p)$-Poincar\'e inequality if and only if $p>2-\alpha$. To see why, both of the sets
\begin{align*}
Y_+ = Y \cap [0,1] \times \Big[0,\frac{1}{2}\Big]
\text{ and }
Y_- = Y \cap [0,1] \times \Big[\frac{1}{2},1\Big]
\end{align*}
are
uniform domains (see Definition \ref{def:unifdom}) and therefore satisfy $(1,1)$-Poincar\'e inequalities (see Theorem \ref{thm:unifcon}).  Moreover, we have
$$
Y = Y_+ \cup Y_-  \text{ and } Y_+ \cap Y_- = C_{1/3} \times \{\frac{1}{2}\},
$$ 
so $Y$ arises from gluing $Y_\pm$ along a $\alpha$-dimensional set and by \cite[Theorem 6.15]{heinonen1998quasiconformal}, it satisfies a $(1,p)$-Poincar\'e inequality for $p>2-\alpha$. On the other hand, $Y$ does not satisfy a Poincar\'e inequality for $p \in [1,2-\alpha]$; indeed, consider the function
 \[
u(x,y)= 
\max\big\{\min\Big\{\frac{1}{h}\Big(y-\frac{1}{2}\Big),1\Big\},0\big\}
\]
On $[0,1] \times (1/2,1/2+h]$ we have $\Lip(u) = \frac{1}{h}$, 
so if $q<2-\alpha$, then for all $h<\frac{1}{3}$ we have
$$
\vint_{[0,1]^2} \left|u-u_{[0,1]^2} \right| ~d\lambda ~\geq~ \frac{1}{6} ~\geq~
h^\frac{2-\alpha-q}{q}
\approx_q 
\left(\vint_{[0,1]^2} \Lip(u)^q ~d\lambda \right)^{1/q} $$
which
contradicts the $(1,q)$-Poincar\'e inequality as $h \to 0$. The case $q=2-\alpha$ is similar, but we consider the function
  \[
u(x,y)=
\begin{cases}
1, & \text{if }  y \leq \frac{1}{2},
\\
\min\left\{\max\left\{\log\left(\frac{h}{y-\frac{1}{2}}\right),0\right\},1\right\}, & \text{if } y> \frac{1}{2}.
\end{cases}
\]
Again $Y$ has certain good fillings that consist of
\[
\Omega_r = [0,1]^2 \setminus \bigcup_{R \in \mR, \diam(R) \geq r/9} R.
\]
At scale $r$, only finitely many sets $R$ with diameters larger than $r/9$ are near points in $\Omega_r$.  It follows that $\Omega_r$ satisfies $(1,1)$-Poincar\'e inequalities at scales comparable to $r$ with constants $(D,C)$ independent of $r$. 
 
However, for balls centered on 
$y=1/2$ the density of $\Omega_r \setminus Y$ 
is bounded from below, say by some constant $\delta>0$. Thus, these are only $(\delta,D,C,1)$-PI-fillable 
and not asymptotically 1-Poincar\'e fillable. This corresponds to the fact that we obtain only a $(1,p)$-Poincar\'e inequality for $p>2-\lambda$, instead of for all $p>1$.

%

 \begin{figure}[h!]
  \centering
    \includegraphics[width=0.4\textwidth]{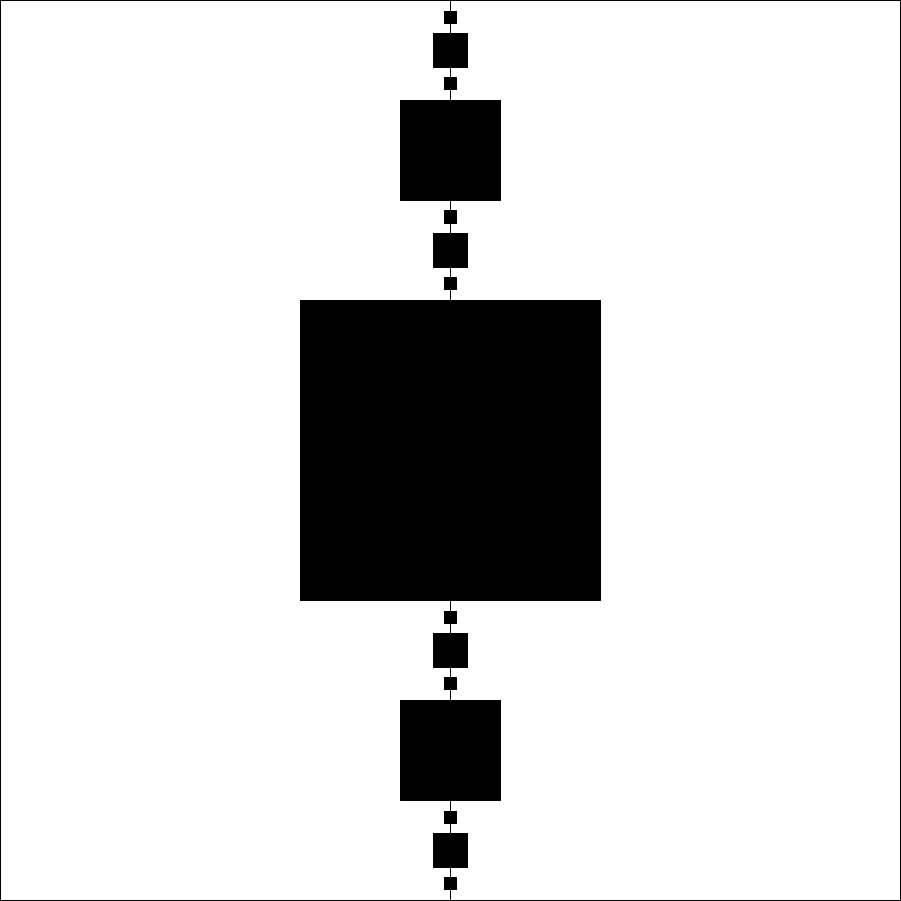}
    \caption{An approximant of the space $Y$ with the squares removed at the first three levels. The image is rotated by 90 degrees.}
    \label{fig:exampcarp}
\end{figure}

\end{example}


\subsection{Poincar\'e inequalities via ``maximal'' connectivity}

The proof of Theorem \ref{thm:PIthm}
is based on general techniques that reduce the 
Poincar\'e inequality to
a certain connectivity property at all scales and with sets (or `obstacles') of
prescribed densities.
 These densities are in turn measured in terms of maximal functions.

The starting point is this very notion of connectivity:
roughly speaking,
``if a set $E$ has small measure 
density
(in a scale invariant way) then there are 
curves of unit speed that spend only a short time within $E$.%
''

\begin{definition}\label{def:finemax}
Let $\delta>0$ and $C,p \geq 1$. We say that a pair of points $x,y \in X$ for a metric measure space $(X,d,\mu)$ is {\sc $(C,\delta,p)$-max connected}, if for every $\tau>0$ with $r := d(x,y)$, and every Borel-measurable set $E$ such that
\begin{align} \label{eq:finemaxdensity}
M_{Cr}(1_E)(x) < \tau
\text{ and }
M_{Cr}(1_E)(y) < \tau
\end{align}
there exists a 1-Lipschitz curve $\gamma \co [0,L]\to X$, for some $L>0$, such that
\begin{enumerate}
\item $\gamma(0)=x$
\item $\gamma(L)=y$
\item $\len(\gamma) \leq Cr$
\item the following integral inequality holds:

\begin{equation}\label{eq:finemax1}
\int_\gamma 1_E ~ds \leq 
\delta\tau^\frac{1}{p}r.
\end{equation}
\end{enumerate}
We say that a space $(X,d,\mu)$ is {\sc $p$-maximally connected at scale $r_0$} with constants $(C,\delta)$ ---
or {\sc $(C,\delta,p)$-max connected at scale $r_0$}, for short --- if every pair $x,y\in X$ with $d(x,y) < r_0$ is $(C,\delta,p)$-max connected.
\end{definition}




\begin{remark}\label{rmk:reduction}
Since the measure is assumed to be Borel regular, it is enough to verify Definition \ref{def:finemax} for all open (or all closed) ``obstacles'' $E$. Indeed, if $\epsilon>0$ and $E\subset X$ is any Borel set, we can find using Borel regularity an open set $E'$ so that $E\subset E'$ and $M_{Cr}(1_{E'\setminus E})(x),M_{Cr}(1_{E'\setminus E})(y)<\epsilon$. The case of closed sets is only slightly harder, and, as we don't use it anywhere, we only sketch the details. One can for each open set $E$ exhaust it with closed sets $E_k = \{x : d(x,X \setminus E) \geq \frac{1}{k}\}$. One then finds a sequence of curves $\gamma_k$ for each $E_k$, and since $E_{k-1} \subset {\rm int}(E_k)$, then after passing to a subsequence and using monotone convergence, we can find a curve $\gamma$ which satisfies (1)--(4) for $E$.
\end{remark}


A technical issue with 
checking for maximal connectivity 
is that the desired maximal function estimates for $X$ are not directly related to those for the filling $\Omega_r$. Furthermore, it can be challenging to prove the property for all density ``levels'' $\tau>0$. This is dealt with the following variants of this connectivity.

\begin{definition} \label{def:maxconnlevel}
We say that a metric measure space $(X,d,\mu)$ is {\sc $p$-maximally connected at level $\tau_0$ and scale $r_0$} (with constants $(C,\delta)$) --- or {\sc $(C,\delta,\tau_0,p)$-max connected at scale $r_0$},
for short ---
if the $p$-maximal connectivity conditions of Definition \ref{def:finemax} hold for only $\tau=\tau_0$, instead of for all $\tau$. 
\end{definition}

This condition may seem technical at first.  The core point, however, is that it allows for characterizing Poincar\'e inequalities in terms of sufficiently good avoidance of 
obstacles
of a fixed level, 
so
one need not consider obstacles of every 
level.
Further, this  ``fixed-level'' property is inherited by  sufficiently dense subsets. 

\begin{lemma} \label{lem:prototype}
Suppose $X$ is $D$-doubling and $(C,\delta,\tau_0,p)$-max connected at scale $r_0$ and that $Y$ is a closed, $\Lambda$-quasiconvex subset of $X$. If $x,y \in Y$ satisfy $d(x,y)<r_0$, as well as 
$$
M_{Cr} 1_{X \setminus Y} (x) < \frac{\tau_0}{2}
\text{ and }
M_{Cr} 1_{X \setminus Y} (y) < \frac{\tau_0}{2},
$$
then the pair $(x,y)$ is 
$(\Lambda C,\Lambda \delta,\frac{\tau_0}{2},p)$-max connected relative to $Y$ with its restricted measure and distance.
\end{lemma}

We will only sketch the main form of the argument, since the lemma will not be used directly and a variant appears later. The main idea, however, is replacing bad portions of an initial curve with better ones, as depicted in Figure \ref{fig:replacement}.

\begin{proof} By Remark \ref{rmk:reduction}, it suffices to consider open sets. Let $E \subset Y$ be a relatively open arbitrary open set with $M^Y_{Cr} 1_{E} (z)<\tau_0/2$ for $z=x,y$ but where the maximal function is computed relative to $Y$; for $F=E \cup (X \setminus Y)$ it then follows that $M_{Cr} 1_{F} (z)<\tau_0$, where the maximal function is once again relative to $X$. 

Thus the definition of max-connectivity gives a curve $\gamma$ that spends at most $\delta \tau_0^{1/p} r$ in the complement of $Y$ and the set $E$. The set $\gamma^{-1}(X \setminus Y)$ consists of countably many disjoint maximal open intervals $(a_i,b_i)$, so we can replace each $\gamma|_{(a_i,b_i)}$ by a $\Lambda$-quasigeodesic in $Y$ that joins $\gamma(a_i)$ and $\gamma(b_i)$. This produces a new curve $\gamma'$ which lies entirely in $Y$, is at most $\Lambda C d(x,y)$ long, and spends at most $\Lambda \delta \tau_0^{1/p} r$ time in $E$, as desired.
\end{proof}

 \begin{figure}[h!]
  \centering
    \includegraphics[width=0.40\textwidth]{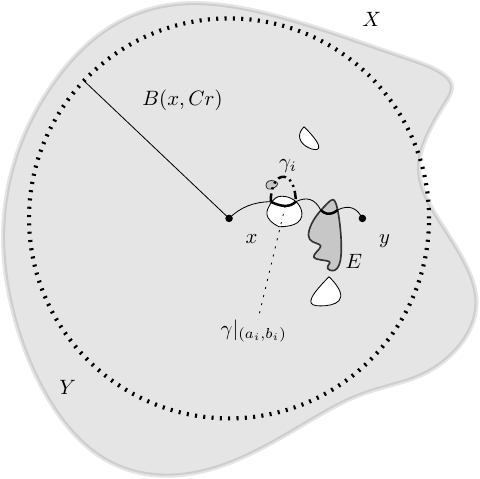}
    \caption{\small Proof of Lemma \ref{lem:prototype}. Connectivity involves constructing a curve in the gray subset $Y$ between a pair of points $x,y$ while avoiding the dark gray subset $E$ as well as possible. The connectivity of $X$ is used to give a ``proto-curve'', whose portions $\gamma|_{(a_i,b_i)}$ in the complement $X \setminus Y$ are replaced by detours $\gamma_i$ constructed using quasiconvexity (the dash-dotted line segment). }
    \label{fig:replacement}
\end{figure}

Our connectivity property is related to the $(1,p)$-Poincar\'e inequality via the following two theorems. 
We discuss their applications first in the next section, and their proofs will appear later in 
Section \S\ref{sect:classification}.

\begin{theorem}\label{thm:classification1} Fix 
structural constants $(p,D,C,r_0)$. 
If $X$ is $D$-doubling at scale $r_0$ and satisfies a $(1,p)$-Poincar\'e inequality at scale $r_0$ with constant $C$, then $X$ is $(C_0,
\Delta,p)$-max connected at scale $r_0/2$, 
where 
$C_0$ 
and $\Delta$ depend
solely on the structural constants.
\end{theorem}

The converse also holds true, but requires a \emph{sufficiently small} value for $\delta$.

\begin{theorem}\label{thm:classification2} 
Fix structural constants $(p,D,C,r_0)$.
There exists $\delta_{p,D} > 0$ such that if $X$ is $D$-doubling at scale $r_0$ and $(C,\delta,\tau_0,p)$-max connected at scale $r_0$ for some $\tau_0 \in (0,1)$ and some $\delta \in (0, \delta_{p,D})$, then it also satisfies a $(1,p)$-Poincar\'e inequality at scale $r_0/C_r$ with constant $C_p$, where $C_p,C_r$ 
are independent of scale $r_0$ but depends quantitatively on 
all the other structural constants, as well as $\delta$ and
$\tau_0$.
\end{theorem}

As emphasized in the notation, the above constant $\delta_{p,D}$ depends only on $p$ and $D$ and no other structural constants.

However, a small parameter value for $\delta$ is not serious; the next 
result 
assures that such values for $\delta$ always occur at some density level $\tau$ 
but
for slightly larger exponents than $p$.

\begin{lemma}\label{lem:deltasmall}
With the same constants as in Theorem \ref{thm:classification1},
let $X$ be a $D$-doubling metric measure space that is $(C,\Delta,p)$-max connected at scale $r$, and let 
$q>p$.
For each $\delta \in (0,1)$, there exists 
$\tau_0=\tau_0(q,\delta) \in (0,1)$ 
so that $X$ is $(C,\delta,\tau',q)$-max connected at scale $r$ for any $\tau'\in (0,\tau_0)$.
\end{lemma}

\begin{proof} Choose $\tau_0(q,\delta) = \min\{1,\left(\frac{\delta}{\Delta} \right)^{\frac{pq}{q-p}}\}$ and $\tau'\in(0,\tau_0(q,\delta))$. We will show $(C,\delta,\tau',q)$-max connectivity Let $x,y,E$ be as in the Definitions 
\ref{def:finemax} and \ref{def:maxconnlevel} at scale $r$, i.e. $d(x,y) < r$ and
 \begin{eqnarray*}
M_{Cd(x,y)}(1_E)(x) &<& \tau'
\\
M_{Cd(x,y)}(1_E)(y) &<& \tau'.
\end{eqnarray*}
By $(C,\Delta,p)$-max connectivity, there is a curve $\gamma$ connecting $x$ to $y$ with length at most $Cd(x,y)$ and with
\[
\int_\gamma 1_E \,ds \leq \Delta \tau'^{\frac{1}{p}}d(x,y).
\]
By our choice of $\tau_0$, we have 
$
\Delta \tau'^{1/p} = 
(\Delta \tau'^{1/p-1/q}) \tau'^{1/q} \leq 
\delta \tau'^{1/q}
$, and thus we also have
\[
\int_\gamma 1_E \,ds \leq \delta \tau'^{\frac{1}{q}}d(x,y),
\]
and in particular, $\gamma$ already verifies the $(C,\delta,\tau',q)$-max connectivity condition.
\end{proof}

To reiterate, to prove that a $p$-fillable subset $Y$ satisfies a $(1,p)$-Poincar\'e inequality,
by Theorem \ref{thm:classification2} it is sufficient to prove the maximal connectivity property for $Y$ at a certain level and for fixed choices $C$ and 
{
$\delta<\delta_{p,D}$.%
}
Similarly as in Lemma \ref{lem:prototype}, this 
property will
be `inherited' from a filling $\Omega_r$ at a comparable scale. 

With these general statements at hand,
we will employ the following strategy for the proof of Theorem \ref{thm:PIthm}:
\begin{enumerate}
\item 
Theorem \ref{thm:classification1} 
guarantees that any filling $\Omega_r$ of $Y$ will satisfy
maximal connectivity properties with exponent $p$ and \emph{some} initial parameter $\Delta$. 
\vspace{.05in}
\item 
From Lemma \ref{lem:deltasmall} we obtain $(C,\delta,\tau_0,q)$-maximal connectivity for $\Omega_r$ at scale $r$ for \emph{arbitrarily small} parameters $\delta$, but at the expense of a slightly larger exponent $q$.
\vspace{.05in}
\item 
Similarly to Lemma \ref{lem:prototype}, due to quasi-convexity 
(see Lemma \ref{lemma:quasiconvex} below)
$Y$ inherits the maximal connectivity property from its filling $\Omega_r$, but with 
$\delta'$ slightly larger than $\delta$.
This parameter $\delta'$ can be ensured to be less than 
the given threshold $\delta_{p,D}$,
however, by an initially small choice of $\delta$ in the previous step. 
\vspace{.05in}
\item 
Using maximal connectivity and quasiconvexity (again), we show $Y$ satisfies a $(1,q)$-Poincar\'e inequality via Theorem \ref{thm:classification2}. 
\end{enumerate}
Here $q>p$ is needed 
to apply the argument from Lemma \ref{lem:deltasmall}. If $p>1$, this could be avoided via Keith-Zhong \cite{keith2008poincare}, since we could first improve the Poincar\'e inequality for each $\Omega_r$  to an exponent $p'<p$.

\section{Proof that ``fillable'' sets satisfy Poincar\'e inequalities} \label{sec:proof}

\subsection{Initial geometric considerations}

Now, we show that the underlying (restricted) measure of a fillable subset is well-behaved. More precisely, we show that a fillable subset $Y$ inherits the doubling property from its fillings $\Omega_r$. Recall 
that throughout this paper,
$Y \subset \Omega_r \subset X$, where $\Omega_r$ will be the relevant fillings.

\begin{lemma}\label{lemma:doubling}
Fix structural
constants $(p,D,C,
r_0)$. If $Y$ is $(\epsilon,D,C,p)$-PI fillable
up to scale $r_0$
for some $\epsilon \in (0,1)$,
then $Y$ 
is $(\frac{D}{1-\epsilon},r_0)$-doubling.
\end{lemma}



\begin{proof}
Let $r \in (0, 
r_0)$ and $x \in Y$. 
From item (2) of Definition \ref{def:PIfilling}, we have
\begin{eqnarray}
\notag
\mu\big( \Omega_r \cap B(x, r
) \big)
&=&
\mu\big( \Omega_r \cap B(x,{r}
) \cap Y \big) +
\mu\big( \Omega_r \cap B(x,{r}
) \setminus Y \big)
\\ &<& 
\notag
\mu\big( Y \cap B(x,{r}
) \big) +
\epsilon \, \mu\big( \Omega_r \cap B(x,{ r}
) \big)
\\
\notag 
\therefore (1-\epsilon) \, \mu\big( \Omega_r \cap B(x,{ r}
) \big) &<&
\mu\big( Y \cap B(x,{ r}
) \big).
\end{eqnarray}
Since $\Omega_r$ is
assumed
$D$-doubling
with respect to the restricted measure $\mu|_{\Omega_r}$ and since
$Y$ is a subset of $\Omega_r$, 
it follows that
$$
\mu\big( Y \cap B(x,2 {r}
) \big) \leq
\mu\big( \Omega_r \cap B(x,2{ r}
) \big) \leq
D \, \mu\big( \Omega_r \cap B(x,{r}
) \big) \leq 
\frac{D}{1-\epsilon} \mu\big( Y \cap B(x,{ r}
) \big).
$$
So the claim follows with doubling constant  $\frac{D}{1-\epsilon}$.
\end{proof}

We next show that PI-fillable subsets $Y$ 
are quasiconvex. 
This connectivity property is derived from stronger ones, i.e. the Poincar\'e inequalities of the fillings $\Omega_r$.
For clarity later, given $f \in L^1(X)$ and $R>0$ we specify the choice of metric space for maximal functions by using the shorthand
\begin{eqnarray*}
M^r_Rf(x) &:=&
\sup_{\rho \in (0,R)} \avint_{B(x,\rho) \cap \Omega_r} |f| ~d\mu
\\
M^0_Rf(x) &:=&
\sup_{\rho \in (0,R)} \avint_{B(x,\rho) \cap Y} |f| ~d\mu,
\end{eqnarray*}
where $\Omega_r$ is as in Definition \ref{def:finemax}.

\begin{lemma} \label{lemma:quasiconvex}
Fix 
structural constants 
$(p,D,C,r_0)$.
There exist $\epsilon_0, \Lambda, r_1> 0$,
depending solely on the structural constants, 
so that if $Y$ is a $(\epsilon,D,C,p)$-PI fillable subset of a metric space $X$
at scale $r_0$,
for some $\epsilon \in (0,\epsilon_0)$, then it is
$\Lambda$-quasiconvex at scale $r_1$. 
\end{lemma}

\begin{proof} By hypothesis, $Y$ is $(\epsilon,D,C,p)$-fillable up to scale $r_0$, for some $\epsilon \in (0,\epsilon_0)$, so there exist fillings $\Omega_r$ for every $r \in (0,r_0)$ with $Y \subset \Omega_r \subset X$  that  are $D$-doubling at scale $2r$, that support a $(1,p)$-Poincar\'e inequality at scale $r$ with constant $C$, and so that
$$
\frac{ \mu(\Omega_r \cap B(z,r) \setminus Y) }{ \mu(\Omega_r \cap B(z,r)) } < \epsilon < \epsilon_0
$$
holds for all $z \in Y$. From Theorem \ref{thm:classification1} we conclude that the fillings $\Omega_r$ are 
$(C_0,\Delta,p)$-max connected for some $C_0$ and $\Delta$ at scale $r/2$. 
Choose $\tau_0 = \frac{1}{\Delta^p 4^p}$ 
so that 
$\Delta\tau_0^{1/p}r \leq r/4$
and fix
$$
\epsilon_0 ~=~ D^{-( 10+\lceil\log_2(C_0))\rceil)}\tau_0
\text{ and }
\Lambda = 2C_0
\text{ and }
r_1 = \frac{r_0}{2^5 C_0}
.
$$
Since $C_0$ and $\Delta$ depend only on the structural constants, by Theorem \ref{thm:classification1}, the same is true of $\epsilon_0$, $\Lambda$, and $r_1$.

We now show that $Y$ is $\Lambda$-quasiconvex at scale $r_1$.
For every $x,y \in Y$ with $r=d(x,y) < r_1$. 
we will construct a $\Lambda$-quasi-geodesic joining
$x$ and $y$, 
using a recursive argument.


\vspace{.05in}

\textit{Base case(s)}.\ Fix $R =  2^5 C_0 r$. The initial curve will be constructed in $\Omega_R$ 
and will lie almost entirely in $Y$.
To begin, define
an obstacle
$$
E ~:=~
X \setminus (Y  \cup \overline{B(x,r/16)} \cup \overline{B(y,r/16)}).
$$
In particular, this implies for $\rho \in (\frac{1}{16}r, R)$ that
$$
\frac{ \mu(\Omega_{ R
} \cap B(x,\rho) \cap E) }{ \mu(\Omega_{ R
} \cap B(x,\rho)) } ~\leq~
\frac{ \mu(\Omega_{ R
}\cap B(x,{ R}
) \setminus Y) }{ \mu\big(\Omega_{ R} \cap B(x,\frac{1}{16}r)\big) } ~\leq~
D^{10+\lceil\log_2(C_0)\rceil} \frac{ \mu(\Omega_{ R} \cap B(x,{ R}) \setminus Y) }{ \mu\big(\Omega_{ R} \cap B(x,{ R})\big) }
$$
and since
$\mu(\Omega_{R} \cap B(x,\rho) \cap E) = 0$
holds whenever $\rho \in (0,\frac{1}{16}r)$, it follows that
\begin{equation} \label{eq:maxconnfillable}
M_{C_0r}^R 1_E(z) < D^{10+\lceil\log_2(C_0)\rceil} \epsilon < \tau_0
\text{ for } z = x,y.
\end{equation}
For future consistency of notation, put $x_{1,1} := x$ and $y_{1,1} := y$ and $x_{i,1} := y_{i,1} := y$ for $i \geq 2$. 
Also define $r_{i,1}=d(x_{i,1},y_{i,1})$, in which case
$$
\sum_i r_{i,1} = r_{1,1} ~\leq~ r.
$$
Recall that $\Omega_{R}$ is $(C_0, \Delta,p)$-max connected at scale $R/2>r$.  By applying Definition \ref{def:finemax} to $E$, Equation \eqref{eq:maxconnfillable} guarantees the existence of a $C_0$-quasi-geodesic $\gamma_1: [0,L_1] \to \Omega_R$, for some length\footnote{
Recall the convention that all rectifiable curves are assumed to be parametrized with respect to arc-length, unless otherwise specified. The only time below that we will need this we will indicate such curves by an asterisk.
}
$L_1 \in (0,C_0r)$, joining $x$ and $y$ in $\Omega_{R} \subset X$, and so that
\begin{equation} \label{eq:obstaclelength1}
\int_{\gamma_1} 1_{E} ~ds \leq \Delta\tau_0^{1/p}r \leq r/4.
\end{equation}
Consider the exit times
\begin{eqnarray*}
t_{1,1} &:=& \sup\big\{t \in [0,L_1] ~|~ d(\gamma_1(t), x) \leq r /8 \big\}
\\
T_{1,1} &:=& \inf\big\{t \in [0,L_1] ~|~ d(\gamma_1(t), y) \leq r /8 \big\}.
\end{eqnarray*}
Since $Y$ is closed, the set $E$ is open and it follows that $\gamma_1^{-1}(E) \cup (0,t_{1,1}) \cup (T_{1,1}, L_1)$ is open in $\R$, so it is a countable union of open intervals
$$
(0,t_{1,1}) \cup (T_{1,1}, L_1) \cup \gamma_1^{-1}(E) ~=~
\bigcup_{i=1}^\infty (a^i,b^i)
$$
with $a^i \leq b^i$ and where each pair $x_{i,2} := \gamma_1(a^i)$ and $y_{i,2} := \gamma_1(b^i)$, of distance
$r_{i,2} := d(x_{i,2},y_{i,2})$
apart, also lie in $Y$.  (If the union is finite, then there exists $N \in \N$ so that $a^n=b^n$ for $n \geq N$.) Also,
\begin{equation} \label{eq:obstacleavoid}
\gamma_1^{-1}(X \setminus Y) \subset \bigcup_{i=1}^\infty (a^i,b^i).
\end{equation}
Equation \eqref{eq:obstaclelength1} thus implies that
\begin{eqnarray*}
\sum_i r_{i,2} &\leq&
{\rm len}(\gamma_1 \cap \Omega_r \setminus Y) + \frac{r}{4} ~\leq~
\Big(\Delta\tau_0^{1/p} + \frac{1}{4}\Big)r ~\leq~
\frac{r}{2}.
\end{eqnarray*}
Since $\gamma_1$ is parametrized by length, and $\len(\gamma_1) \leq \Lambda r=\Lambda r_{1,1}$, it trivially holds that 
\begin{equation} 
\gamma_{1} \setminus Y \subset \bigcup_{i=1}^\infty B(x_{i,1}, \Lambda r_{i,1}).
\end{equation}

\textit{Recursive step}.\
Let $k \in \N$ be given, with $k \geq 2$, and suppose the sequence $\big( (x_{j,k},y_{j,k}) \big)_{j=1}^\infty$ in $Y \times Y$ has already been defined, with $r_{j,k} := d(x_{j,k},y_{j,k}) < r_0$ and with the property that
\begin{equation} \label{eq:gaplengthk}
\sum_j r_{j,k} \,\leq\, 2^{1-k}r.
\end{equation}
Assume further that a $C_k$-quasi-geodesic $\gamma_{k-1}: [0,L_{k-1}] \to X$ joining $x$ and $y$ has already been defined for some $L_{k-1} \in (0,C_{k-1}r)$, where
\begin{equation} \label{eq:curvelengthold}
C_{k-1} ~:=~ 2(1-2^{-(k-1)})C_0
\end{equation}
and with the property that there exist $\{ a_{k-1}^j, b_{k-1}^j \}_{j=1}^\infty \subset [0,L_{k-1}]$
with
$$
x_{j,k} ~=~ \gamma_{k-1}(a_{k-1}^j)
\text{ and }
y_{j,k} ~=~ \gamma_{k-1}(b_{k-1}^j)
\text{ and } r_{j,k}=d(x_{j,k},y_{j,k})
$$
and which satisfies the avoidance properties
\begin{align} \label{eq:obstacleavoidk}
\gamma_{k-1}^{-1}(X \setminus Y) ~\subset&~ \bigcup_{j=1}^\infty (a^{j}_{k-1},b^{j}_{k-1}),
\\
\label{eq:obstacleavoidkbetter}
\gamma_{k-1} \setminus Y ~\subset&~ \bigcup_{j=1}^\infty B(x_{j,k-1},\Lambda r_{j,k-1}).
\end{align}
By applying the same argument as in the base case, with $x_{j,k}$ and $y_{j,k}$ and $r_{j,k}$ in place of $x$ and $y$ and $r$, take fillings $\Omega_{j,k} := \Omega_{2^5C_0r_{j,k}}$ of $Y$ that are $(C_0,\Delta,p)$-max connected at scales $2^4C_0r_{j,k}$.  Using obstacles
$$
E_{j,k} ~:=~
X \setminus (Y \cup \overline{B(x_{j,k},r_{j,k}/16)} \cup \overline{B(y_{j,k},r_{j,k}/16)}).
$$
and estimating similarly as \eqref{eq:maxconnfillable}, there exist
$C_0$-quasigeodesics $\gamma_{j,k}: [0,L_{j,k}] \to \Omega_{j,k} \subset X$ joining $x_{j,k}$ to $y_{j,k}$ in $\Omega_{j,k}$, so that
\begin{equation} \label{eq:obstaclelengthk}
\sum_{j=1}^\infty \int_{\gamma_{j,k}} 1_{E_{j,k}} ~ds ~\leq~
\sum_{j=1}^\infty \Delta\tau_0^{1/p} r_{j,k} ~\stackrel{\eqref{eq:gaplengthk}}{\leq}~
2^{-1-k}r
\end{equation}
and whose lengths $L_{j,k} \leq C_0 r_{j,k}$ satisfy
\begin{equation} \label{eq:gaplengthk+1}
\mathcal{H}^1\Big( \bigcup_j \gamma_{j,k} \Big) ~\leq~ \sum_{j} L_{j,k} 
~\leq~
\sum_j C_0 r_{j,k} ~\stackrel{\eqref{eq:gaplengthk}}{\leq}~
2^{1-k}C_0r.
\end{equation}
As before, for each $j \in \N$ set exit times
\begin{eqnarray*}
t_{j,k} &:=& \sup\big\{t \in [0,L_{j,k}] ~|~ d(\gamma_{j,k}(t), x_{j,k}) \leq r_{j,k}/8 \big\}
\\
T_{j,k} &:=& \inf\big\{t \in [0,L_{j,k}] ~|~ d(\gamma_{j,k}(t), y_{j,k}) \leq r_{j,k}/8 \big\}.
\end{eqnarray*}
The preimage $\gamma_{j,k}^{-1}(X \setminus Y)$ is open in $\R$ and satisfies
$$
(0,t_{j,k}) \cup
(T_{j,k},L_{j,k}) \cup
\gamma_{j,k}^{-1}(X \setminus Y) ~=~
\bigcup_{l=1}^\infty (a_{j,k}^{l*},b_{j,k}^{l*})
$$
for sequences of pairs $a_{j,k}^{l*} \leq b_{j,k}^{l*}$.  Reindexing $i = i(j,l)$ as needed, put
\begin{eqnarray*}
x_{i,k+1} &:=& \gamma_{j,k}(a_k^{i*}), \text{ where } a_k^{i*} ~:=~ a_{j,k}^{l*}
\\
y_{i,k+1} &:=& \gamma_{j,k}(b_k^{i*}), \text{ where } b_k^{i*} ~:=~ b_{j,k}^{l*}
\\
r_{i,k+1} &:=& d(x_{i,k+1},y_{i,k+1}).
\end{eqnarray*}
Based on \eqref{eq:gaplengthk} and \eqref{eq:obstaclelengthk} and our choice of $t_{j,k}$ and $T_{j,k}$, it holds that
\begin{eqnarray*}
\sum_{i=1}^\infty r_{i,k+1} &=&
\sum_{j=1}^\infty \sum_{l=1}^\infty d( \gamma_{j,k}(a_{j,k}^{l*}), \gamma_{j,k}(b_{j,k}^{l*}) ) \\ &\leq&
\sum_{j=1}^\infty \Big(
\frac{r_{j,k}}{4} + \int_{\gamma_{j,k}} 1_{E_{j,k}} ~ds
\Big) ~\leq~
\frac{1}{4}\cdot 2^{-(k-1)}r + 2^{-1-k}r ~\leq~
2^{-k}r.
\end{eqnarray*}
Towards a new curve, consider sub-curve lengths
\begin{eqnarray*}
L_{k-1}^j &:=&
{\rm len}( \gamma_{k-1}\big([a_{k-1}^{j},b_{k-1}^{j}]\big) ) 
\\
L_k^* &:=&
{\rm len}(\gamma_{k-1}) +
\sum_{j=1}^\infty ( L_{j,k} - L^j_{k-1} ).
\end{eqnarray*}
for all $j$ and $k$.  
We further define a parametrization for a curve 
of length $L^*_k$, and where each $\gamma_{j,k}$ replaces $\gamma_{k-1}|_{[a^{j}_{k-1},b^{j}_{k-1}]}$, as follows:
$$
\gamma_k^*(t) :=
\begin{cases}
\gamma_{j,k} \Big(
\frac{ L_{j,k} }{ b_{k-1}^{j}-a_{k-1}^{j} } \big(t-a_{k-1}^{j}\big)
\Big), &
\text{ if } t \in [a_{k-1}^{j},b_{k-1}^{j}] \text{ for some } j \in \N,
\\
\gamma_{k-1}(t), &
\text{ otherwise.}
\end{cases}
$$
Let $\gamma_k$ be the arclength parametrisation of $\gamma_k^*$. Let $a_{k}^j,b_k^j$ correspond to $a_k^{j*}, b_k^{j*}$ under this reparametrization. 
By Equation \eqref{eq:obstacleavoidk}, $\gamma_{k-1}(t)$ can only lie in $X \setminus Y$ whenever $t \in[a_{k-1}^j,b_{k-1}^j] \text{ for some } j \in \N$, i.e where the images of $\gamma_{j,k}$ and $\gamma_k$ agree.  With the same reindexing $i=i(j,l)$, this gives the avoidance property
\begin{equation} 
\gamma_{k}^{-1}(X \setminus Y) \subset \bigcup_{i=1}^\infty (a^i_{k},b^i_{k})
\end{equation}
and since the $\gamma_{j,k}$ have length at most $\Lambda r_{j,k}$, the other avoidance property follows:
\begin{equation} 
\gamma_{k}  \setminus Y \subset \bigcup_{i=1}^\infty B(x_{i,k}, \Lambda r_{i,k}).
\end{equation}
From \eqref{eq:gaplengthk} and \eqref{eq:curvelengthold} it follows that
\begin{eqnarray}
\notag
{\rm len}(\gamma_k) &\leq& L^*_k ~\leq~
{\rm len}(\gamma_{k-1}) + \sum_{j=1}^\infty L_{j,k} \\ &\leq&
\notag
2(1-2^{-(k-1)})C_0 r + C_0 \sum_{j=1}^\infty r_{j,k} ~\\ &\leq&
\label{eq:curvelengthnew}
\frac{1-2^{-(k-1)}}{1-\frac{1}{2}}C_0 r + 2^{-(k-1)}C_0 r ~=~
\frac{1-2^{-k}}{1-\frac{1}{2}}C_0 r ~=~
2(1-2^{-k})C_0 r.
\end{eqnarray}
By construction, for each $k \in \N$ there exists $j_1,j_2\in \N$ so that $x = a_{k-1}^{j_1}$ and $y = b_{k-1}^{j_2}$, so $\gamma_k$ therefore joins $x$ and $y$.  By the previous estimate, it is therefore a $C_k$-quasigeodesic with
$$
C_k ~:=~ 2(1-2^{-k})C_0
$$
which completes the induction step.

\vspace{.05in}

\textit{ A limiting curve}.\ Putting $\gamma_k'(t) := \gamma_k\big( \frac{\Lambda r}{L_k}t \big)$, it follows that $\{ \gamma_k' \}_{k=1}^\infty$ is a family of $1$-Lipschitz functions on $[0,{ \Lambda r}]$, each joining $x$ to $y$.  By the Arzel\'a-Ascoli theorem, there therefore exists a sublimit function $\gamma: [0,{ \Lambda r}] \to X$ that is $1$-Lipschitz and joins $x$ and $y$. Since $\gamma$ is $1$-Lipschitz we obtain
$$\len(\gamma) \leq \Lambda r \leq \Lambda d(x,y),$$
and $\gamma$ is the desired $\Lambda$-quasigeodesic connecting $x$ to $y$.




We lastly claim that $\gamma([0,L]) \subset Y$.  
From the inclusion \eqref{eq:obstacleavoidkbetter} and the estimate \eqref{eq:gaplengthk}, 
the Hausdorff $1$-content 
of $\gamma_k \setminus Y$ satisfies
\begin{eqnarray*}
\mathcal{H}_\infty^1(\gamma_k \setminus Y) ~\leq~ 
\mathcal{H}_\infty^1( \gamma_k \cap \bigcup_{j=1}^\infty B(x_{j,k}, \Lambda r_{j,k})) ~\leq~ 
2^{1-k} \Lambda r 
\end{eqnarray*}
and therefore vanishes, 
as $k \to \infty$; 
we therefore
conclude $\mathcal{H}^1(\gamma \setminus Y) = 0$ since $\gamma$ is continuous and $Y$ is closed. Indeed, if $\gamma$ spent any time in the complement of $Y$, then by continuity, the Hausdorff content of $\gamma_k \setminus Y$ would have a definite lower bound for large $k$, contradicting the previous limit calculation.
\end{proof}

\subsection{Proof of Theorem \ref{thm:PIthm}, Part (a)}


In light of Theorem \ref{thm:classification2}, it suffices to prove the following statement instead
 of the original statement of Theorem \ref{thm:PIthm}:\

\begin{theorem} \label{thm:mainresult2}
Let $X$ be a 
metric measure space, 
fix 
structural constants $(p,D,C,
r_0)$, 
and let $\delta
>0$ be arbitrary. 
For every $q>p$, there exist $\epsilon_q, \tau
\in (0,1)$, $C'
\geq 1$ and $r_1\in (0,r_0)$,
such that
if $\epsilon \in (0,\epsilon_q)$ then every $(\epsilon,D,C,p)$-fillable subset $Y$ of $X$ up to scale $r_0$ is
\begin{enumerate}
\item $2D$-doubling at
scale $r_1$, and
\item $(C'
,\delta
,\tau
,q)$-max connected at scale $r_1$.
\end{enumerate}
\end{theorem}

\begin{remark}[Dependence on parameters] \label{rmk:mainresult2}
Here $r_1$ is the only constant that depends on the original scale $r_0$.  In fact, it suffices that $r_1 = r_0/(20C'
)$; see the end of Step 1 of the proof.
As for $\epsilon_q$, $\tau
$, and $C'$, they all depend on the remaining structural constants but $\epsilon_q$ and $\tau
$ depend additionally on $\delta
$ and $q$.
\end{remark}

 \begin{figure}[h!]
  \centering
    \includegraphics[width=0.70\textwidth]{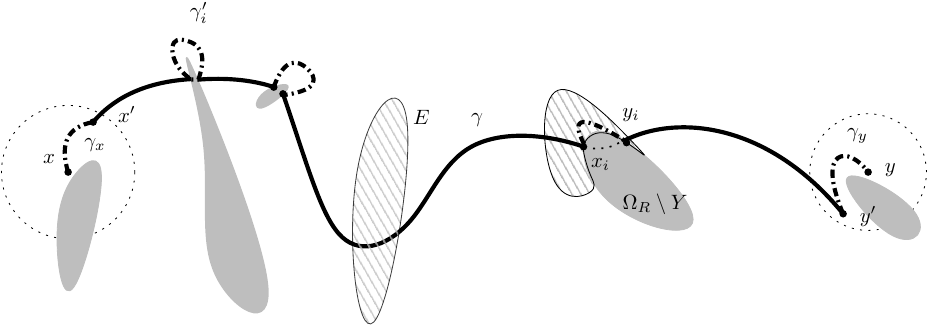}
    \caption{\small Connectivity involves constructing a curve $\gamma$ 
    that almost avoids a prescribed obstacle $E$ with small density.  In the proof of Theorem \ref{thm:mainresult2}, this requires first finding a curve in the filling $\Omega_R$ from 
    nearby 
    points $x',y'$, and then patching the curve 
    with ``detours'' $\gamma_i'$ to fully avoid $\Omega_R \setminus Y$, and 
    $\gamma_x$, $\gamma_y$ to connect $x$ and $y$. In the figure, the solid black curve indicates 
    $\gamma$
    in the filling $\Omega_R$, with the dotted parts indicating the parts replaced by the dash-dotted detours. 
    }
    \label{fig:connectivity}
\end{figure}


\begin{proof}

We proceed in three steps:\ (1) fixing parameters for definiteness, (2) passing the density conditions \eqref{eq:finemaxdensity} from points in $Y$ to points in the fillings $\Omega_r$, and then (3) constructing the quasi-geodesics explicitly.

\vspace{.05in}

\textbf{Step 1:\ Fixing parameters and their dependencies}.\
Let $\delta 
\in (0,1)$ and $q>p$ be given. Let $\Lambda = \Lambda(D,C,p)$ be the constant from Lemma  \ref{lemma:quasiconvex}, and let $\epsilon_0$ be the filling threshold 
for $\Lambda$-quasiconvexity to be guaranteed for $Y$. Each filling $\Omega_r$ satisfies $(1,p)$-Poincar\'e inequalities at  scale $r$, so by Theorem \ref{thm:classification1} and Lemma  \ref{lem:deltasmall} there exists a constant $C_0 \geq 1$ 
such that for any $\delta'>0$ there is some $\tau_0 \in (0,1)$ such that $\Omega_r$ is $(C_0,
\delta',\tau',q)$-max connected at scale $r/2$ --- i.e.\ it is $q$-maximally connected at scale $r/2$ and level $\tau'$
with constants $(C_0,\delta')$ for each $\tau' \in (0,\tau_0)$. 

Now choose
$\delta' \in (0,1)$
sufficiently small so that both conditions below hold:
\begin{eqnarray}
\Lambda(6(2D)^\frac{4}{q}+2)\delta' ~<~ \delta
, 
\label{eq:deltachoice1} \\
(2\Lambda+2C_0
+ 4\Lambda (2D)^{\frac{4}{q}})\delta' \leq C_0.
\label{eq:deltachoice2}
\end{eqnarray}
In particular, \eqref{eq:deltachoice1} implies that $\delta' \leq \frac{1}{4}$.
This 
fixes $\tau_0 \in (0,1)$  with dependence on data $\tau_0 = \tau_0(C,D,\delta',q)$ as from 
Theorem \ref{thm:classification1} and Lemma \ref{lem:deltasmall}, in which case the fillings $\Omega_{r}$ are $(C_0, 
\delta', \tau_0, q)$-max connected at scale $r/2$.
In particular, we may assume $\tau_0 < C_0^{-q}$.

Next, choose $\tau
 \in (0,1)$ with analogous dependence
$\tau
 = \tau
(C,D,\delta
,q)$ so that 
$$
(2D)^{4}\tau
 \leq \frac{\tau_0}{2},
$$
and let $m = m(C_0) \in \N$ and $n = n(\tau
,\delta,p) \in \mathbb{N}$ satisfy
\begin{align}
\label{eq:m}
2^{m-1} < C_0
+1 <
2^m
\text{ and } 
\frac{1}{2}\delta' \tau^{1/p}
\leq 2^{-n} <
\delta' \tau^{1/p}
< \delta' \tau^{1/q}.
\end{align}
Letting
$
\epsilon_q := \min\big\{
\frac{1}{4}D^{-(5+n+m)}\tau
, 
\epsilon_0
\big\}
$, 
it follows that $\epsilon_q < \frac{1}{2}$ and each $\epsilon \in (0,\epsilon_q)$ satisfies
\begin{equation} \label{eq:exponents}
((2D)^{4}\tau
 + 4D^{5+n+m}\epsilon)^{1/q} < 2(2D)^{4/q}
 \tau^{1/q}
\end{equation}
 and in particular, that
 \begin{equation} \label{eq:epschoice}
(2D)^{4}\tau
 + 4D^{5+n+m}\epsilon < \tau_0.
\end{equation}
Now let
$\epsilon \in (0,\epsilon_q)$ and $r_0 > 0$ be given,
and assume that $Y$ is a $(\epsilon,D,C,p)$-PI fillable subset of $X$,
up to scale $r_0$.  Since $\epsilon_q \leq \frac{1}{2}$, it follows from Lemma \ref{lemma:doubling} that $Y$ is $(2D,r_0)$-doubling.

Fix $C'
 = 2C_0
$.  We now show $Y$ is $(C'
, \delta
, \tau
, q)$-max connected at scale $r_1 = r_0/(20C'
)$. 

\vspace{.05in}
\textbf{ Step 2:\ Finding nearby dense points}.\
To verify $(C'
, \delta
, \tau
, q)$-max connectivity at scale $\frac{1}{20 C'
}r_0$, 
take an arbitrary pair $x,y \in Y$ satisfying $r:=d(x,y) \in (0,\frac{1}{20 C' 
}r_0)$ and an arbitrary Borel set $E$ such that
\begin{align} \label{eq:densitypt}
 M^0_{C'
 r}1_E(x) < \tau
 \text{ and } M^0_{C'
 r}1_E(y) < \tau.
\end{align}
Our goal is to construct a curve $\gamma$ in $Y$ with length at most $C'
r$ which connects $x$ and $y$ with
\[
 \int_\gamma 1_E~ds \leq \delta
 \tau^{1/q}
 r.
\]
Let $\Omega_{2C'
 r}$ be a filling of $Y$ 
from Definition \ref{def:PIfilling}, so
\begin{equation} \label{eq:fillable}
\frac{ \mu( \Omega_{2C'
 r} \cap B(x,2C' 
 r) \setminus Y ) }{ \mu(\Omega_{2C' 
r} \cap B(x,2C' 
 r)) } < \epsilon
\end{equation}
and as a shorthand, for $\rho > 0$ put 
$$
B_{2C'r}(x,\rho) = B(x,\rho) \cap \Omega_{2C'r}.
$$
Computing first with \eqref{eq:fillable} and the $D$-doubling property of $\Omega_{2C'
r}$ yields
\begin{align*}
\mu(B_{2C'r}(x,2C'
 r) 
 \setminus Y) \stackrel{\eqref{eq:fillable}}{\leq}&~
\epsilon\mu(B_{2C'r}(x,2C'
 r) 
 ) \\ ~\stackrel{
\eqref{eq:m}}{\leq}&~ D^{m+1}\epsilon\mu(B_{2C'r}(x,r) 
 ) 
 \stackrel{
\eqref{eq:m}
}{\leq}~ D^{m+n+1}\epsilon\mu(B_{2C'r}(x,\delta' \tau^{1/q}
r) 
 )
\end{align*}
as well as the estimate below, where $B_Y$ is the ball in $Y$:
\begin{eqnarray}
 \mu(B_Y(x,\delta' \tau^{1/q}
 r) 
 ) &\geq& \mu(B_{2C'r}(x,\delta' \tau^{1/q}
 r) 
 ) - \mu(B_{2C'r}(x,2C'
 r) 
 \setminus Y) \nonumber \\
 &\geq& (1- D^{m+n+1}\epsilon)\mu(B_{2C'r}(x,\delta' \tau^{1/q}
 r)) 
 ~\stackrel{\eqref{eq:epschoice}}{\geq}~
\frac{3}{4} \mu(B_{2C'r}(x,\delta' \tau^{1/q}
r) 
 ). \label{eq:xvolest}
\end{eqnarray}
Putting $R := (1+2\delta'\tau^{1/q}
)
r$,
for $l = 4D^{n+m+5}\epsilon$ consider the set 
$$
\mathcal{D}\defeq\Big\{
x' \in B_{2C'r}(x,\delta' \tau^{1/q}
r) 
 :\
M_{C_0
R}^{2C'
 r}1_{\Omega_{2C' 
 r} \setminus Y}(x') > l
\Big\},
$$
and 
note that 
$C_0
(1+3 \delta' \tau^{1/q}
)r \leq C'
r$,
so 
Lemma \ref{lem:localmaximal} implies that
\begin{eqnarray*}
\mu\left(\mathcal{D}\right) &\leq&
\frac{ D^3 }{ l } \mu(B_{2C'r}(x,C_0
(1+3\delta' \tau^{1/q})
r)\setminus Y 
 ) \\ &\leq&
\frac{ D^3 }{ l } \mu(B_{2C'r}(x,2C'
r) \setminus Y 
) \\ &\stackrel{\eqref{eq:fillable}}{\leq}&
\frac{ D^3\epsilon }{ l } \mu(
B_{2C'r}(x,2C'
 r)) \\ &\stackrel{ \eqref{eq:m}
 }{\leq}&
\frac{ D^{n+m+5}\epsilon }{ l } \mu(B_{2C'r}(x,\delta' \tau^{1/q}
r) 
) 
~=~ 
\frac{ \mu(B_{2C'r}(x,\delta' \tau^{1/q}
r) 
) }{4} 
~\stackrel{\eqref{eq:xvolest}}{\leq}~
\frac{\mu(B_Y(x,\delta' \tau^\frac{1}{q} 
r)
) }{3}.
\end{eqnarray*}
A similar argument with $l = (2D)^{4}\tau
$ 
yields
\begin{eqnarray*}
\mu( \{ x' \in  B_Y(x,{\delta' \tau^{1/q}
}r) 
; M^0_{\delta' \tau^{1/q}
r}1_E(x') > l \} ) &\leq&
\frac{ (2D)^3 \mu(E \cap B(x,2\delta' \tau^{1/q}
r)) }{l} \\ 
 &\stackrel{\eqref{eq:m}}{\leq}&
\frac{ (4D)^{4} \mu(
B_Y(x, \delta' \tau^{1/q}
r)) }{l} M^0_{C'
 r}1_E(x) \\ &\leq&
\frac{ (4D)^{4} \mu(
B_Y(x, \delta' \tau^{1/q}
r)) }{l} \tau
 \\
&<&
\frac{\mu(
B_Y(x,\delta' \tau^\frac{1}{q}
r ) )}{2}.
\end{eqnarray*}
As a result of 
the previous estimates, 
there exist 
$x' \in B(x,\delta' \tau^{1/q}
r) \cap Y \subset \Omega_{2C' r}$ so that
\begin{equation} \label{eq:fillable2}
4D^{n+m+5}\epsilon >
M_{C_0
R}^{2C'r}1_{\Omega_r \setminus Y}(x') 
\end{equation}
as well as
\begin{equation} \label{eq:maxnearby0}
(2D)^{4}\tau
 >
M^0_{\delta' \tau^{1/q}r
}1_E(x').
\end{equation}
With $R$ as before, note that 
any $s \in (\delta' \tau^{1/q}
r, C_0
R)$ and 
$x' \in B(x,\delta' \tau^{1/q}
r)$
satisfy 
\begin{eqnarray*}
B(x',s) \,\subset\,
B(x,s+ \delta' \tau^{1/q}
r) \,\subset\,
B(x,2s) \subset B(x,C'
 r).
\end{eqnarray*}
Then, doubling and our previous assumption \eqref{eq:densitypt} on $x$
yield
\begin{eqnarray}
\notag
\avint_{
 B_{2C'r}(x',s)} 1_E ~d\mu &=&
\frac{ \mu( E \cap B(x',s) ) }{ \mu(
B_{2C'r}(x',s)) } \leq
\frac{ \mu( E \cap B(x,s+\delta' \tau^{1/q}
r
) ) }{ \mu(
B_Y(x',s)) } \\ &\leq&
\notag
\tau
 \frac{ \mu(
 B_Y(x,s+\delta' \tau^{1/q}
r
) ) }{ \mu(
B_Y(x',s)) } \leq 
\notag 
\tau
 \frac{ \mu(
 B_Y(x,2s) ) }{ \mu(
 B_Y(x',s)) } \\ &\leq&
\notag 
\tau
 \frac{ \mu(
 B_Y(x',4s) ) }{ \mu(
 B_Y(x',s)) } \leq
(2D)^2\tau.
\end{eqnarray}
As for
$s \in (0,\delta' \tau^{1/q}
r)$
and for 
$x'$ satisfying \eqref{eq:maxnearby0}, we have
$$
\avint_{
B_{2C'r}(x',s)} 1_E ~d\mu =
\frac{ \mu( E \cap B(x',s) ) }{ \mu(
 B_{2C'r}(x',s)) } \leq
\frac{ \mu( E \cap B(x',s) ) }{ \mu(
B_Y(x',s)) } \leq
(2D)^{4}\tau
,
$$
so the previous two estimates 
combine to yield
\begin{equation} \label{eq:maxnearby}
M_{C'R}^{2C'r}(1_E)(x') \leq (2D)^{4}\tau.
\end{equation}
Put
$F_r = E \cup \Omega_{2C' r} \setminus Y$.
Subadditivity of the maximal function and Equations 
\eqref{eq:fillable2} and \eqref{eq:maxnearby} further yield
$$
M_{C'R}^{2C' r}(1_{F_r})(x') \,\leq\,
M_{C' R}^{2C' r}(1_E)(x') + M_{C' R}^{2C'
 r}(1_{\Omega_{2C'
 R} \setminus X})(x') \,\leq\,
(2D)^{4}\tau
 + 4D^{5+n+m}\epsilon \stackrel{\eqref{eq:epschoice}}{<} \tau_0.
$$
Similarly, since $M^0_{C' r}1_E(y) < \tau$ there exists $y' \in B(y,\delta' \tau^\frac{1}{q}) \cap Y \subset \Omega_{2C' r}$ so that
$$
M^{2C'
 r}_{C'
 R}1_{F_r}(y') < 
\tau_0.
$$
%

\textbf{Step 3:\ Arranging quasi-geodesics}.\
The space
$\Omega_{2C' 
 r}$ is $(C_0
, \delta', (2D)^{4}\tau
 + 4D^{5+n+m}\epsilon 
, q)$-max connected at scale $C'
r$. 
Since
$$
d(x',y')  \leq
d(x',x) + d(x,y) + d(y,y') \leq
\delta' \tau^{1/q}
r+r+ \delta' \tau^{1/q}
r \leq R \leq 2r < C'
r,
$$
there thus exists $L > 0$ and a rectifiable curve $\gamma_1: [0,L] \to \Omega_{2C' r}$ of length at most 
$C_0 R$
and so that $\gamma_1(0) = x'$ and $\gamma_1(L) = y'$ and
\begin{equation} \label{eq:gap}
\int_{\gamma_1} 1_E ~ds \leq
\int_{\gamma_1} 1_{F_r} ~ds \leq
\delta' \big( (2D)^{4}\tau
 + 4D^{5+n+m}\epsilon \big)^{1/q}R \stackrel{\eqref{eq:exponents}}{\leq} 
2\delta' (2D)^{4/q}
\tau^{1/q}r.
\end{equation}
 
We now modify $\gamma_1$ so that it lies entirely in $Y$ and joins $x$ and $y$.
This is done by replacing portions of the curve with curves in $Y$, and appending two segments on each end. 
(See Figure \ref{fig:connectivity}.) 
This uses the $\Lambda$-quasiconvexity of $Y$ 
at scale $r_1 = \frac{r_0}{2C_0}
$ from Lemma \ref{lemma:quasiconvex}.

First, the set $\gamma_1^{-1}(\Omega_{2C'r} \setminus Y)$ is open and can be expressed as a (possibly finite) union of countably many open disjoint intervals:
\[
 \gamma_1^{-1}(\Omega_{2C'
 r} \setminus Y) = \bigcup_{i} (a_i,b_i).
\]
Let $x_i = \gamma_1(a_i)$ and $ y_i = \gamma_1(b_i)$. Since $Y$ is $\Lambda$-quasiconvex, we can find curves $\gamma'_i \co [0,L_i]\to Y$ connecting $x_i$ to $y_i$, which are parametrized by length 
and satisfy 
\[
 \sum_i L_i \leq 
 \sum_i \Lambda d(x_i,y_i) \leq 
 \Lambda \int_{\gamma_1} 1_{F_r} \, ds \leq
 2\Lambda\delta'(2D)^{4/q}
 \tau^{1/q}
 r. 
\]
Similarly as in the proof of Lemma \ref{lemma:quasiconvex},
define a curve by patching the intervals $(a_i,b_i)$ with the curves $\gamma'_i$, i.e.\
\[
\gamma^*_2(t) :=
\begin{cases}
\gamma'_{i} \Big(
\frac{ L_{i} }{ b_{i}-a_{i} } \big(t-a_{i}\big)
\Big), &
\text{ if } t \in [a_{i},b_{i}] \text{ for some } i \in \N,
\\
\gamma_1(t), &
\text{ otherwise},
\end{cases}
\]
and let $\gamma_2$ be its arclength parametrisation. 
Now, $\gamma_2$ lies entirely in $Y$, since $\gamma_1$ only lies outside of $Y$ in the intervals $(a_i,b_i)$. 
Further,
\begin{equation}\label{eq:gamma2est}
\left\{\hspace{.25in}
\begin{split}
 \int_{\gamma_2} 1_E\, ds \leq&~
 \int_{\gamma_1} 1_E\, ds + \sum_i L_i \\ \leq&~
 2\delta'(2D)^{4/q}
 \tau^{1/q}
 r + 
 2 \Lambda \delta'(2D)^{4/q}
 \tau^{1/q}
 r \leq 
 6 \Lambda \delta'(2D)^{4/q}
 \tau^{1/q}
 r,
 \end{split}
 \hspace{.75in}
 \right.
\end{equation}
and
$$
\len(\gamma_2) \leq \len(\gamma_1) + \sum_i L_i \leq 
C_0 R + 4\Lambda \delta'(2D)^{4/q}
\tau^{1/q}
r.
$$
Next, the pairs of points $x$, $x'$ and $y$, $y'$, can be joined by $\Lambda$-quasigeodesics $\gamma_x$ and $\gamma_y$, respectively. Taking the concatenated curve
$$
\gamma := \gamma_x \cup \gamma_2 \cup \gamma_y,
$$
it follows from \eqref{eq:gap} 
that the required avoidance holds
\begin{eqnarray*}
\int_\gamma 1_E~ds &\leq&
{\rm len}(\gamma_x) +
\int_{\gamma_2} 1_E ~ds +
{\rm len}(\gamma_y) \\ 
&\stackrel{\eqref{eq:gamma2est}}{<}&
\Lambda \delta' \tau^{1/q}
r +
   6\Lambda \delta'(2D)^{4/q}
   \tau^{1/q}
   r +
\Lambda \delta' \tau^{1/q}
r 
\stackrel{\eqref{eq:deltachoice1}}{<}
\delta
\tau^{1/q}
r,
\end{eqnarray*}
as well as
\begin{eqnarray*}
 \len(\gamma) &\leq& \len(\gamma_x) + \len(\gamma_2) + \len(\gamma_y) \\ &\leq& 
 \Lambda \delta' \tau^{1/q}
 r + C_0
 R + 4\Lambda \delta'(2D)^{4/q}
 \tau^{1/q}
 r + \Lambda \delta' \tau^{1/q}
 r \\ &\leq& 
 \big(C_0
 + (2\Lambda+2C_0
 + 4\Lambda (2D)^{4/q}
 )\delta' \tau^{1/q}
 \big)r 
 ~\stackrel{\eqref{eq:deltachoice2}}{\leq}~
 2 C_0
 r = C'
r.
\end{eqnarray*}
This curve satisfies the desired 
estimates, and shows $(C'
, \delta
, \tau
,q)$-max connectivity.
\end{proof}

We now apply the previous theorem to obtain Poincar\'e inequalities for fillable sets.

\begin{proof}[{Proof of Theorem \ref{thm:PIthm}}, Part (a)] 
Fix structural constants $(p,D,C,r_0)$, which in turn fix the constant $C'
= C'(D,C,p)%
$ in Theorem \ref{thm:mainresult2}. Next, let $q>p$ be given and let $\delta_{q,2D} \in (0,1)$ be as in Theorem \ref{thm:classification2} under the choice of 
structural constants $(q,2D,C',r_0)$.

Applying now
Theorem \ref{thm:mainresult2} and Remark \ref{rmk:mainresult2}, 
there exists $\epsilon_q 
> 0$ such that if $\epsilon \in (0,\epsilon_q)$ and if $Y$ is $(\epsilon,D,C,p)$-PI fillable, then $Y$ is also $2D$-doubling and $(C',\delta_{q,2D}, \tau,q)$-max connected for some $\tau$, both at scale $r_1 = r_0/(20C')$.

Since $\delta_{q,2D}$ was chosen as in Theorem \ref{thm:classification2}, the space $Y$ satisfies a $(1,q)$-Poincar\'e inequality 
with
constant $C_q = C_q(q,D,C',\tau)$ 
at scale $r_1/C_r' = r_0/C_r$ for some constants $C_r$ and $C_r'$.
\end{proof}

\begin{proof}[{Proof of Theorem \ref{thm:PIthm}, Part (b)}] 
By Part $(a)$ there is a density parameter $\epsilon_q$ such that the $(1,q)$-Poincar\'e inequality holds. 
Now, if $Y$ is asymptotically $p$-Poincar\'e fillable, then there exists for any $\epsilon>0$ a scale $r_\epsilon > 0$ where $Y$ is $(\epsilon,D,C,p)$-PI fillable. Choosing $\epsilon \in (0,\epsilon_q)$ for any fixed $q>p$, 
the local $(1,q)$-Poincar\'e inequality follows.
\end{proof}


\section{Application: Generalized Sierpi\'nski sponges and uniform domains} \label{sec:almostunif}

Here we apply the general filling theorem to prove Poincar\'e inequalities in various new contexts.

\subsection{Sierpi\'nski sponges} \label{sec:sierpinskisponge}

In this subsection we prove Theorem \ref{thm:ndimSiercarpet} for sponges $S_{\nseq}$.
A crucial property is the following separation condition, given below, for sub-cubes $R \in \nmR_{\nseq,k}$ removed through stages $1$ through $k$ in the construction of $S_{\nseq}$.

\begin{lemma}\label{lem:squaresep}
If $R, R'
\in \nmR_{\nseq, k}$ with $R \neq R'$, then 
$$
d(R,R') \geq \frac{1}{3}s_{k-1}
\text{ and }
d(R, \partial [0,1]^d) \geq \frac{1}{3}s_{k-1}.
$$
In particular, the removed sub-cubes are uniformly $\frac{1}{3\sqrt{d}}$-separated.
\end{lemma}

\begin{proof}
Without loss of generality let $R \in \mR_{\nseq,l}$ and $R' \in \mR_{\nseq,l'}$
with $k \geq l \geq l'$. 
Let $T$ be the unique cube in $\mathcal{T}_{l-1,{\nseq}}$ 
that contains $R$. Clearly $R' \cap T \subset \partial T$ and $n_l \geq 3$, so
$$
d(R,R') \geq d(R, \partial T) ~\geq~
\frac{1}{3}s_{l-1} ~\geq~
\frac{1}{3}s_{k-1}
$$
and moreover
$$
\frac{1}{3}s_{l-1} ~\geq~
\frac{1}{3}s_{l} ~\geq~
\frac{1}{3}\min\left\{
\frac{\diam(R)}{\sqrt{d}},
\frac{\diam(R')}{\sqrt{d}}
\right\}.
$$
The same argument works for $\partial [0,1]^d$.
\end{proof}





%

To clarify the relationship between Case (4) in Theorem \ref{thm:ndimSiercarpet} and the other cases below, we note that
the set $S_{\nseq}$ has positive Lebesgue measure if and only if
${\bf n}^{-1} \in \ell^d(\N)$,
that is
$$
\sum_{i=1}^\infty \frac{1}{n_i^d}<\infty,
$$
and this follows directly from Lemma \ref{lem:arreg} below.



The proof of Theorem \ref{thm:ndimSiercarpet} will be given
in separate lemmas. First, Case (4) is proven directly from certain consequences of Poincar\'e inequalities, namely Cheeger's Rademacher Theorem \cite{ChDiff99}.  To keep the discussion self-contained, we introduce the relevant notions in context, below.

%
\begin{proof}[Proof of Case (4) of  Theorem \ref{thm:ndimSiercarpet}]
If $S_{\nseq}$ supports a $(1,p)$-Poincar\'e inequality for some $p \geq 1$ with respect to some doubling measure $\mu$, then Cheeger's theorem \cite{ChDiff99} holds.  In particular, there exist a partition $\{S^j_{\nseq}\}$ of $S_{\nseq}$ and Lipschitz maps $\varphi^j: S^j_{\nseq} \to \R^{m_j}$ so that for every Lipschitz function $f: S_{\nseq} \to \mathbb{R}$ there exists a unique $L^\infty$-vectorfield $D^jf: S^j_{\nseq} \to \mathbb{R}^{m_j}$ so that, for $\mu$-a.e.\ $x \in S^j_{\nseq}$, it holds that
$$
\frac{f(y)-f(x) - D^jf(x)\cdot(\varphi^j(y)-\varphi^j(x))}{|x-y|} \to 0
$$
as $y \to x$.  By a result of Keith \cite[Theorem 2.7]{keithdiff}, the components $\varphi^j_k$ of each $\varphi^j$ can be chosen to be distance functions of the form
$$
\varphi^j_k(x) = |x - x^j_k|
$$
for some $x^j_k \in S^j_{\nseq}$.  Each
is (classically) differentiable everywhere except at $x^j_k$,  so each 
$D^jf(x)$ 
can be replaced 
with the vectorfield 
$$
\nabla\varphi^j(x)D^jf(x): S^j_{\nseq} \to \mathbb{R}^d,
$$
where $\nabla \varphi^j$ is the $d\times m_j$ matrix whose columns are the gradients of the components. 
In other words, each $f$ is $\mu$-a.e.\ differentiable with respect to the linear coordinate functions $x_j$ as well as the generalized ``coordinates'' $\varphi^j$. Thus, for every $U_i$ the chart $\phi^j$ can be chosen using a subset of the coordinates. Since on every positive $\mu$-measured subset of $S_\nseq$ the coordinates $x_j$ are linearly independent on $S_\nseq$, then we need all the coordinates and we can chooose the charts as $\phi^j(x)=x$. The result of De Philippis, Rindler and Marchese \cite{philiprindler}, which proves a conjecture of Cheeger, ensures that $\phi^j(S^j_\nseq)=S^j_\nseq$ has positive Lebesgue measure.
\end{proof}


As we will see, the equivalence of Conditions (1)--(3) is a special case of Theorem \ref{thm:PIthm}.  We begin with checking properties of the Lebesgue measure $\lambda$ restricted to $S_{\nseq}$.

\begin{lemma}[Basic volume estimate]\label{lem:basicvol} 
Let $T \in \mathcal{T}_{\nseq,k}$, then
$$
\exp\Big(-2\sum_{i=k+1}^\infty \frac{1}{n_j^d}\Big) \leq 
\frac{ \lambda(T \cap S_{\nseq}) }{ \lambda(T) } = \prod_{i=k+1}^\infty \left(1-\frac{1}{n_i^d}\right) \leq 
\exp\Big(-\sum_{i=k+1}^\infty \frac{1}{n_j^d}\Big). 
$$

\end{lemma}

\begin{proof}
 It is easy to show inductively that
 $$\lambda(T \cap S_{\nseq}) = \lambda(T)\prod_{i=k+1}^\infty \left(1-\frac{1}{n_i^d}\right),$$
 from which the estimate follows, since $e^{-2x} \leq 1-x \leq e^{-x}$ for $x = \frac{1}{n_j^d} \in [0,\frac{1}{2}]$.
\end{proof}

%

\begin{lemma}\label{lem:arreg}
If $\nseq$ is a sequence of odd positive integers with ${\bf n}^{-1} \in \ell^d(\N)$, then $S_{\nseq}$ is Ahlfors $d$-regular for some constant $C_{AR} = C_{AR}({{\nseq},d})$.
In particular $S_{\nseq}$ is 
$2^dC_{AR}$-doubling.
\end{lemma}

\begin{proof}
Given $x \in S_{\nseq}$, $r \in (0,\diam(S_\nseq))=(0,\sqrt{d})$, and $\rho \in (0,r]$, let $Q(x,\rho)$ be the cube with center $x$ and edges parallel to the coordinate axes and of length $\rho/\sqrt{d}$, so $Q(x,r) \subset B(x,r)$.  Choose $k \geq 1$ so that
\begin{equation} \label{eq:stageradius}
8\sqrt{d}s_k ~\leq~ r ~<~ 8\sqrt{d}s_{k-1}
\end{equation}
and let $T_{x,r} \in \mT_{k-1,\nseq}$ be such that $x \in T_{x,r}$ and define
$$
\mathcal{T}_{x,r} \defeq
\{
T \in \mathcal{T}_{k,{\nseq}} \  | \
T \subset Q(x,r) \cap T_{x,r}
\}.
$$
Let $R \in \mR_{k,\nseq}$ be the central square of $T_{x,r}$. Then $\mathcal{T}_{x,r}$ covers $Q(x,\frac{r}{2}) \cap T_{x,r} \setminus R$.  Moreover $$\lambda(Q\big(x,\frac{r}{2}) \cap T_{x,r}) \leq \lambda(Q\big(x,\frac{r}{2}) \cap T_{x,r} \setminus R) + \lambda(R).$$ Thus,
$$
 2|\mathcal{T}_{x,r}| s_k^d \geq |\mathcal{T}_{x,r}| s_k^d + \lambda(R) \geq \lambda(Q\big(x,\frac{r}{2}\big) \cap
T_{x,r}
\setminus R) + \lambda(R)
\geq \lambda(Q\big(x,\frac{r}{2}) \cap T_{x,r}),
$$
since $|\mathcal{T}_{x,r}| \geq 2$, and $\lambda(R) = s_k^d$. The estimate
\begin{equation}\label{eq:txrest}
|\mathcal{T}_{x,r}|s_k^d \geq
\frac{1}{2}\lambda(Q(x,\frac{r}{2}) \cap T_{x,r}) \geq
\frac{1}{2}\min\{r/(2\sqrt{d}),s_{k-1}/2\}^d \geq \frac{r^d}{2 \left(2^{4}\sqrt{d}\right)^d}
\end{equation}
follows easily from \eqref{eq:stageradius}, because $Q(x,\frac{r}{2}) \cap T_{x,r}$ is a rectangle
with 
side lengths at least $\min\{r/(2\sqrt{d}),s_{k-1}/2\}$.
Thus, using the fact that for any $k$ and any $T \in \mT_{\nseq,k}$,
\begin{equation} \label{eq:subcubemeasure}
\lambda(T \cap S_{\nseq}) = c_{{\nseq},k}\lambda(T),
\text{ where }
c_{{\nseq},k} = \prod_{j=k+1}^\infty \Big( 1 - \frac{1}{n_j^d} \Big).
\end{equation}
Lemma \ref{lem:basicvol} implies that
\begin{eqnarray*}
 \notag
 \lambda(B(x,r) \cap S_{\nseq}) &\geq& 
 \lambda(Q(x,r) \cap S_{\nseq} \cap T_{x,r}) 
 \\ &\geq& 
 \sum_{T \in \mathcal{T}_{x,r}} \lambda(T \cap S_{\nseq}) 
 ~\stackrel{\eqref{eq:subcubemeasure}}{=}~
 c_{{\nseq},k} \sum_{T \in \mathcal{T}_{x,r}} \lambda(T) 
 ~\geq~ 
 c_{{\nseq},k} |\mathcal{T}_{x,r}| s_k^d 
 ~\stackrel{\eqref{eq:txrest}}{\geq}~
 \frac{c_{{\nseq},0}}{2}
 \frac{r^d}{(2^4\sqrt{d})^d}.
\end{eqnarray*}

The result then follows with constant
$C_{AR} = \frac{2(2^4\sqrt{d})^d} {c_{{\nseq},0}}$. Note that the upper bound for Ahlfors regularity is trivial.
\end{proof}




\begin{lemma}\label{lem:unifasympt} 
The set $S_{\nseq}$ is an asymptotically $1$-Poincar\'e fillable subset of $\R^d$.
\end{lemma}

\begin{proof}
Let $D = 2^dC_{AR}$ be the doubling constant from Lemma \ref{lem:arreg}.  
 Now, 
 consider the  
 domains $Y_1 = \R^d$ and $Y_2 = [0,1]^d$ and $Y_3=\R^d \setminus R$, for $R \in \nmR_{\nseq,k}$. Each of these satisfies a Poincar\'e inequality
with inflation factor $1$, 
that is, $CB \cap Y_i = B \cap Y_i$; see Equation \eqref{eq:atmostoneR};
this follows, for example, from \cite{hajlasz1995sobolev} and the chained ball condition which is easy to verify in this case.
In particular, for each $i = 1, 2, 3$ and
 for any ball $B\defeq B(x,s)$ and any Lipschitz function $f$ on $Y_i$ we have 
\begin{equation} \label{eq:atmostoneR}
\fint_{B \cap Y_i} |f-f_{B \cap Y_i}| \,\,d\lambda \leq C_{PI} s  \fint_{B \cap Y_i} \Lip [f]\,\, d\lambda,
\end{equation}
where 
the constant $C_{PI}$ is independent of $i$, $B$ 
and $f$. 
This holds, a priori, for any Lipschitz function in $\R^d$ 
and taking extensions as necessary,
for any Lipschitz function defined on $Y_i \cap B$. 

For each 
$\epsilon \in (0,1)$,
choose $\delta \in (0,\epsilon/4)$
so that
$$
1-(1-\delta)^d < \frac{\epsilon}{4^{d+1}\sqrt{d}^d\lambda(B(0,1))}
$$
in which case it holds, for all $r > 0$, that
\begin{equation} \label{eq:deltaannulus}
\lambda(B(x,r) \setminus B(x,r(1-\delta))) =
\lambda(B(0,1)) \cdot \big(1 - (1-\delta)^d\big)r^d < \frac{\epsilon r^d}{4^{d+1}\sqrt{d}^d}.
\end{equation}
Next, choose $j_0 \in \N$ so that both 
$\sum_{i=j_0}^\infty \frac{1}{n_i^d} < \frac{\epsilon}{4}$ 
and
$n_i \geq 2^5\sqrt{d}\delta^{-1}$ 
for all $i \geq j_0$. 
We now claim that $S_{\nseq}$ is $1$-Poincar\'e $\epsilon$-fillable
(Definition \ref{def:PIfilling}) at scale
$$
r_0 = s_{j_0+1} = \prod_{i=1}^{j_0+1} \frac{1}{n_i}
$$
with the above constants $(C_{PI},D)$.%

To see why, let $r \in (0,r_0)$ and $x \in S_{\nseq}$ be given.  Since $d,n_{j_0+1} \in \N$, it follows that $\frac{2\sqrt{d}}{\delta} s_{j_0} \geq \frac{1}{n_{j_0+1}} s_{j_0} = r_0$, 
so 
choose $k \geq j_0$ so that
$$
\frac{2\sqrt{d}}{\delta} s_{k+1} ~\leq~
r ~<~ \frac{2\sqrt{d}}{\delta} s_{k}.
$$

Now let
$\Omega_r := S_{k, {\nseq}}$. To show fillability, we need to show {\bf (i)} doubling, {\bf (ii)} a  local Poincar\'e inequality and {\bf (iii)} an $\epsilon$-density bound. By Lemma \ref{lem:arreg} the set $\Omega_r$, which contains $S_{\nseq}$ and is contained in $[0,1]^2$, is Ahlfors $2$-regular when equipped with the (restricted) Lebesgue measure and hence doubling. 

With {\bf (i)} now settled, we show the local Poincar\'e inequality {\bf (ii)}. Based on our choice of $j_0$ and $k$, we have %

%
%
$$
s_{k-1} ~=~
n_ks_k ~>~
\frac{2^5\sqrt{d}}{ 
\delta}
\frac{\delta}{2\sqrt{d}}r ~\geq~
2^4r
$$
in which case Lemma \ref{lem:squaresep} implies
\begin{align} \label{eq:sierpinskiseparation}
d(R,R') \geq \frac{1}{3}s_{k-1} > 4r
\end{align}
for all $R, R' \in \nmR_{\nseq,k}$ with $R \neq R'$. Thus for each $x \in S_{\nseq}$ there is at most one 
$R \in \nmR_{\nseq,k}$ 
that meets $B(x,2r)$.
Also, if such a cube $R$ exists, then similarly from Lemma \ref{lem:squaresep} it follows that
$$
d(R,\partial [0,1]^d) \geq 2r
$$
so $B(x,2r)$ would not intersect $\partial [0,1]^d$.

Now, for arbitrary $x \in S_{\nseq}$, fix a ball
$B(x,s) \cap \Omega_r$ with $s \leq 2r$. As before, at most one $R$ can meet $B(x,s)$, so 
$$
B(x,s) \cap \Omega_r = B(x,s) \cap Y_i
$$ 
holds for some $i=1, 2, 3$ as above, 
and Equation \eqref{eq:atmostoneR} is precisely
the local Poincar\'e inequality for $\Omega_r$ at scale $s$, 
as desired.

Finally, we show the density bound {\bf (iii)};
that is, condition (2) in Definition \ref{def:PIfilling}.
First observe that
$B(x,r) \cap \Omega_r$ contains a cube with side length $r/(4\sqrt{d})$, in which case it holds that
\begin{equation}\label{eq:volumeomega}
\lambda(B(x,r) \cap \Omega_r) \geq \frac{r^d}{4^d\sqrt{d}^d}.
\end{equation}
%
Now, 
consider all remaining $(k+1)$'th order subcubes that are sufficiently near $x$, i.e.\
$$
\mathcal{T}_{x,r} = \{T \in \mathcal{T}_{k+1,\nseq} ~|~ T \cap B(x,(1-\delta)r) \neq \emptyset\}.
$$
From our previous choice of $k$, we have for all $T \in \mathcal{T}_{x,r}$ that
$$
\diam(T) \leq 2\sqrt{d}s_{k+1} < \delta r,
$$
and thus $T \subset B(x,r)$. The cubes in $\mathcal{T}_{k+1,\nseq}$ that are contained in $\Omega_r \cap B(x,r)$ thus cover $\Omega_r \cap B(x,r)$ except for a portion of the annulus $B(x,r) \setminus B(x,(1-\delta)r)$ as well as the removed 
cubes
in $\mathcal{R}_{k+1,\nseq}$ which intersect $B(x,r)$. Let $\mathbf{R}$ be the union of such removed cubes. These extra portions have small volume, as we will see.

 Each cube 
in $\mathcal{R}_{k+1,\nseq}$ that intersects $B(x,r)$ is contained in a cube 
in $\mathcal{T}_{k,\nseq}$ of side length $s_k$, and such larger cubes have pairwise-disjoint interiors. If $r \leq s_k$ then there are at most $3^d$ such cubes, so for dimensions $d \geq 2$ we have
$$
\lambda(\mathbf{R}) \leq 
3^d s_{k+1}^d \leq 
3^d\Big(\frac{\delta r}{2\sqrt{d}}\Big)^d \leq
\Big(\frac{3\delta }{2\sqrt{2}}\Big)^dr^d \leq
(\sqrt{2}\delta)^dr^d.
$$
If $r \geq s_k$ then there are at most 
$(\frac{2r}{s_k}+2)^d$
such cubes.  Recalling that $s_k = n_{k+1}s_{k+1}$, our previous choices of $j_0$ and $k$ now yield
\begin{align*}
\lambda(\mathbf{R}) \leq 
\Big( \frac{2r}{s_k}+2 \Big)^d s_{k+1}^d =&~
\frac{2^ds_{k+1}^d}{s_k^d}(r+n_{k+1}s_{k+1})^d \\ \leq&~
\frac{2^d}{n_{k+1}^d}\Big(1+\frac{n_{k+1}\delta}{2\sqrt{d}}\Big)^dr^d =
2^d\Big(\frac{1}{n_{k+1}}+\frac{\delta}{2\sqrt{d}}\Big)^dr^d \leq
\frac{2^d\delta^d}{\sqrt{d}^d}r^d \leq
(\sqrt{2}\delta)^dr^d.
\end{align*}
Note that $\delta < \frac{\epsilon}{4} < \frac{1}{4}$ from before implies that $2-\delta > \sqrt{2}$ as well as
$$
\sqrt{2}\delta \leq
\big(1-(1-\delta)\big)(2-\delta) \leq
\big(1-(1-\delta)\big) \sum_{m=0}^{d-1} (1-\delta)^m =
1-(1-\delta)^d,
$$
so the previous paragraph, the choice of $\delta$ from before, and 
\eqref{eq:deltaannulus}--\eqref{eq:volumeomega} imply
\begin{eqnarray*}
\sum_{T \in \mathcal{T}_{x,r}} \lambda(T) &\geq&
\lambda(B(x,r) \cap \Omega_r) - \lambda(B(x,r)\setminus B(x,(1-\delta)r)) - \lambda(\mathbf{R})\\ &\geq&
\lambda(B(x, r) \cap \Omega_r) - \frac{\epsilon r^d}{4^{d+1}\sqrt{d}^d} - 
(\sqrt{2}\delta)^dr^d
\\ &\geq&
\left(1-\frac{\epsilon}{2}\right)\lambda(B(x,r) \cap \Omega_r).
\end{eqnarray*}
  Also, from Lemma \ref{lem:basicvol} 
  for every $T \in \mathcal{T}_{x,r}$ 
  we get
$$
\lambda(T\cap S_{\nseq}) ~\geq~
\exp\Big(
-2\sum_{i=k}^\infty \frac{1}{n_i^d}
\Big) \lambda(T) ~\geq~
\left(1-\sum_{i=j_0}^\infty \frac{2}{n_j^d}\right)\lambda(T) ~\geq~
\Big(1-\frac{\epsilon}{4^{d}\sqrt{d}^d}\Big)\lambda(T)
$$
and as a result,
\begin{eqnarray*}
     \lambda(B(x,r) \cap S_{\nseq}) &~\geq~& \sum_{T \in \mathcal{T}_{x,r}} \lambda(T \cap S_{\nseq}) \\
                                &~\geq~& \left(1-\frac{\epsilon}{2}\right) \sum_{T \in \mathcal{T}_{x,r}} \lambda(T 
                                ) \\
                                &~\geq~& \left(1-\frac{\epsilon}{2}\right)^2
                                \lambda(B(x,r) \cap \Omega_r) \geq (1-\epsilon) \lambda(B(x,r) \cap \Omega_r).
  \end{eqnarray*}
Thus subtracting $\lambda(B(x,r) \cap \Omega_r)$ from both sides yields the result.
\end{proof}

The equivalence of Conditions (1) through (3) in Theorem \ref{thm:ndimSiercarpet} is now easy to see.

\begin{proof}[
Proof
of
$(1) \Leftrightarrow (2) \Leftrightarrow (3)$ 
in Theorem \ref{thm:ndimSiercarpet}
]
The statement $(2) \Rightarrow (3)$ is trivial.  Note that
the contrapositive of (4) also proves that $(3) \Rightarrow (1)$. 

As for $(1) \Rightarrow (2)$,
Lemma \ref{lem:arreg} shows that $S_{\nseq}$ is in fact Ahlfors $d$-regular. Then Lemma
\ref{lem:unifasympt} shows that $S_{\nseq}$ is asymptotically $1$-Poincar\'e fillable, and thus by Theorem \ref{thm:PIthm} it satisfies a local $(1,p)$-Poincar\'e inequality at scale $r_0 = r_0(p,d, \nseq)$ for any $p>1$.
However, since $S_{\nseq}$ is connected and uniformly doubling, then as a consequence of \cite[Theorem 1.3]{bjornlocal} the entire space $S_{\nseq}$ satisfies a (global) $(1,p)$-Poincar\'e inequality. Note that, while the reference \cite{bjornlocal} deals with so called "semi-local" inequalities, in our case of bounded diameter these suffice for a global inequality.
\end{proof}

\subsection{General metric carpets}

In this section, we extend the proof of the previous section to give    examples of Sierpi\'nski 
sponges in general metric spaces. In particular, we prove Theorem \ref{thm:metrisponge}. 

The crucial role here is played by uniform domains. 
We 
note 
that conventionally, 
uniform domains are assumed to be open sets. 
Our definition,
however, will allow for closed sets
as well. Indeed, one can show that if a closed set $\Omega$ is uniform then its interior ${\rm int}(\Omega)$ is uniform.  The converse holds, at least in doubling metric spaces, if $\Omega$ is the closure of its interior.
It is worth noting that, on the other hand, a closure of a non-uniform domain may be uniform, such as in the case of a slit disk. However, our starting point will always be closed sets.


\begin{definition}[Uniform Domains]\label{def:unifdom}
Given a metric space $X = (X,d)$, $A > 0$, a subset $\Omega \subset X$, and points $x, y \in \Omega$, a continuous curve $\gamma \co [0,1] \to \Omega$ 
is called 
an {\sc $A$-uniform curve (with respect to $x$, $y$, and $\Omega$)} if it connects $x$ and $y$ with
$\diam{\gamma} \leq A d(x,y)$ and
\begin{equation}\label{eq:unifdom}
d(\gamma(t),\Omega^c) \geq A^{-1}\min(\diam(\gamma|_{[0,t]}), \diam(\gamma|_{[t,1]})).
\end{equation}
We say that $\Omega$ is {\sc $A$-uniform up to scale $r$} if for all $x,y \in \Omega$  with $d(x,y) < r$ there exists an $A$-uniform curve with respect to $x$, $y$, and $\Omega$.  

Lastly,
$\Omega$ is {\sc $A$-uniform} if 
it is $A$-uniform up to scale $r$, for all $r > 0$.
\end{definition}

Alternative definitions, and their mutual equivalence, are discussed in \cite{vaisala1988,martiosarvas}.  For example, if the space is doubling and quasi-convex, then $\gamma$ could be assumed to be a rectifiable curve and diameter could be replaced with length in the definition. So in the context of uniformity (and only in this context), by a ``curve'' we allow for curves to be continuous only, and not necessarily Lipschitz.

We remark, that in the case $\Omega=X$, the condition is vacuously satisfied if $X$ is quasi-convex, as the distance to an empty set is interpreted to be infinity.

For us, 
uniform domains are quite flexible to construct, and they 
inherit good geometric properties from the spaces containing them.
In particular, there is the following version of  \cite[Theorem 4.4]{bjornuniform}.

\begin{theorem}[Bj\"orn-Shanmugalingam] \label{thm:unifcon}
Let $1 \leq p < \infty$.  If $(X,d,\mu)$ is $D$-doubling and satisfies a $(1,p)$-Poincar\'e inequality with constant $C$, and if $\Omega$ is a closed, $A$-uniform domain up to scale $r_0$ in $X$ then, with its restricted measure and metric, $\Omega$ is also $\nD$-doubling and satisfies a $(1,p)$-Poincar\'e inequality at scale $r_0/2$ with constants 
$\nD ={\nD}(D,A)$ and $\nC = {\nC}(D,C,A,p)$.
\end{theorem}

\begin{remark} \label{rmk:localtosemi}
To be clear, 
in \cite[Theorem 4.4]{bjornuniform} only the global 
case of 
$r_0 = \infty$ and an open set $\Omega$ 
is explicitly discussed. Next, we briefly indicate the required modifications.  Indeed, uniformity implies that $\partial \Omega$ is porous, and thus has measure zero. See, e.g. \cite[Lemma 3.2]{bonkheinrho} for a result on and definition of porosity. Then, as remarked before Definition \ref{def:unifdom}, $\tilde{\Omega}={\rm int}(\Omega)$ is an open uniform domain, and satisfies the Poincar\'e inequality at scale $r_0/C$ by the argument in \cite[Theorem 4.4]{bjornuniform}. Since $\partial \Omega$ has measure zero, and $\tilde{\Omega}$ is dense in $\Omega$, the Poincar\'e inequality and doubling also hold for $\Omega$.  Following their proof, these properties hold initially at some scale $r_0/C$ with a constant $C$. 

However, 
following the proof of \cite[Theorem 4.4]{bjornlocal}
and under the additional hypothesis that $\Omega$ is metric doubling and $A$-uniform up to scale $r_0$, we may upgrade the scale to $r_0$ with a uniform constant.  In \cite{bjornlocal}, the proof uses properness and connectivity to get non-quantitative bounds on the number of balls involved and that need to be chained. However, the only modification needed is a quantitative bound on the number of such balls needed, which follows here from doubling and uniformity. We refer the reader to the proof of \cite[Theorem 4.4]{bjornlocal} for more details.
\end{remark}

\begin{remark} \label{rmk:egunif}
There are many examples of uniform domains.

\begin{enumerate}
 \item Bounded convex subsets of $\R^d$ are uniform, where the uniformity constant $A$ depends on the 
eccentricity of the convex subset.
 \item  Compact domains 
 with Lipschitz-regular boundaries in $\R^n$ 
are uniform, 
as well as their complements. 
The constants depend quantitatively on the Lipschitz constants of the local representations and the sizes of the charts covering the boundary.

 \item $C^{1,1}$-compact domains and their complements in any step-2 Carnot group, including the (first) Heisenberg group, are uniform with respect to their Carnot-Carath\'eodory metrics \cite{morbidelliheis}.  Here, $C^{1,1}$-regularity is with respect to the 
 Euclidean smooth structure.
 For an introduction to Carnot groups, we refer the reader to \cite{morbidelliheis}. See also Section \S\ref{sect:heisenberg} for a discussion of the Heisenberg group (from a purely metric space perspective). 
 \item Let $f \co X \to Y$ be a quasisymmetric map between metric spaces $(X,d)$ and $(Y,d')$, i.e.\ that there is a homeomorphism $\eta \co [0,\infty) \to [0,\infty)$ with necessarily $\eta(0) = 0$ and $\eta(t) \to
 \infty$ as $t\to\infty$ so that
$$
\frac{ d'(f(x),f(y)) }{ d'(f(x),f(z)) } \leq
\eta\Big(
\frac{ d(x,y) }{ d(x,z) }
\Big)
\text{ for all } x,y,z \in X.
$$
If $\Omega$ is a uniform domain in $X$ then $f(\Omega)$ is also uniform in $Y$. The constants are quantitative with respect to the uniformity of $\Omega$ and the distortion function $\eta$.

In particular, 
if $f \co \R^d \to \R^d$ is a $K$-quasiconformal map, then it is $\eta$-quasisymmetric \cite{vaisalabook}, and so $f(B(0,1))$ and $f(\R^d \setminus B(0,1))$ are uniform.
\item {
Recently, T.\ Rajala \cite{rajalaunif} has proven that in any quasiconvex doubling space there exists an abundance of uniform domains. 
In fact, every bounded domain can be approximated by uniform domains in the Hausdorff metric.  (The dependence on constants is not given explicitly there, but can likely 
be made explicit in some cases.%
)}
\end{enumerate}
\end{remark}

Our main theorem
has an immediate consequence
for uniform domains,
or more generally, what we call ``almost-uniform'' domains.

\begin{definition}\label{def:epscon}
A subset $Y$ of $X$  is called {\sc $(\epsilon,A)$-almost uniform at scale $r_0$} if for every $r \in (0,r_0)$ there is a
 connected, closed
subset $\Omega_r$ of $X$ that is $A$-uniform up to scale $4r$, and so that $Y \subset \Omega_r$ and for every $x \in Y$ it holds that
\begin{equation}\label{eq:epscondens}
\frac{\mu(\Omega_r \cap B(x,r) \setminus Y)}{\mu(\Omega_r \cap B(x,r))} < \epsilon. 
\end{equation}
\end{definition}

\begin{corollary}\label{cor:almostunifcon} Let 
$(p,D,C,A)$ be
structural constants and $r_0>0$.

If $(X,d,\mu)$ is a $D$-doubling space that satisfies 
a $(1,p)$-Poincar\'e inequality with constant $C$, then for any $q>p$ there exists $\epsilon>0$, depending on the structural constants, such that if $Y \subset X$ is $(\epsilon,A)$-almost uniform at scale $r_0 > 0$, then $Y$ with its restricted metric and measure 
satisfies a $(1,q)$-Poincar\'e inequality at scale $r_1 = r_1(D,C,A,r_0)$.

Moreover, if $Y$ is $(\epsilon,A)$-almost uniform for all $\epsilon \in (0,\frac{1}{2})$,
then it satisfies a $(1,q)$-Poincar\'e inequality for every $q>p$.
\end{corollary}

\begin{proof}
By applying Definition \ref{def:epscon} and Theorem \ref{thm:unifcon} to $Y$, for each $r \in (0,r_0)$ the filling $\Omega_r$ with its restricted measure is $\nD$-doubling at scale $2r$ and satisfies a $(1,p)$-Poincar\'e inequality 
at scale $2r$ with constant
$\overline{C} = \overline{C}(D,C,A,p)$ 
independent of 
$r$.
Thus, together with 
$Y \subset \Omega_r$
we see that for each $r>0$ the filling $\Omega_r$ 
satisfies Definition \ref{def:PIfilling} and thus the claim follows from Theorem \ref{thm:PIthm}.
\end{proof}

Instead of prescribing a priori ``fillings'' to subsets in the sense of Theorem \ref{thm:PIthm}, we now return to the perspective in the Introduction (\S\ref{subsect:intrometriccarnot}) and
consider constructions on general PI-spaces akin to Sierpi\'nski sponges.  In this original but opposite viewpoint, we first consider \textit{complements} of certain domains. 

\begin{definition}\label{def:exteriorunif}
Let $A > 0$.  An open, bounded subset $\Omega$ of a metric space $X$ is called {\sc $A$-co-uniform} if $X \setminus \Omega$ is $A$-uniform and $\partial \Omega$ is connected.
\end{definition}

To define ``metric sponges'' in terms of dyadic decompositions is nontrivial, as compared with Sierpi\'nski sponges in $\R^d$.
In general, metric measure spaces need not admit dyadic decompositions; even in the case of doubling measures, the cells of a Christ dyadic decomposition do not necessarily form a collection of uniform domains with a uniform constant. 

We therefore define a construction in terms of removed sets (or ``obstacles'') instead.
As 
there is no guarantee of self-similarity in an arbitrary metric space, these sets are 
given in terms of a strengthening of item (2) of Theorem \ref{thm:planar}, the uniform relative separation property applied to co-uniform domains instead of quasidisks; see item (5) below.

\begin{defn} \label{def:sparsecoll}
Let $\nseq = \{n_k\}_{k=1}^\infty$ be a sequence of positive integers, and consider scales, 
given inductively as
$s_0=1$ and
$$
s_k = \frac{1}{n_k}s_{k-1}
$$%
for $k \in \N$. 
A sequence of collections of domains $\{\mR_{\nseq,k}\}_{k=1}^\infty$ in $\Omega$ forms a {\sc uniformly $\nseq$-sparse collection of 
co-uniform sets in $\Omega$} if
there exist
constants $\delta,L > 0$ and $A \geq 1$ so that
for each $R \in \mR_{\nseq,k}$:\
\begin{enumerate}
 \item $R \subset \Omega$;
 \item $R$ is $A$-co-uniform and $\Omega$ is $A$-uniform; 
 \item $\diam(R) \leq Ls_k \diam(\Omega)$;
 \item $d(R,\Omega^c) \geq \delta s_{k-1}\diam(\Omega)$;  \item if moreover $R' \in \mR_{\nseq,k'}$ with $k \geq k'$, then
$d(R,R') \geq \delta s_{k-1} \diam(\Omega)$.
\end{enumerate}
Moreover, 
$\{\mR_{\nseq,k}\}$ is called {\sc dense} in $\Omega$ whenever $\bigcup_{k \in \N} \bigcup_{R \in \mR_{\nseq,k}}R$ is dense in $\Omega$.
We lastly define 
$$
S_\nseq \defeq \Omega \setminus \bigcup_k \bigcup_{R \in \mathcal{R}_{\nseq,k}} R.
$$
\end{defn}

It is worth mentioning here that Condition (5) appears as Equation \eqref{eq:sierpinskiseparation} and was crucial in the proof for Sierpi\'nski sponges.  It will be similarly useful in the sequel.

Recall that Theorem \ref{thm:metrisponge} asserts that:
\begin{itemize}
\item[] \textit{On an Ahlfors-regular $p$-PI space, 
the complement of a uniformly sparse collection of co-uniform sets is also an Ahlfors-regular $p$-PI space. 
}
\end{itemize}
As an initial, geometric idea of the proof,
we now state our main technical tool.

\begin{theorem}\label{thm:cutout}
Fix structural constants $A_1,A_2,C,D \geq 1$.
Let $X$ be a $C$-quasiconvex, $D$-metric doubling metric space, let $\Omega$ be an $A_1$-uniform subset of $X$, and let $S$ be a bounded, 
$A_2$-co-uniform subset of $X$.  If
$$
\overline{S} \subset {\rm int}(\Omega)
$$
then $\Omega \setminus S$ is 
$A'$-uniform 
in $X$,
with dependence $A' = A'(A_1,A_2,C,D, \frac{d(S,\Omega^c)}{\diam(S)})$.
\end{theorem}

For clarity, we postpone its proof to Appendix \ref{a:cutout}. 
Applying it to an induction argument, however, yields the following useful result:\ 
cutting out a finite collection of co-uniform domains preserves uniformity.
For simplicity, it is formulated in terms of the
relative distance, from item (2) of Theorem \ref{thm:planar}:
$$
\Delta(E,F) := \frac{d(E,F)}{\min\{\diam(E),\diam(F)\}}.
$$

\begin{corollary} \label{cor:unifmultiple} 
Fix structural constants 
$A_1,A_2,C,D \geq 1$.
Let $X$ be a
$D$-metric doubling, $C$-quasiconvex metric space, let
$\Omega$ be a $A_1$-uniform domain in $X$
and for $i=1, \dots, N$ let $S_i$ be a $A_2$-co-
uniform domain in $X$ such that $\Delta(S_i,S_j) \geq \epsilon$ for $i \neq j$ and 
{
$d(S_i, \Omega^c) \geq \epsilon \diam(S_i)$.%
}
Then $\Omega \setminus \bigcup_{i=1}^N S_i$ is also uniform in $X$.
\end{corollary}

\begin{proof} 
 Order the elements $S_i$ 
 so that $\diam(S_i) \leq \diam(S_j)$ for $i \geq j$ and 
define recursively
$$
\Omega_{i} =
\begin{cases}
\Omega \setminus S_1, & \text{ if } i=1
\\
\Omega_{i-1} \setminus S_i, &
\text{ if } 2 \leq i \leq N.
\end{cases}
$$
Put $A_0' = A_1$.  By Theorem \ref{thm:cutout} we have that $\Omega_1$ is $A_1'$-uniform with $A_1' = A'(A_0',A_2,C,D,\epsilon)$, where $A'$ is now treated as a function of the given parameters.

Proceed by induction and assume now that 
$\Omega_{n}$ is 
$A_n'$-uniform with dependence $A_n' = A'(A_{n-1}',A_2,C,D,\epsilon)$.
By the separation condition, we know that
$$
d(S_{n+1}, \Omega_{n}^c) \geq \epsilon \diam(S_{n+1}).
$$
Therefore, again by Theorem \ref{thm:cutout} we have that $\Omega_{n+1}$ is $A_{n+1}$-uniform 
with dependence $A_{n+1}' = A'(A_n',A_2,C,D,\epsilon)$.
\end{proof}

As in the 
proof 
of Theorem \ref{thm:ndimSiercarpet}, we need 
analogues of
Lemmas \ref{lem:basicvol} and \ref{lem:arreg},
but for uniformly sparse collections of co-uniform sets instead of
Sierpi\'nski sponges.  Their proofs 
being similarly straightforward, we postpone them
to Appendix \ref{a:cutout} 
and focus on how they imply Theorem \ref{thm:metrisponge} instead.

\begin{lemma} \label{lem:unifarreg}
Let $\Omega \subset X$ be an $A$-uniform subset, and assume that  $(X,d,\mu)$
is Ahlfors $Q$-regular with constant $C_{AR}$. Then $\Omega$ is Ahlfors $Q$-regular with constant $C_{AR, \Omega}=(4A)^Q C_{AR}$ when equipped with the restricted measure and metric.
\end{lemma}

\begin{lemma} \label{lem:volobstacles} 
Under the hypotheses of Theorem \ref{thm:metrisponge}, if $r\geq s_k {\rm diam}(\Omega)$ then
$$
\mu\Big(B(x,r) \cap \bigcup_{l=k+1}^\infty \bigcup_{R \in \mR_{\nseq,l}} R\Big) ~\leq~
C_\delta r^Q \sum_{i=k+1}^\infty \frac{1}{n_i^Q},
$$
holds for each $x \in S_{\nseq}$, where $C_{\delta}$ depends quantitatively on $C_{AR}$ and $Q$, as well as on $\delta$ and $L$ from Definition \ref{def:sparsecoll}.
\end{lemma}

We are now ready to verify the Poincar\'e inequality, for metric space sponges formed from uniformly sparse collections of co-uniform sets.

\begin{proof}[Proof of Theorem \ref{thm:metrisponge}] Scale the statement so that ${\rm diam}(\Omega)=1$. The domains $Y_1 = X$ and $Y_2 = \Omega$ and $Y_3 = X \setminus R$ 
are uniform domains with some constant $A$ by definition,
for any $R \in \bigcup_{k=1}^\infty \mR_{\nseq,k}$. 
So, each $Y_i$ is uniformly Ahlfors $Q$-regular with constant $C_{AR,Y}$ by Lemma \ref{lem:unifarreg}. Let $C$ be the constant of the Poincar\'e inequality of $X$, and $D$ be the doubling constant of $X$. These fix the structural constants $(p,D,C,A)$ in Corollary \ref{cor:almostunifcon}. Applying this corollary yields an $\epsilon>0$.

Local doubling and Poincar\'e inequalities will follow once we show that $S_{\nseq}$ is almost uniform. 

Let $C_\delta$ be the constant from Lemma \ref{lem:volobstacles}.
Choose first $K_\epsilon \in \N$ so large that
$$
\sum_{i=K_{\epsilon}}^\infty \frac{1}{n_i^Q} \leq \frac{\epsilon }{C_\delta C_{AR,Y}}
$$
and so that $n_i \geq \frac{2^5 A}{\delta}$ for every $i \geq K_\epsilon$. Then, define $r_0 = \delta s_{K_\epsilon+1}/(2^4 AL)$. Now, we show that $S_\mathbf{n}$ is 
$(\epsilon,A)$-almost 
uniform at level $r_0$, with the aforementioned fixed structural constants. To that avail, let $x \in S_\mathbf{n}$ and $r \in (0,r_0)$ be arbitrary. Choose $k \geq K_\epsilon$ so that 
$$
\frac{\delta s_{k}}{2^4{A}} < r \leq \frac{\delta s_{k-1}}{2^4{A}}.
$$
{
Analogously as for Sierpi\'nski sponges, put
$$
\nmR_{\nseq, k} = \bigcup_{l=1}^{k} 
\mathcal{R}_{\nseq,l}
\text{ and }
{
S_{\nseq,l} =
\Omega%
}
\setminus 
\bigcup_{R \in \nmR_{\nseq,l}}
R
$$
and just as in the proof of Lemma \ref{lem:unifasympt},
define the filling
$\Omega_r \defeq S_{\nseq,l}$.%
} 

Since ${8Ar \leq \delta s_{k-1}/2}$, there 
is
at most one $R \in  
\nmR_{\nseq, k}$ which intersects $B(x,8Ar)$, 
so
\begin{equation} \label{eq:atmostoneR2}
\Omega_r \cap B(x,8Ar) = Y_i \cap B(x,8Ar)
\end{equation}
for some $i=1,2,3$.  
Since $Y_i$ is $A$-uniform, any $y \in B(x,4r)$ can be connected to $x$ with an $A$-uniform curve with respect to $Y_i$, 
so
by \eqref{eq:atmostoneR2} that same curve is an $A$-uniform curve with respect to $\Omega_r$. That is, $\Omega_r$ is $A$-uniform at scale $4r.$ 

So to satisfy Definition \ref{def:epscon} we only need to check the density
 condition \eqref{eq:epscondens}. But, by the choice of $K_\epsilon$, we have $s_{k+1} \leq r$, and thus by Lemma \ref{lem:volobstacles}
$$
\mu\bigg(B(x,r) \cap \bigcup_{l=k+1}^\infty \bigcup_{R \in \mR_{\nseq,l}} 
R\bigg) \leq 
C_\delta r^Q \sum_{i=k+1}^\infty \frac{1}{n_i^Q} \leq \frac{\epsilon}{C_{AR,Y}} r^Q.
$$
Since $\Omega_{r} \setminus S_{\nseq}$ lies in $\displaystyle 
\bigcup_{l=k+1}^\infty \bigcup_{R \in \mR_{\nseq,l}} 
R$, we estimate its density in $B(x,r)$ to be
\begin{align*}
\frac{\mu(\Omega_{r} \setminus S_{\nseq} \cap B(x,r))}{\mu\left(\Omega_{r} \cap B(x,r)\right)} &\leq  
\frac{\displaystyle
\mu\bigg(B(x,r) \cap 
\bigcup_{l=k+1}^\infty \bigcup_{R \in \mR_{\nseq,l}} 
R\bigg)}{\mu(\Omega_{r} \cap B(x,r))} 
\leq \frac{\frac{\epsilon}{C_{AR,Y}} r^Q}{\frac{1}{C_{AR,Y}} r^Q} \leq \epsilon.
\end{align*}
Here, we again used 
\eqref{eq:atmostoneR2}
and that $Y_i$ are Ahlfors $C_{AR,Y}$-regular, 
for some $i = 1,2,3$.

This 
verifies all
the conditions in Definition \ref{def:epscon}, 
in which case
the conclusion of the Theorem follows by Corollary \ref{cor:almostunifcon}. Finally,
the remark on density is trivial, and the remark on the exponent $p$ follows from Keith-Zhong \cite{keith2008poincare}, since our spaces are complete. To be more specific, Keith-Zhong is applied first to $X$ to improve its Poincar\'e inequality, and then the first part is applied to obtain a better inequality for the fillable set $Y$. The density is also explained in more detail in the context of the Heisenberg group below.

Finally, an estimate as above using Lemma \ref{lem:volobstacles} gives the Ahlfors regularity of $S_{\nseq}$ for balls of size $r<r_0$. Since $\Omega$ is bounded, 
the Ahlfors regularity then follows immediately. Indeed, the upper bound in Ahlfors regularity follows from that of $X$, and the lower bound from $\mu(B(x,r)) \geq \mu(B(x,r_0))$ if 
$r \geq r_0$. Further, the local Poincar\'e inequality upgrades to a Poincar\'e inequality (since $\Omega$ is bounded) from \cite[Theorem 7.3]{bjornlocal} once we see that $S_{\nseq}$ is connected. To see this let $x,y \in S_{\nseq}$ be arbitrary, and let $\gamma$ be any continuous curve in $\Omega$ connecting $x,y$. Let

$$E = (\gamma \cap S_{\nseq}) \cup \bigcup_{k = 1}^\infty \bigcup_{R \in \mR_{\nseq,k}, R \cap \gamma \neq \emptyset} \partial R.$$
The set $E$ is easily seen to be a connected compact subset of $S_{\nseq}$ (since $\partial R$ are connected by assumption), and thus $S_{\nseq}$ is connected.
\end{proof}

\subsection{ Non-Euclidean 
examples:\
Heisenberg meets Sierpi\'nski} \label{sect:heisenberg}

We 
briefly discuss 
the (first) Heisenberg group 
$\mathbb{H}$, which is
a nilpotent Lie group of step $2$ and 
in particular, a topological $3$-manifold.  Though the same results apply to all step-$2$ Carnot groups, we restrict our discussion to this case, for ease of exposition. 

When equipped with the so-called Carnot-Carath\'eodory metric $d_{CC}$ induced from its Lie algebra of vector fields, $\mathbb{H}$ becomes a highly non-Euclidean metric space.
In particular, 
recent
theorems of Cheeger and Kleiner \cite{cheegerkleiner} imply
that $(\mathbb{H},d_{CC})$ admits no isometric (or even bi-Lipschitz) embedding into any Hilbert space. Their proof uses the fact that $\mathbb{H}$ 
satisfies a $(1,1)$-Poincar\'e inequality
and therefore a Rademacher-type theorem for Lipschitz functions.

As for specific properties, topologically we have $\mathbb{H} = \R^3$ but 
the
group law 
$$
(x,y,t) \times (u,v,w) = (x+u, y+v,t+w + \frac{1}{2}(xv-uy))
$$
induces a 
Lie group structure on $\mathbb{H}$
with an associated nilpotent Lie algebra.
%
%
For simplicity, 
instead of the Carnot-Carath\'eodory distance $d_{CC}$ on $\mathbb{H}$, as discussed say in Montgomery's book \cite{montgomery},
we introduce
the Korany\'i norm 
$$
N(x,y,t) = \left((x^2 + y^2)^2 + t^2\right)^\frac{1}{4},
$$
which induces another
distance 
$
d(p,q) = 
N(q^{-1}p)
$,
between points $p,q \in \mathbb{H}$, 
that is bi-Lipschitz equivalent to 
$d_{CC}$.
Moreover, $N(x,y,t) \leq \sqrt{\|(x,y,t)\|_2}$ if $\|(x,y,t)\|_2 \leq 1$.

It is known that the Haar measure on $\mathbb{H}$ is the usual Lebesgue measure $\lambda$ on $\mathbb{R}^3$ and that $\mathbb{H}$ is Ahlfors $4$-regular with respect to it.
Somewhat surprisingly, $(\mathbb{H},d_{CC},\lambda)$ 
satisfies a $(1,p)$-Poincar\'e inequality. 
The $p=2$ case was first observed by Jerison \cite{jerisonpoincare}; for the optimal exponent $p=1$, see the proof of Lanconelli and Morbidelli \cite{lanconellimorbidelli}. %
(For more discussion about the geometry of these spaces, as well as the general theory of Carnot groups, we refer the reader to 
\cite{bellaiche}, 
\cite{montgomery}, or \cite{varopoulosgroups}.)

In the spirit of the prior subsection, we now show the existence of metric sponges in the Heisenberg group,  %
so it suffices to show the existence 
and uniform sparsity
of co-
uniform domains in $\mathbb{H}$. 
To this end, we proceed in two steps:
\begin{enumerate}
\item {\bf Geometric preliminaries}.\ Recall that
on $\mathbb{H}$ there are natural dilations
$$
\delta_s(x,y,t) = (s^{-1}x, s^{-1}y, s^{-2}t)
$$ 
that are also Lie group automorphisms. Moreover, 
for any $g \in \mathbb{H}$,
the left-translation
$$
L_g(x) = g \times x
$$
is an isometry in both the Lie group and the metric space senses, so
consider the
``conformal mappings''
$$
A_{\lambda, g} = 
L_g \circ \delta_{\lambda}. 
$$
Now if $E, \Omega$ are fixed, bounded subsets of $\mathbb{H}$ with $C^{1,1}$-boundary, then a result of Morbidelli \cite{morbidelliheis} implies that
$\Omega$ and $\mathbb{H} \setminus E$ are $A$-uniform domains for some $A>0$. 
(As an example, the 
Euclidean unit ball
$B_{\rm eucl}(0,1)$ 
as a subset of $\mathbb{H}$
has boundary $\partial E = \partial B_{\rm eucl}(0,1)$ with this regularity.) 

Further, since $A_{\lambda,g}$ act by an isometry and a scaling map, the domains
$$
A_{\lambda,g}(\mathbb{H} \setminus E) = 
\mathbb{H} \setminus A_{\lambda,g}(E)
$$ 
remain $A$-uniform as 
$\lambda \in (0,\infty)$ and $g \in \mathbb{H}$ 
vary. 

\vspace{.05in}
\item {\bf The iterative construction}.\ Fix a sequence $\nseq = \{n_i\}_{i=1}^\infty$ in $\mathbb{N}$ such that 
$\nseq^{-1} \in \ell^4(\N)$
and $n_i \geq 3$ for all $i \in \mathbb{N}$, and define 
scales $\{s_k\}_{k=0}^\infty$ exactly as in Definition \ref{def:sparsecoll}.
We will define inductively our 
obstacles
by first
choosing center points 
at every scale, and then 
choosing collections of scaled and translated copies of the Euclidean unit ball with these centers
as the
obstacles.
(In what follows, all the metric notions will be with respect to the distance on $\mathbb{H}$ defined above.)

First, let $\Omega = \overline{B}_{\rm eucl}(0,1)$, so
${\rm diam}(\Omega) \leq 2$. Now define $G_1 = \{0\}$ and 
$$
\mR_{1,\nseq} = \{
A_{s_1,0}(B_{\rm eucl}(0,1))
\}$$ 
and let 
$
S_{1, \nseq} = \Omega \setminus B_{\rm eucl}(0,s_1)
$
be the ``pre-sponge'' at the first stage.

Assuming 
$G_k, \mR_{k,\nseq}, S_{k,\nseq}$ have already been defined at some stage $k \in \mathbb{N}$, we next define 
$G_{k+1}, \mR_{k+1,\nseq}, S_{k+1,\nseq}$ 
at the next stage
as follows. Let 
$G_{k+1}$ 
be a 
collection of points such that each $g \in G_{k+1}$ satisfies 
\begin{align} \label{eq:heissepar}
d(g,\partial S_{k,\nseq}) \geq {s_k} \text{ and } d(g, g') \geq s_k 
\end{align}
for each $g' \in G_{k+1}$. 
(Such a collection 
could 
be empty.)
Moreover, call $G_{k+1}$ \textit{maximal} if no other collection of points $G'$ satisfying \eqref{eq:heissepar} strictly contains $G_{k+1}$.
Putting
$$
\mR_{k+1,\nseq} = 
\{
A_{s_{k+1},g}(B_{\rm eucl}(0,1))
~|~ g \in G_{k+1}
\},
$$
the $(k+1)$-stage pre-sponge is
$$
S_{k+1,\nseq} ~=~
S_{k, \nseq} \setminus \bigcup_{R \in \mR_{k+1,\nseq}} R ~=~
\Omega \setminus 
\bigcup_{l=1}^{k+1} \bigcup_{R \in \mR_{l,\nseq}} R.
$$
Finally, define
$$
S_{\nseq} = \bigcap_{k=1}^\infty S_{k,\nseq}.
$$
\end{enumerate}

\begin{lemma}\label{lem:sparsecollectionheis} Let $\nseq, G_k,
\mR_{k,\nseq}, S_{\nseq}, A$ be 
as above. Then, the sets
$\{\mR_{\nseq,k}\}_{k=1}^\infty$
in $\Omega$ form a uniformly $\nseq$-sparse collection of 
co-uniform 
subsets in $\Omega$.

Moreover, if each $G_{k+1}$ is chosen to be maximal, relative to $\{G_i\}_{i=1}^k$, then $\{\mR_{\nseq,k}\}_{k=1}^\infty$ is dense in $\Omega$ and $S_{\nseq}$ has empty interior.
\end{lemma}

\begin{proof}
 First, let $R_k \in \mR_{k,\nseq}$ and $R_l \in \mR_{l,\nseq}$ be arbitrary with $k \geq l$, 
so 
$R_k = A_{s_k,g_k}(B_{\rm eucl}(0,1))$ and $R_l = A_{s_l,g_l}(B_{\rm eucl}(0,1))$ 
for some $g_k \in G_k$ and $g_l \in G_l$. 

To show the separation property,
as a first case let $k > l$, so \eqref{eq:heissepar} implies that
\begin{align} \label{eq:heissepar2}
d(g_k, R_l) \geq 
d(g_k, \partial S_{l,\nseq}) \geq 
d(g_k, \partial S_{k-1,\nseq}) \geq 
s_{k-1},
\end{align}
in which case the Triangle inequality further implies 
$$
d(R_k, R_l) \geq 
d(g_k,R_l) - s_k \geq 
s_{k-1}-s_{k} \geq 
\frac{s_k}{2}.
$$
As for $k=l$, applying
\eqref{eq:heissepar2} with $l-1=k-1$ in place of $k$, as well as \eqref{eq:heissepar}, yields
$$
d(R_k,R_l) \geq 
d(g_k,g_l) - d(g_k,\partial R_k) - d(g_l,\partial R_l) \geq
s_{k-1} - 2s_k \geq
s_{k-1} - \frac{2s_{k-1}}{3} \geq
\frac{1}{6}s_{k-1} \diam(\Omega).
$$
Similarly if $k \geq l$ then \eqref{eq:heissepar} implies that%
$$
d(R_k, \Omega^c) \geq 
d(R_k, \partial S_{k-1,\nseq}) \geq
d(g_k, \partial S_{k-1,\nseq}) - s_k \geq
s_{k-1} - \frac{s_{k-1}}{2} =
\frac{1}{2}s_{k-1} \geq \frac{1}{6}s_{k-1} \diam(\Omega),
$$
so 
$\delta = \frac{1}{6}$ 
yields 
the desired separation. 
Moreover, $\diam(R_k) \leq 2s_k$ follows from construction, so the diameter bound follows with $L = 2$.

As in (1) before the statement of the Lemma, each $R_k$ has $C^{1,1}$-boundary, 
so 
each $X \setminus R_k$ is $A$-uniform 
with $A$ independent of $k$; 
the same is true of $\Omega$.
It follows that the collection $\{\mR_{\nseq,k}\}_{k=1}^\infty$ is uniformly $\nseq$-sparse.

As for density, let $x \in \Omega$ be arbitrary, 
let $r \in (0,\frac{1}{3}s_1)$,
and choose $k \geq 1$ so that 
$$
s_{k+1} < 
r \leq 
s_k.
$$
Now, $B_{\rm eucl}(x,s_{k+1})$ 
and hence $B_{\rm eucl}(x,r)$ 
must intersect some $R_l \in \mR_{l,\nseq}$ for some $l \leq k+2$, otherwise $G_{k+2} \cup \{x\}$ would form a larger collection of points satisfying the desired separation bounds; this, however, 
would
contradict maximality of $G_{k+2}$. 
\end{proof}

 Finally, we can apply Lemma \ref{lem:sparsecollectionheis} and Theorem \ref{thm:metrisponge} to 
 conclude the following 
 result.

 \begin{corollary} Let $G_k,n_k,\mR_{k,\nseq}, S_{\nseq}, \Omega,A$ be defined as above. Then $S_{\nseq}$ is a compact subset of $\mathbb{H}$ which has empty interior, is Ahlfors $4$-regular and satisfies a $(1,p)$-Poincar\'e inequality for any $p>1$.
 \end{corollary}


In conclusion, we note that the above construction applies to all step-2 Carnot groups, such as higher-dimensional Heisenberg groups,
or for that matter, any Carnot group where uniform domains exist at all scales and locations.  
Moreover, replacing the left-translations $L_g$ with Euclidean translations $x \mapsto x+g$ and the anisotropic dilations $\delta_s$ with Euclidean dilations, the analogous construction still works for Euclidean spaces $\mathbb{R}^d$. In
this case, 
this gives new examples of Sierpi\'nski carpets and sponges supporting Poincar\'e inequalities, where the complementary domains are self-similar copies of $E$, 
with $\mathbb{R}^d \setminus E$ uniform.

\begin{corollary} \label{thm:euclideansponges} 
Let $d \in \mathbb{N}$ with $d \geq 2$, let $\Omega$ be a uniform domain in $\R^d$, and 
let $E$ be a bounded open subset of $\Omega$ that is co-uniform in $\R^d$ 
with
$0 \in E$
and ${\rm diam}(E) \leq 1$. 
Given a sequence $\nseq = (n_i)_{i=1}^\infty$ in $\N$ with each $n_i \geq 3$ and with $\nseq^{-1} \in \ell^d(\N)$, if $\{G_k\}_{k=1}^\infty$ 
is a sequence of
uniformly $\nseq$-sparse collections of points in $\Omega$, defined analogously as above, then the 
set 
$$
S = 
\Omega \setminus 
\bigcup_{k=1}^{\infty} 
\bigcup_{g \in G_k} (s_kE + g)
$$ 
is Ahlfors $d$-regular and satisfies a $(1,p)$-Poincar\'e inequality for each $p>1$.  Moreover, $S$ can be chosen to have
empty interior.
\end{corollary}





\subsection{The problem of classifying Loewner carpets}

The 
previous
subsections gave a general 
construction for 
``sponges'' that satisfy Poincar\'e inequalities, including on Euclidean spaces.

By varying  the choice for subsets $E$ in Corollary \ref{thm:euclideansponges}, 
we obtain many new 
possibilities 
beyond those
in \cite{mackaytysonwildrick}.  
Instead of symmetry considerations, it is enough to impose regularity and sparsity conditions on $E$. For example, permissible subsets include $E$ convex, $E$ with connected and smooth boundary, or $E$ any quasi-ball --- that is, $E = f(B(0,1))$ where $f \co \R^d \to \R^d$ is any quasiconformal map.
Moreover, rescaled translates $s_kE+g$ of a single subset $E$ can be replaced by collections of 
uniformly co-uniform subsets $\{E_{gk}\}$, provided that each $E_{gk}$ contains the origin and has at most unit diameter.

Motivated by Corollary \ref{thm:euclideansponges},
we return to the planar case and study whether such examples of carpets are generic. In this context, we can make stronger conclusions.

We begin with the following theorem from \cite{whyburncarpet}, which gives 
topological criteria
for carpets. 
Recall that a point $x$ on a connected metric space $X$ is called a {\sc cut point} if $X \setminus \{x\}$ is disconnected and it is called a {\sc local cut point} if there exists $r > 0$ so that $x$ is a cut point of $B(x,r)$. Also, $S_3$ will be the usual $1/3$-Sierpi\'nski carpet, which in our notation from the introduction corresponds with $S_{\nseq}$ with $\nseq=(1/3, 1/3, \dots)$.

\begin{theorem}[Whyburn] \label{thm:whyburn}
Let $S$ be a compact, connected, and locally connected subset of $\R^2$ with empty interior. If $S$ has no cut points, then it is homeomorphic to $S_3$.
\end{theorem}

In what follows we refer to such sets $S$ as {\sc topological carpets}, which must satisfy
$$
\R^2 \setminus S = D_0 \cup \bigcup_{i=1}^\infty D_i,
$$
where $\{D_i\}_{i=0}^\infty$ is a dense collection of open, pairwise-disjoint Jordan domains, with $D_i$ bounded for $i \geq 1$ and with $D_0$ unbounded. (To be clear, a connected open subset $D \subset \R^2$ is called a {\sc Jordan domain} if $\partial D$ coincides with a Jordan curve.)

In fact, the Loewner condition for planar carpets implies being a topological carpet.  Formulated below as Corollary \ref{cor:loewnertop}, it is an easy consequence of the following result \cite[Theorem 3.3]{heinonen1998quasiconformal}.

\begin{theorem}[Heinonen-Koskela] \label{thm:annularlyquasi} 
Let $S$ be a Ahlfors $Q$-regular metric measure space that satisfies a $(1,Q)$-Poincar\'e inequality. Then, there is a constant $C\geq 1$ such that it is $C$-quasiconvex as well as $C$-annularly quasiconvex, that is for every $z \in S$ and any $r>0$, if $x,y \in S \setminus B(z, r)$, then there exists a curve $\gamma$ 
in $X \setminus B(z,r/C)$
connecting $x$ to $y$ with $\len(\gamma) \leq Cd(x,y)$. 
\end{theorem}

\begin{corollary} \label{cor:loewnertop}
If a compact subset $S$ of $\R^2$ is Loewner --- that is, it satisfies a $(1,2)$-Poincar\'e inequality and is Ahlfors $2$-regular --- and has empty interior  then $S$ is a topological carpet. 
\end{corollary}

\begin{proof}
It is well-known from \cite{semmescurves,ChDiff99} that 
$p$-PI spaces
are quasi-convex, and are therefore both connected and locally connected.  Moreover, Loewner 
spaces lack local cut points, 
by Theorem \ref{thm:annularlyquasi}.
Thus the conditions of Theorem \ref{thm:whyburn} are met, and we know that $S$ is a topological carpet.
\end{proof}

This motivates the following definition.

\begin{definition}\label{def:poincarecarpets} A compact subset $S \subset \R^n$ 
is called a {\sc $p$-Poincar\'e sponge} 
if it has empty interior, is Ahlfors $n$-regular, 
and satisfies a $(1,p)$-Poincar\'e inequality. 
If $n=2$ then $S$ is also called a {\sc $p$-Poincar\'e carpet}.

In particular, if $n \geq 3$ and $p \leq n$, 
then $S$ is  called a {\sc Loewner sponge}. 
Also, if instead $p 
\leq 
n = 2$ then $S$ is called a {\sc Loewner carpet}.

\end{definition}

It is 
now
natural to 
reformulate the Planar Loewner problem (Question \ref{ques:planarLoewner}):

\begin{question} \label{ques:Loewnercarpets}
Can one classify 
Loewner carpets, or even $p$-Poincar\'e carpets,
in terms of the construction from Corollary \ref{thm:euclideansponges}%
?
\end{question}

There are few techniques available to treat the case of sponges in dimensions $d \geq 3$, 
but for $d=2$ 
techniques such as uniformization (see e.g. \cite{bonkuniform})
provide more possibilities for carpets.

In this subsection we give a partial answer to 
Question \ref{ques:Loewnercarpets}.
In particular, we give sufficient conditions for a 
topological carpet to be a $p$-Poincar\'e carpet, or even Loewner.  
In fact, 
two of these conditions 
are also necessary. 

To formulate our result, we proceed with a well-known characterization of quasi-disks (i.e.\ quasi-balls in dimension $d=2$) from the literature \cite{beurlingahlfors}, \cite{tukiavaisala}.  This first requires a few geometric definitions.  A Jordan curve $\gamma \co S^1 \to \R^2$ is of {\sc $C$-bounded turning}, for some $C \geq 1$, if for every $s,t \in S^1$ it holds that
\begin{equation} \label{eq:boundturn}
\min\{\diam(\gamma(J_1)), \diam(\gamma(J_2))\} ~\leq~ Cd(\gamma(s), \gamma(t)),
\end{equation}
where $J_1, J_2$ are the two 
open arcs in $S^1$ 
that satisfy $J_1 \cup J_2 = S^1 \setminus \{s,t\}$.  

A Jordan curve $\gamma \co S^1 \to \R^2$ is called a {\sc $\eta$-quasicircle}, if there exists $\gamma' \co S^1 \to \R^2$ with the same image as $\gamma$, and which is $\eta$-quasisymmetric, as given in Item (4) of Remark \ref{rmk:egunif}. A {\sc quasidisk} is a domain of the form $D=f(B(0,1))$, where $f:\R^2 \to \R^2$ is quasisymmetric. 

\begin{theorem}[Beurling-Ahlfors]\label{thm:beurlingahlfors}
A bounded Jordan domain $D$ is a quasidisk if and only if $\partial D$ is a quasicircle.
\end{theorem}

\begin{theorem}[Tukia-V\"ais\"al\"a
] 
A Jordan curve $\gamma$ is a quasicircle if and only if it of bounded turning.
\end{theorem}

Now recall the notion of relative distance from 
item (2) of Theorem \ref{thm:planar}:\ 
a collection of sets $\mathcal{R}$ is called {\sc uniformly relatively $s$-separated} if $\Delta(E,F) \geq s$ for every disjoint pair $E,F \in \mathcal{R}$.


\begin{theorem}\label{thm:necessary} If $S$ is a 
Loewner carpet, 
then there are countably many 
pairwise disjoint, Jordan domains $D_i, \Omega$ such that 
$$
S = \Omega \setminus \bigcup_{i=1}^\infty D_i.
$$
and where each
$\partial D_i$ and $\partial \Omega$ form a uniformly relatively $s$-separated collection of uniformly $\eta$-quasicircles for some $s>0$ and some distortion function $\eta: [0,\infty) \to [0,\infty)$.
\end{theorem}


\begin{proof}
As $S$ is closed, we decompose the complement into open components
$$
\R^2 \setminus S = \bigcup_{i=0}^\infty D_i
$$ 
where 
at most one component, say $D_0$, is unbounded. Define $\Omega = \R^2 \setminus D_0$. Since $S$ is Loewner, by \cite[Theorem 3.3]{heinonen1998quasiconformal}, it lacks 
local cut points. Further, by Theorem \ref{thm:annularlyquasi} we obtain that $S$ is $C$-quasiconvex and $C$-annularly quasiconvex, 
with $C\geq 1$.  It then follows from Theorems \ref{thm:whyburn} and \ref{thm:annularlyquasi} that the $D_i$ are Jordan domains with pairwise disjoint closures.

Put 
$C_b = 2C^2+1$. 
We now 
show that each $\partial D_i$ is 
of $C_b-$bounded turning, 
for all $i \in \N$. 
(For $i = 0$ the argument is similar and we omit it here.)

Let $\gamma \co S^1 \to \partial D_i$ be a parametrization of the boundary as a Jordan curve. Let $s,t \in S^1$ be arbitrary and distinct and let $J_1, J_2$ be the arcs in $S^1$ defined by these points.  Now, if $\gamma(J_1)$ or $\gamma(J_2)$ is contained in the ball $B(\gamma(s), C_bR_{st})$, 
where
$$
R_{s,t} = |\gamma(s) - \gamma(t)|
$$
then 
\eqref{eq:boundturn} clearly follows.  So assume instead that
$$
\gamma(J_j) \nsubseteq 
B(\gamma(s),C_b R_{s,t})
$$
for both $j=1,2$, so there are points 
$
x_j \in \gamma(J_j) \setminus B(\gamma(s), 2C^2 
R_{s,t})
$
for both $j=1,2$.

Since $S$ is $C$-quasiconvex, there is a rectifiable curve $\sigma_S$ joining $\gamma(s)$ and $\gamma(t)$ of length at most $CR_{s,t}$ within $S$. It is well-known, say by
C. B. Moore's work
\cite[Theorem 1]{moore}, that there exists a simple subcurve $\sigma_L'$ in $\sigma_S$ that also joins $\gamma(s)$ and $\gamma(t)$. Also, since $D_i$ is a Jordan domain, there is a simple curve $\sigma_D$ joining $\gamma(s)$ and $\gamma(t)$ while intersecting $\partial D_i$ only at those two points. Form the Jordan curve $\sigma$ by concatenating the two simple arcs $\sigma_L'$ and $\sigma_D$. Since $\sigma \subset D_i \cup B(\gamma(s), CR_{st})$, we know that $x_1,x_2\not\in \sigma$.

%

The curve $\sigma$ divides $\R^2$ into two components $U,V$ so that $\partial U=\sigma = \partial V$. Since $D_i$ is an open set containing a point of $\partial U$ and $\partial V$, we must have that $D_i$ intersects both $U$ and $V$. However, since $D_i$ is Jordan, every point in $D_i \setminus \sigma$ can be connected either to $x_1$ or $x_2$ while avoiding $\sigma$. Now, if $x_1,x_2\in U$, then every point of $D_i \setminus \sigma$ would belong to $U$, which is not possible. Similarly for $V$, and thus $x_i$ must lie in separate components of $\R^2 \setminus \sigma$, i.e. one belongs to $U$ and another to $V$. In particular $\sigma$ separates the points $x_1,x_2$.

However, $x_j \in S$, and by annular quasiconvexity there exists a curve 
connecting $x_1$ and $x_2$, within $S$ and contained in $\R^2 \setminus B(\gamma(s), 2CR_{s,t})$ and thus avoiding $\sigma$. Thus $x_1$ and $x_2$ belong to the same component of 
$\R^2 \setminus \sigma$, which is a contradiction. 

We now show uniform $s$-separation for $s = \frac{1}{2^4 C^2+2}$; 
i.e.\
for all $D_i$,$D_j$ with $D_i \neq D_j$ 
that
\begin{equation} \label{eq:disksepar}
d(D_i,D_j) \geq s \min\{\diam(D_i),\diam(D_j)\}.
\end{equation}
Supposing otherwise, there would exist a pair,
say $D_i,D_j$, 
where 
\eqref{eq:disksepar}
fails. Choose a pair of points $a \in \partial D_i, b \in \partial D_j$ with $|a-b|=d(D_i,D_j)$. Next, let $\ell$ be the line segment joining $a$ and $b$, which is contained in $\R^2 \setminus (D_i \cup D_j)$. Choose two points $x_1 \in D_i,x_2 \in D_j$ with
$$
d(x_1,a) \geq \diam(D_i)/2 \geq 8C^2 d(D_i,D_j)
\text{ and }
d(x_2,b) \geq 8C^2 d(D_i,D_j).
$$
The points $x_1,a$ divide $\partial D_i$ into two arcs $J_1,J_2$. 
Next, since $J_i$ are connected, we can find points $s_i \in J_i$ with $d(s_i,a) = 2Cd(D_i,D_j)$. Thus $d(s_1,s_2) \leq 4Cd(D_i,D_j)$ By the annular quasiconvexity condition, and combined with 
\cite[Theorem 1]{moore}, we can find a curve $\sigma_L$ connecting $s_1$ to $s_2$ within $B(a, 4C^2 d(D_i,D_j)) \setminus B(a, 2d(D_j,D_j))$. Again find a curve $\sigma_D$ within $D_i$ connecting $s_i$, and form the Jordan curve $\sigma$ by concatenation of $\sigma_L$ and $\sigma_D$. As above, this curve will separate $x_1$ and $a$. However, since $\sigma$ can not intersect $\ell$, and $x_2$ can be connected to $\ell$ while lying strictly within $D_j$, we see that $x_2$ lies in the same component defined by $\sigma$ as $a$. Hence, $x_2$ lies in a different component of $\R^2 \setminus \sigma$ than $x_1$. But this contradicts the annular quasiconvexity condition, just as before. 
\end{proof}

The assumptions of uniform separation and uniform quasidisks have appeared before in \cite[Theorem 1.1]{bonkuniform}. 

\begin{theorem}[Bonk] 
If $S= \Omega \setminus \bigcup_{i\in I} D_i$, where $D_i$ and $\Omega$, for $i \in I$ are an at most countable collection of uniformly $\eta$-quasidisks, with $\{\partial \Omega\} \cup \{\partial D_i\}_i$ uniformly relatively separated, 
then there exists a quasisymmetry $f \co \R^2 \to \R^2$, such that $$
f(S) = B(0,1) \setminus \bigcup_{i \in I} B(x_i, r_i).
$$
\end{theorem}


In other words, every such set $S$ is quasisymmetric to a similar set with circle boundaries. One can also find quasisymmetric maps with images with square boundaries, or any other self-similar shapes. The proof follows from identical arguments to \cite[Theorem 1.6]{bonkuniform}.

As a corollary we obtain a result, which is known to many specialists.

\begin{corollary}\label{cor:loewnerunif} 
If $S$ is a 
Loewner carpet, then there exist 
quasisymmetries $f: S \to S'$ and $g \co S \to S''$ so that
$$
S' = B(0,1) \setminus \bigcup_{i \in I} B(x_i, r_i)
\text{ and }
S'' = [0,1]^2 \setminus \bigcup_{i \in I} Q_i
$$
where $\{\bar{B}(x_i, r_i)\}_{i\in I}$ is a pairwise disjoint collection of 
disks
in $B(0,1)$ and 
$\{Q_i\}_{i \in I}$ is a collection of open
squares 
in $[0,1]^2$
with pairwise disjoint closures. 
\end{corollary}

This reduces the classification of Loewner carpets to the problem of classifying square carpets. As of now, though, no such classification exists, even with such explicit boundaries. 
 However, we give instead a sufficient condition in terms of an assumption on density. Let 
$\mathcal{R} \defeq \{D_i\}_{i \in I}$ 
be a countable collection of connected open sets in 
$\R^2$, 
consider the indices of those sets near a fixed ball, denoted as
\begin{equation} \label{eq:closeindices}
I(x,r) \defeq \{ i \in I : D_i \cap B(x,r) \neq \emptyset \},
\end{equation}
and for $N \in \N$, 
consider a variant of the ``$N$-fold density function'' from \eqref{eq:densityfunction}, given as
\begin{align} 
s_N(x,r) \defeq
\inf\Big\{
\sum_{i \in I(x,r) \setminus J}
\frac{\lambda(D_i)}{r^2} :
J \subset I, |J| \leq N
\Big\}.
\end{align}
Note that if $D_i$ are uniform quasidisks, then $\diam(D_i)^2 \sim \lambda(D_i)$.

The following is a more quantitative version of Theorem \ref{thm:planar}, which can be considered its corollary.

\begin{theorem} \label{thm:densitycarp}
Let $\Omega, D_i$, for $i \in I$, be a countable collection of uniform $\eta$-quasidisks such that 
$D_i \subset \Omega$ and that 
$\{\partial \Omega\} \cup \{\partial D_i\}_i$ are uniformly relatively $s$-separated. Fix $N \in \N$. For every $p \in (1,\infty)$ 
there exists $\epsilon_{p,N} > 0$, depending on $s$, $\eta$,
such that if
$$
\limsup_{r\to 0} \sup_{x \in X} s_N(x,r) < \epsilon_{p,N},
$$
then 
$S= 
\Omega \setminus \bigcup_{i\in I} D_i$
is a $p$-Poincar\'e carpet.
In particular, if there exists $N \in \N$ such that
$$
\lim_{r\to 0} \sup_{x \in X} s_N(x,r) = 0,
$$
then $S$ is a Loewner carpet.
\end{theorem}

We remark, that for self-similar Sierpi\'nski carpets $S_{\nseq}$ 
it 
follows from the proof in Theorem \ref{thm:ndimSiercarpet} that

$$
\lim_{r\to 0} \sup_{x \in X} s_1(x,r) = 0.
$$

\begin{proof}
 It is sufficient to show the first claim. 

 Firstly, as a consequence of 
 Theorem \ref{thm:beurlingahlfors}, the set $\R^2 \setminus D_i$ is a quasisymmetric image of $\R^2 \setminus B(0,1)$. Then, since uniformity is preserved under quasisymmetries \cite{martiosarvas}, we see that the $D_i$ are co-uniform domains in the sense of Definition \ref{def:exteriorunif} with 
 the same 
 uniform constant.  Similarly, the $D_i$ are all uniform domains and there is a constant $C_d$, independent of $i$, so that $\diam(D_i)^2 \leq C_d \lambda(D_i)$. Similarly $\Omega$ is a uniform domain. Let $D \leq 9$ 
be the metric doubling constant of $\R^2$.

Now fix $N$ and define for any subset $J \subset I$ the set
$$
\Omega_J \defeq \Omega \setminus \bigcup_{i \in J}D_i.
$$ 
By Corollary \ref{cor:unifmultiple}, each $\Omega_J$, with $|J| \leq D^8 N$, is an $A$-uniform domain with constant $A$ depending only on $N,s,\eta$ 
and in particular, independent of $J$,
so by Lemma \ref{lem:unifarreg} %
it is also Ahlfors $2$-regular with constant $C_\lambda$ depending only on $N, s, \eta$.

With 
$A$, $C_\lambda$, and $C_d$ now fixed, let $\epsilon 
>0$ be the constant from Corollary \ref{cor:almostunifcon} such that 
any 
$(\epsilon,A)$-almost uniform 
subset of $\R^2$ 
necessarily
satisfies a $(1,p)$-Poincar\'e inequality. Define
\begin{equation}\label{eq:epsilonpNchoice}
\epsilon_{p,N} = 2^{-3}A^{-2}C_\lambda^{-1}C_d^{-1}D^{-8}\epsilon.
\end{equation}
Now, by assumption there exists $r_0 > 0$ such that
$$
\sup_{x \in X} s_N(x,2Ar) < \epsilon_{p,N}
$$
for all $r \in (0,r_0)$. Fix such an $r\in (0,r_0)$.

To 
construct the filling, 
take an $Ar$-net\footnote{A set $\mathcal{N}$ is a {\sc$\epsilon$-net}, if it is maximal subject to the condition that for each $x_i,x_j\in\mathcal{N}$ distinct it holds that $d(x_i,x_j) \geq \epsilon$. } $\mathcal{N} = \{x_i\}$ of $S$ and 
define a covering of $S$ by balls $\mathcal{B}=\{B(x_i,2Ar)\}$. By 
the $D$-metric doubling condition, 
for $x \in S$, each
$B(x,2^4 Ar)$ 
intersects at most $D^8$ many balls in $\mathcal{B}$. Let $\mathcal{N}_{x,r}$ be the collection of the indices $i$ so that $B(x,2^4 Ar) \cap B(x_i,2Ar)$ is not empty. In other words, we have $|\mathcal{N}_{x,r}|\leq D^8$ for any $x\in S$. 

Now for each $B(x_i,2Ar) \in \mathcal{B}$,
let $I(x_i,2Ar)$ be the set of indices as in \eqref{eq:closeindices}, and 
choose a subset $J_i \subset I(x_i,2Ar)$ with $|J_i|\leq N$ so that 
$$
\sum_{j \in I(x_i,2Ar) \setminus J_i} \frac{\lambda(D_j)}{(2Ar)^2} = s_N(x_i,2Ar)<\epsilon_{p,N}.
$$
By choice of $\epsilon_{p,N}$, we have that if $j\in  I(x_i,2Ar) \setminus J_i$, then $\diam(D_j)\leq r$, as otherwise 
$$
\lambda(D_j) \geq 
C_d^{-1}\diam(D_j)^2 \geq 
(2^3A^2\epsilon_{p,N}) r^2
$$
would be a contradiction. In particular, if $D_j$ is such that $D_j \cap B(x_i,2Ar) \neq \emptyset$ and $\diam(D_j) \geq r$, then $j\in J_i$.

Now let 
$\mathcal{J} = \bigcup_{x_i \in \mathcal{N}} J_i$, and define $\Omega_r \defeq \Omega_{\mathcal{J}} = \Omega \setminus \bigcup_{i \in \mathcal{J}} D_i$. We will show that $\Omega_r$ is our desired filling.

 We first show the local uniformity at scale $4r$. Take $x,y \in \Omega_r$ with $d(x,y) \leq 4r$.  Define 
$$
J = \bigcup_{i\in \mathcal{N}_{x,r}} J_i.
$$
Since $|\mathcal{N}_{x,r}|\leq D^8$, we have $|J|\leq D^8 N$. Consider now some $j\in \mathcal{J}$ with $D_j \cap B(x,8Ar) \neq \emptyset$.  If $\diam(D_j) \geq r$, then we have an $i$ so that $B(x_i,2Ar) \cap D_j \cap B(x,2^4Ar) \neq \emptyset$ and we must have $i\in J_i \subset J$ by the choice of $\epsilon_{p,N}$ and the previous two paragraphs. If instead $\diam(D_j) \leq r$ we can take any $B(x_i,2Ar)$ which intersects $D_j$ and thus $B(x,2^4Ar)$ with $j \in J_i\subset J$. Either way, any $j\in \mathcal{J}$ such that $D_j \cap B(x,8Ar) \neq  \emptyset$ will satisfy $j\in J$. 
It follows that, for each $\rho \in (0,8Ar]$,
$$
\Omega_r \cap B(x,\rho)  = \Omega_\mathcal{J} \cap B(x,\rho) = 
\Omega_J \cap B(x,\rho).
$$

 Since $\Omega_J$ is $A$-uniform, we have that $x,y$ can be connected by an $A$-uniform curve within $\Omega_J$, which will also automatically be an $A$-uniform curve within $\Omega_r$. Similarly, we obtain that $\Omega_r$ is Ahlfors $2$-regular with constant $C_\lambda$ up to scale $2r$.

Next, we show the desired density bound. We have that 
\begin{equation}\label{eq:inclusionOmega}
\Omega_r \setminus S \cap B(x,r) = \Omega_J \setminus S \cap B(x,r) \subset \bigcup_{i\in\mathcal{N}_{x,r}} \bigcup_{j \in I(x_i,2Ar) \setminus J_i} D_j.
\end{equation} 

Then the choice in Equation \eqref{eq:epsilonpNchoice}, Inclusion \eqref{eq:inclusionOmega} and Ahlfors regularity of $\Omega_J$  lead to
\begin{eqnarray*}
\frac{\lambda(\Omega_{r} \setminus S \cap B(x,r))}{\lambda(B(x,r) \cap \Omega_{r})} &= & \frac{\lambda(\Omega_{J} \setminus S \cap B(x,r))}{\lambda(B(x,r) \cap \Omega_{J})} \leq
\frac{\displaystyle \sum_{ i\in\mathcal{N}_{x,r}} \sum_{j\in I(x_i,2Ar) \setminus J_i} \lambda(D_i)}{ \frac{1}{C_\lambda} r^2} \\
&=&  
4A^2C_\lambda \sum_{ i\in\mathcal{N}_{x,r}} s_N(x_i,2Ar) 
~<~
8A^2D^8 C_\lambda\epsilon_{p,N} ~<~
\epsilon,
\end{eqnarray*}
which is the desired density condition; the Poincar\'e inequality follows.
\end{proof}




\section{General Poincar\'e results} \label{sect:classification}

We 
begin with
some basic definitions.
In what follows, $X=(X,d)$ always refers to a metric space.

\begin{definition} A Lipschitz map $\gamma \co K \to X$ from a compact 
subset $K$ of $\R$
is called a {\sc curve fragment}
in $X$.
The domain $K$ 
is also
denoted by $\mathrm{Dom}(\gamma)$.
\end{definition}
Length for curve fragments 
is defined analogously as for curves,
that is
$$
\len(\gamma) \defeq 
{
\sup_{n \in \N}%
}
\sup_{t_1, \dots t_n \in K} \sum_{i=1}^{n-1} d(\gamma(t_i), \gamma(t_{i+1})),
$$
where we further assume $t_i \leq t_j$ for $i \leq j$. 
Furthermore, the set
$$
\undf(\gamma) = \big(\min(K), \max(K)\big) \setminus K,
$$
is always a countable union of disjoint open intervals, called {\sc gaps}, as follows:
\begin{equation}\label{eq:gaps}
\undf(\gamma) = \bigcup_i (a_i, b_i).
\end{equation}
From
this, we define the {\sc total gap size} as
$$
\gap(\gamma) \defeq \sum_{i} d(\gamma(a_i),\gamma(b_i)).
$$

The {\sc path integral} of a Lipschitz function $f: X \to \R$ over a curve fragment $\gamma$ is canonically defined as
$$
\int_\gamma f ~ds = \int_K f(\gamma(t)) d_\gamma(t) ~dt,
$$
where $d_\gamma(t)$ is the metric derivative of $\gamma$, i.e.
$$d_\gamma (t) \defeq \lim_{h \to 0} \frac{d(\gamma(t), \gamma(t+h))}{h},$$
which exists for almost every $t \in K$.  This coincides with the definition of Ambrosio \cite{ambrosio2008gradient} for curves, when first
embedding the metric space $X$ into a Banach space, such as $L^\infty$, and filling in the gaps of $\gamma$ with line segments to construct a curve. This enlarged curve has a 
well-defined
metric derivative and integral, and the ones for  curve fragments are obtained by restriction. 
For a similar discussion see \cite{sylvester:poincare,bate2015geometry}.

We will employ the proof of the characterization 
of (global) Poincar\'e inequalities 
from \cite[Lemma 5.1]{heinonen1998quasiconformal},
in order to prove new characterizations.


\begin{definition} \label{def:PIptwise}
Let $1\leq p < \infty$. A proper metric measure space $(X,d,\mu)$
is said to satisfy a
{\sc pointwise $(1,p)$-Poincar\'e inequality}
at scale $r_0>0$ with constant $C \geq 1$, if for all locally Lipschitz functions $f:\ X \to \mathbb{R}$ and all 
$x,y \in X$ with $r := d(x,y) \in(0,r_0)$ we have 
\begin{equation}\label{eq:hajlasz}
|f(x) - f(y)|\leq
C r  \left(
M_{Cr} \Lip [f]^p (x)^{\frac{1}{p}}+
M_{Cr} \Lip [f]^p (y)^{\frac{1}{p}}
\right).
\end{equation}
\end{definition}

By \cite[Lemma 5.15]{heinonen1998quasiconformal} this is equivalent to a Poincar\'e inequality. The proof in \cite{heinonen1998quasiconformal} covers  global 
Poincar\'e inequalities, but the same argument applies to the local version as well. 
For completeness, we state the result and show
the modifications,
which
only involve tracking the scales of the balls/pairs of points used. 

\begin{theorem}\label{thm:classificationAp} 
Let $D \geq 1$.  For a proper space $X$, the following conditions are equivalent:
\begin{enumerate}
\item  $X$ is $(D,r_0)$-doubling and satisfies a $(1,p)$-Poincar\'e inequality with constant $C_1 \geq 1$ at some scale $r_0>0$.
\item   $X$  is $(D,r_2)$-doubling and satisfies a
$(1,p)$-pointwise Poincar\'e inequality with constant $C_2\geq 1$ at scale $r_2>0$.
\end{enumerate}
Here, the constants in Items (1) and (2) depend quantitatively on one another, with $r_2 = r_0/2$ when going from $(1) \Longrightarrow (2)$ and $r_0 = r_2/(2C_2)$ when going $(2) \Longrightarrow (1)$. Also, in either direction,
$$
\frac{1}{C} ~\leq~ \frac{C_1}{C_2} ~\leq~ C
$$
for some universal constant $C = C(D,p)$.
\end{theorem}

\begin{proof} Assume throughout that $f$ is an arbitrary Lipschitz function. 

We first prove $(1) \Rightarrow (2)$. Choose $r_2=r_0/2$ and 
let
$x,y \in X$ satisfy $r:= d(x,y)<r_2$.  Consider balls $B_i = B(x,2^{1+i}r)$ for $i \leq 0$ and $B_i = B(y,2^{1-i}r)$ for $i > 0$, 
 all of which
have radius less than $r_0$ and thus the  local Poincar\'e inequality can be applied to them. Then for $i \leq -1$, we obtain $B_{i+1}=2B_{i}$, 
as well as
\begin{eqnarray*}
|f_{B_i} - f_{B_{i+1}}| & \leq &
\avint_{B_{i}}  |f - f_{B_{i+1}}| ~d\mu \\
                        &\leq& 
D^2 \avint_{B_{i+1}}  |f - f_{B_{i+1}}| ~d\mu 
                        ~\leq~ 
D^2 C_12^{2+i}r \left(\avint_{C_1 B_{i+1}} \Lip [f]^p ~d\mu \right)^{\frac{1}{p}}
\end{eqnarray*}
while
for $i \geq 0$, we have $B_{i+1} \subset B_i \subset 4B_{i+1}$ and

\begin{eqnarray*}
|f_{B_i} - f_{B_{i+1}}| &\leq& 
\avint_{B_{i+1}}  |f - f_{B_{i}}| ~d\mu \\
                        &\leq& 
                        \frac{\mu(B_i)}{\mu(B_{i+1})} \avint_{B_{i}}  |f - f_{B_{i}}| ~d\mu 
                        ~\leq~ 
D^2 C_12^{1-i}r  \left(\avint_{C_1 B_i} \Lip [f]^p ~d\mu \right)^{\frac{1}{p}}.
\end{eqnarray*}
Thus, we get by a telescoping sum argument that
$$
|f(x) - f(y)| \leq \sum_{i \in \Z} |f_{B_i}-f_{B_{i+1}}| \leq 4D^2 C_1 r \left(M_{2C_1r}(\Lip [f](x)^p)^{\frac{1}{p}}+M_{2C_1r}(\Lip [f](y)^p)^{\frac{1}{p}}\right).
$$

Next, we prove $(2) \Rightarrow (1)$. 
Let $r_0 = r_2/(2C_2)$ and 
fix $B=B(x,r)$
with $r<r_0
$. By subtracting the median from $f$ we can assume that
$$
\min\Big(
\mu(\{f\leq 0\} \cap B),
\mu(\{f \geq 0\} \cap B)
\Big) ~\geq~
\frac{1}{2}\mu(B).
$$
Now
define $E^\pm_k=\{\pm f\geq 2^k\} \cap B$. We first prove a weak type bound using a covering argument.
Now if $z \in E^\pm_k$ and $y \in \{\pm f \leq 0\} \cap B$, then 
$$
d(z,y) \leq 2r < 2 r_0 < r_2,
$$
so by the pointwise Poincar\'e inequality,
there exist $w \in X$ and $r_w \leq C_2r$ such that

\begin{equation}\label{eq:intest}
\fint_{B(w,r_w)} \Lip[f]^p
~d\mu \geq \frac{2^{kp-1}}{r^pC_2^p},
\end{equation}
and either $z \in B(w,r_w)$ or $y \in B(w,r_w)$. 

Suppose first that $r_w \leq r_0/8$ for each $w$ so arising.
Now by an easy argument such as in \cite[Lemma 5.1]{heinonen1998quasiconformal} 
the collection of balls $B(w,r_w)$ 
cover either $E^\pm_k$ or 
$\{\pm f \leq 0\} \cap B$.
In the latter case then we get a cover of $\{\pm f \leq 0\} \cap B$, and thus using the 5B-Covering Lemma \cite{Mattila1999} (since we have doubling at scale $2r_0$) we get
\begin{equation} \label{eq:premazya}
\mu(E^\pm_k) \leq
\frac{1}{2} \leq
\mu(\{\pm f \leq 0\} \cap B) \leq \frac{D^{3p}C_2^pr^p}{2^{kp-1}} \int_{2C_2B} \Lip[f](x)^p ~d\mu.
\end{equation}
In the case that they cover $E^\pm_k$, we obtain the same estimate by covering $E^\pm_k$ directly. 

If instead
$r_w > r_02^{-3}$ for some $w$, then the claim follows easily from doubling and using a single ball. 
%
%
%
By applying Maz'ya's trick, i.e.\ applying the above argument with the truncated function
$$
u^\pm_k(x) = \pm(\min(\max(\pm f,2^{k-1}), 2^k)-2^{k-1})
$$
in place of $f$ and at level $2^{k-1}$
 in place of $2^k$, 
and since
$$
\Lip u^\pm_k = 1_{E^\pm_{k-1} \setminus E^\pm_{k}}\Lip f
$$
almost everywhere (see e.g. \cite[Lemma 2.6]{bate2015geometry}),
then
analogously as \eqref{eq:premazya}
we obtain 
\begin{equation} \label{eq:weakest}
\mu(E^\pm_k) 
\leq \frac{2^{p+1} D^{3p}C_2^pr^p}{2^{kp}} \int_{2C_2B \cap (E^\pm_{k-1} \setminus E^\pm_k)} \Lip[f](x)^p ~d\mu,
\end{equation}
which when multiplied by $2^{kp}$ and summed over $k$ gives
\begin{equation}
\fint_B |f|^p ~d\mu \leq  2^{p+1} D^{3p}C_2^pr^p \frac{\mu(2C_2B)}{\mu(B)} \fint_{2C_2B} \Lip[f](x)^p ~d\mu.
\end{equation}
Then, via H\"older's inequality, doubling and the triangle inequality, we obtain
\begin{eqnarray*}
\fint_B |f-f_B| ~d\mu &\leq&
2\fint_B |f| ~d\mu \\
&\leq&
2 \left(\fint_B |f|^p ~d\mu\right)^\frac{1}{p} 
~\leq~
2^3 D^{5+\log_2(C_2)}C_2 r \left(\fint_{2C_2B} \Lip[f](x)^p ~d\mu \right)^\frac{1}{p},
\end{eqnarray*}
which concludes the proof.
\end{proof}

The proofs of Theorems \ref{thm:classification1} and \ref{thm:classification2} 
can be 
more 
succinctly
formulated with a certain function that measures the connectivity of a space
by rectifiable curves.
Let $p \in [1,\infty)$ be fixed.  Since 
we consider a local notion of connectivity,
we include the scale $r_0>0$ used.

First define {$\Gamma_{x,y}(L)$} to be the set of Lipschitz curve fragments connecting $x$ to $y$ and with length at most $Ld(x,y)$, let $LSC_{0,1}(X)$ be the collection of lower semi-continuous functions from $X$ to 
$[0,1]$,
%
and let {$\mathcal{E}_{x,y,C}^p(\tau)$} be the class of $\tau$-admissible functions
$$
\mathcal{E}_{x,y,C}^p(\tau) \defeq \{
g \in LSC_{0,1}(X)~|~ \left(M_{Cd(x,y)} g^p(x)\right)^{\frac{1}{p}} < \tau, \left(M_{Cd(x,y)} g^p(y)\right)^{\frac{1}{p}}< \tau
\}.
$$
Finally, define the connectivity function as follows:\
$$
\alpha^p_{r_0,C}(L,\tau) \defeq
\sup_{x \in X}
\sup_{y \in \bar{B}(x,r_0)}
\sup_{g \in \mathcal{E}_{x,y,C}^p(\tau)}
\inf_{\gamma \in \Gamma_{x,y}(L)}
\frac{\int_\gamma g ~ds + \gap(\gamma)}{d(x,y)}.
$$
Clearly 
$\alpha^p_{r_0,C}(L,\tau) \leq 1$ 
always holds,
since the trivial curve fragment $\gamma \co \{0,d(x,y)\} \to X$ with $\gamma(0)=x$ and $\gamma(d(x,y))=y$ attains the bound $1$.
For every $c \geq 1$, it is also clear that 
\begin{equation}\label{eq:sublin}
\alpha^p_{r_0,C}(L,c\tau) \leq
c\alpha^p_{r_0,C}(L,\tau),
\end{equation}
{
whereas nontrivial consequences occur for $X$ when \eqref{eq:sublin} holds for all $c > 0$.
}

\begin{lemma}\label{lem:alphachar}
Let $1 \leq p < \infty$, let $D \geq 1$, let $r_0 > 0$, and let $X$ be a $(D,r_0)$-doubling metric measure space. If for some
$C,C',L \geq 1$ with $C \leq 2C'$ we have
$$
\alpha^p_{r_0,C}(L,\tau) \leq
C'\tau
$$
for all $\tau \in (0,1]$,
then $X$ satisfies a pointwise $(1,p)$-Poincar\'e inequality {with constant $2C'$} at scale $r_0$, and moreover a
$(1,p)$-Poincar\'e inequality at scale $r_0/(2C')$.
\end{lemma}


\begin{proof}Let $x,y \in X$ with $r:=d(x,y) \in (0,r_0)$ be arbitrary and let $f:X \to \mathbb{R}$ be any %
Lipschitz function.
 By scale invariance of the Poincar\'e inequality, it suffices to assume that $f$ is $1/2$-Lipschitz, so
 by defining
$$
\tau :=
\max\Big(\left(M_{Cr} (\Lip f)^p(x)\right)^{\frac{1}{p}}, \left(M_{Cr} (\Lip f)^p(y)\right)^{\frac{1}{p}} \Big) \leq
\frac{1}{2},
$$
then, by a variant of the Vitali-Caratheodory theorem (see \cite[Lemma 2.5]{sylvesterantti} for details) for any small $\epsilon \in (0,\frac{1}{2})$, there exists a lower semi-continuous $g: X \to \mathbb{R}$ so that $\Lip f \leq g < 1$ (except possibly at $x,y$) and so that
$$
\max\big(\left(M_{Cr} g^p(x)\right)^{\frac{1}{p}}, \left(M_{Cr} g^p(y)\right)^{\frac{1}{p}} \big) \leq \tau + \epsilon \leq 1.
$$
Since $f$ is assumed $1/2$-Lipschitz, every
curve fragment $\gamma {\in \Gamma_{x,y}(L)}$
satisfies
$$
|f(x) - f(y)| \leq \int_{\gamma} g ~ds + \gap(\gamma)
$$
 so by infimizing over $\gamma \in \Gamma_{x,y}(L)$, letting $\epsilon<\tau$ and by the definition of $\tau$ above we have also
\begin{align*}
|f(x)-f(y)| ~\leq&~
r\alpha^p_{r_0,C}(L,2\tau) \\ \leq&~
2C'r \tau \leq
2C'r(
\left(M_{Cr} \Lip[f]^p(x)\right)^{\frac{1}{p}} +
\left(M_{Cr} \Lip[f]^p(y)\right)^{\frac{1}{p}}
).
\end{align*}
This is the desired pointwise estimate at scale $r_0$. Here, we use $Cr \leq 2C'r$, which is needed for the precise constants in our pointwise estimates\footnote{We remark, that one could also, alternatively, deal with two constants, that is an estimate of the form $|f(x)-f(y)|\leq Cd(x,y)(
\left(M_{\Lambda r} \Lip[f]^p(x)\right)^{\frac{1}{p}} +
\left(M_{\Lambda r} \Lip[f]^p(y)\right)^{\frac{1}{p}}
),$ where $C,\Lambda$ would be constants and not necessarily equal. As we already have many constants to keep track of, we simplify these as equal with the slightly unfortunate restriction of $C\leq 2C'$. However, as $C'$ can always be made larger, this is not significant for us.}. Finally by Lemma \ref{thm:classificationAp} we also have a $(1,p)$-Poincar\'e inequality at scale $r_0/(2C')$.
\end{proof}

The crucial part of the proof of Theorem \ref{thm:classification2} is the following estimate.

\begin{lemma}\label{lem:submultip} 
Let $1 \leq p < \infty$, let $D \geq 1$, and let 
$X$ be a $(D,r_0)$-doubling metric measure space. 
If $\tau_0 \in (0,1)$ and 
$\delta \in (0,\frac{1}{2}D^{-5/p}) 
$ are such that
$X$ is $(C,\delta,\tau_0,p)$-max connected at scale $r_0$, then
\begin{equation} \label{eq:submultip}
\alpha_{r_1,2C}^p\left(L,\tau\right)
\leq C'\tau
\end{equation}
for every $\tau \in (0,1)$ and for the choice of parameters
\begin{equation} \label{eq:submultip2}
L = \frac{C}{1-\delta \tau_0^{1/p}}
\text{ and }
r_1 = \frac{r_0}{5C}
\text{ and }
C' = \frac{2D^{5/p}C}{\tau_0^{1/p}(1-2\delta D^{5/p})}.
\end{equation}
\end{lemma}

\begin{proof}
Fix $\tau,\delta,r_1>0$ as in the statement, and let $\Lambda=2D^{5/p}\tau_0^{-1/p}$.   Let $x,y$ be arbitrary with $r := d(x,y) \in (0,r_1)$, and let $g \in \mathcal{E}_{x,y,2C}^p(\tau)$. 
Define
\renewcommand{\E}{E}
$$
\E=\{\ z \ |\   \left(M_{Cr} g^p(z)\right)^{\frac{1}{p}} > \Lambda\tau\ \}.
$$
We first prove that $\E$ has a desired maximal function bound at $x$ and $y$. 

Let $s \in (0,Cr)$ be arbitrary. We first show that, for every $z \in \E \cap B(x,s)$ we have 
\begin{equation}\label{eq:maxattainment}
M_{Cr} g^p(z) \leq  M_{2s}g^p(z).
\end{equation}
This is trivial when $2s\geq Cr$. Then consider $2s<Cr$; 
for the same reasons,
the averages 
of $g$
at scales $R \in (2s, Cr)$ are strictly smaller than the left hand side of Equation \eqref{eq:maxattainment}. 
Since
$g \in \mathcal{E}_{x,y,2C}^p(\tau)$, 
for such $R$
our choice of $\Lambda$ implies
$$
\avint_{B(z,R)} g^p ~d\mu \leq 
D \avint_{B(x,2R)} g^p ~d\mu \leq 
D\tau^p <
\frac{\Lambda^p\tau^p}{2} <
\frac{M_{Cr} g^p(z)}{2}.
$$
Thus the supremum of 
$M_{Cr} g^p(z)$ must already be attained for radii $R\in (0,2s)$. 

Then, from Equation \eqref{eq:maxattainment} we have $E \cap B(x,s) =\{z \in B(x,s)| \left(M_{\min\{2s,Cr\}} g^p \right)^{\frac{1}{p}}>\Lambda \tau \}$.
Noting first that 
$$
\min\{2s,Cr\}+s\leq Cr + s \leq 2Cr
\text{ and } 
\min\{2s,Cr\}+s\leq 4s,
$$
by Lemma \ref{lem:localmaximal} applied to the scale $s<r_0/4$, and the maximal function bound for $g$ and by local doubling we get
$$
\avint_{B(x,s)} 1_{\E} ~d\mu \leq
\frac{\mu(\{M_{\min\{2s,Cr\}} g^p >\Lambda^p \tau^p\} \cap B(x,s)) }{\mu(B(x,s))} \leq
\frac{D^3 \int_{B(x,\min\{2s,Cr\}+s)} g^p ~d\mu}{\mu(B(x,s))\Lambda^p\tau^p} <
\frac{D^5}{\Lambda^p} <\!
\tau_0.
$$

In this application of Lemma \ref{lem:localmaximal} we need the doubling at a larger scale.
Taking the supremum over $s$ we get $M_{Cr} 1_\E(x) < \tau_0$ and symmetrically $M_{Cr} 1_\E(y) < \tau_0$.
 Let $\epsilon > 0$ be arbitrary.
By
Definition \ref{def:maxconnlevel},
there exists a curve $\gamma \co I \to X$, with
$$
\int_{\gamma} 1_{\E} ~ds
~\leq~
\delta \tau_0^{\frac{1}{p}} r.
$$
Let $O=\gamma^{-1}(\E)$, which is open since the Hardy-Littlewood maximal function is lower semicontinuous, 
and define $K = (I \setminus O) \cup \{\min(I), \max(I)\}$. Then, defining $\gamma' = \gamma|_K$ we obtain a curve fragment $\gamma' \co K \to X$ with 
$$
\len(\gamma') \leq \len(\gamma) \leq Cr.
$$
Now let 
$\undf(\gamma') = \bigcup_{i} (a_i, b_i)$ as in \eqref{eq:gaps} 
and 
note that for every gap $(a_i, b_i)$ of $\gamma'$, we have $\gamma((a_i,b_i)) \subset \E$ and 
$$
d_i \defeq
d(\gamma(a_i),\gamma(b_i)) \leq
\len(\gamma|_{[a_i,b_i] \cap K}) \leq
\int_{\gamma|_{[a_i,b_i]}}1 ~ds =
\int_{\gamma|_{[a_i,b_i]}} 1_{\E} ~ds.
$$
 Thus
 summing over $i$ gives 
$$
\gap(\gamma') \leq \int_{\gamma} 1_{\E} ~ds \leq
\delta \tau_0^{\frac{1}{p}} r.
$$
Now, clearly $\gamma'$ avoids $\E$ except possibly at $x,y$. Thus, by the lower semi-continuity of $g$ we also have $g(\gamma'(t)) \leq \Lambda\tau$ for every $t \in K$. In particular,
\begin{equation}\label{eq:goodsetest}
 \int_{\gamma'} g ~ds \leq \Lambda \tau \len(\gamma') \leq \Lambda \tau Cr.
\end{equation}
By the assumption,
$\delta \tau_0^{1/p}<\frac{1}{2}$, so each of these gaps is of size less than $r_1$.
By our prior estimates, we obtain
$$
\sum_i d_i =
\gap(\gamma') \leq
\delta \tau_0^{\frac{1}{p}}r.
$$
Now let $\epsilon > 0$ be given. 
We have 
$M_{2Cd_i} \left(g^p(\gamma'(t
)\right)^{1/p} < \Lambda\tau$ 
for $t=a_i,b_i$,
so by 
the definition of $\alpha^p_{r_1,2C}(L,\Lambda\tau)$ there are curve fragments $\gamma_i$ of length at most $Ld_i$ connecting $\gamma'(a_i)$ and $\gamma'(b_i)$ and
$$
\int_{\gamma_i} g ~ds + \gap(\gamma_i)
\leq \alpha^p_{r_1,2C}(L,\Lambda\tau)d_i + 2^{-i}\epsilon.
$$
Now, by a dilation and translation we can assume that the domains of $\gamma_i$ are $[a_i, b_i]$, and that the curves are uniformly Lipschitz. Thus, we can define a new curve
$\gamma''$ by the choices
$\gamma''(t)=\gamma'(t)$ for $t \in K$ and $\gamma''(t)=\gamma_i(t)$ for $t \in 
 [a_i',b_i']
$. This is clearly Lipschitz and

$$\len(\gamma'') \leq \len(\gamma') + \sum_{i} \len(\gamma_i) \leq (C + \delta \tau_0^{1/p}L)r \leq L r.$$

Further, using the above estimates and estimate \eqref{eq:goodsetest} 

\begin{eqnarray*}
\inf_{\overline{\gamma} \in \Gamma_{x,y}(L)} \int_{\overline{\gamma}} g ~ds + \gap(\ngam) &\leq & \int_{\gamma''} g ~ds + \gap(\gamma'') \leq \int_{\gamma'} g ~ds + \sum_i \int_{\gamma_i} g ~ds + \gap(\gamma_i) \\
&\leq & C\Lambda\tau r + \delta \tau_0^{1/p}r\alpha^p_{r_1,2C}(L,\Lambda\tau)+ \epsilon.
\end{eqnarray*}
Letting first $\epsilon \to 0$, 
taking suprema 
over 
$g$ and $y$ and $x$,
and dividing by $r$, we obtain 
$$
\alpha^p_{r_1,2C}(L,\tau) \leq C\Lambda\tau + \delta \tau_0^{1/p} \alpha^p_{r_1,2C}(L,\Lambda\tau).
$$
Finally combining this with Equation \eqref{eq:sublin}, our initial choice of $\Lambda$ yields 

$$
\alpha^p_{r_1,2C}(L,\tau) \leq \frac{2D^{5/p}C}{\tau_0^{1/p}}\tau + 2\delta D^{5/p}\alpha^p_{r_1,2C}(L,\tau),
$$
and solving for $\alpha^p_{r_1,2C}(L,\tau)$ 
gives
$$
\alpha^p_{r_1,2C}(L,\tau) \leq \frac{2D^{5/p}C}{\tau_0^{1/p}(1-2\delta D^{5/p})}\tau =
C'\tau
$$
as desired.
\end{proof}

We now have all the tools to prove 
Theorems \ref{thm:classification1} and \ref{thm:classification2}.  The argument for the first result is similar to the one presented in \cite{sylvester:poincare}, so we only sketch the details.

\begin{proof}[Proof of Theorem \ref{thm:classification1}]
 Assume that the space satisfies a $(1,p)$-Poincar\'e inequality at scale $r_0$ with constant $C_1=C$, so by Theorem \ref{thm:classificationAp} it also satisfies a pointwise $(1,p)$-Poincar\'e inequality at scale $r_0/2$ with constant $C_2$.
To prove the maximal connectivity condition,
fix $x,y \in X$, put $r=d(x,y)$, fix $\tau \in (0,1)$, and fix a Borel set $E$ with $M_{C_2r} 1_E (z)
< \tau$ for $z = x,y$. By Remark \ref{rmk:reduction} it is sufficient to assume $E$ open. We will construct a curve $\gamma$ with controlled length and which almost avoids the set $E$. Define 
 $$\mathcal{F}_x(z) = \inf_{\gamma} \int_{\gamma} (1_E + \tau)~ds.$$
 The infimum is taken over rectifiable curves $\gamma$ connecting $x$ to $y$.
 
 Since the space is $\Lambda$-quasiconvex at scale $r_0/2$ with $\Lambda$ depending only on $C$ and $D$ 
 (see e.g. \cite{ChDiff99}), this infimum is finite.\footnote{This step requires a proof using a local Poincar\'e inequality which is a fairly straightforward modification of the previous one. See, e.g. \cite[Proposition 4.8]{bjornlocal}.} It is easy to see that $\Lip[\mathcal{F}_x] \leq \Lambda(1_E + \tau)$. Thus, by the pointwise Poincar\'e inequality we have
 $$
 \mathcal{F}_x(y) = 
 \mathcal{F}_x(y) - \mathcal{F}_x(x) \leq C_2 \Lambda r (M_{C_2r} 1_E (x) + M_{C_2r} 1_E (y) + 2\tau).
 $$
 Thus, there must be some curve $\gamma$ such that
 $$
 \int_{\gamma} (1_E + \tau)~ds \leq 
 C_2 \Lambda r \big(
 M_{C_2r} 1_E (x) + M_{C_2r} 1_E (y) + 3\tau
 \big) < 
 C_2\Lambda (2\tau+3\tau) r < 
 6C_2\Lambda \tau r.
 $$
 In particular, 
 $\len(\gamma) \leq 6C_2\Lambda r$.  The same inequality also
verifies the $(C_0,\Delta,p)$-maximal connectivity condition \eqref{eq:finemax1} for $\gamma$ with constants $C_0 = 6C_2\Lambda$ and $\Delta = 6C_2\Lambda$.
\end{proof}

\begin{proof}[Proof of Theorem \ref{thm:classification2}]
Let $\delta_{p,D}=\frac{1}{2}D^{-5/p}$.  If the space is $(C,\delta,\tau_0,p)$-max connected and $\delta \in (0,\frac{1}{2}D^{-5/p})$, then by Lemma \ref{lem:submultip} we have
$$
\alpha_{r_1,2C}^p\left(L,\tau\right)
\leq C'\tau$$
for $r_1 = r_0/5C$, with $2C \leq C'$.  So by Lemma \ref{lem:alphachar}, the space satisfies a $(1,p)$-Poincar\'e inequality at scale $r_1/(2C') = r_0/(10CC')$ with constant $C_p$, where $C_p$ depends quantitatively on $C'$ and hence on $\delta,D,C,\tau_0$, and $p$.
\end{proof}

\appendix
\section{On preserving uniformity by removal processes} \label{a:cutout}

Here we give a proof of Theorem \ref{thm:cutout}, our main technical tool in the construction of metric sponges.  This requires some preliminary lemmas for uniform domains.

\subsection{Initial properties of the measure}

One
useful property of a uniform domain $\Omega$ 
corresponds roughly to the boundary $\partial \Omega$ being porous (see e.g. \cite{bonkheinrho} for a definition).  
We recall 
a variant of \cite[Lemma 4.2]{bjornuniform} first, and sketch the proof.

\begin{lemma}[Bj\"orn-Shanmugalingam, \cite{bjornuniform}] \label{lem:corkscrew} 
If $\Omega$ is an $A$-uniform subset of $X$ then it satisfies the 
following corkscrew condition:\
for all $x \in \Omega$ and $r \in (0,\diam(\Omega))$, there exists $y \in B_\Omega(x,r)$ so that 
$$
B\Big(y,\frac{r}{4A}\Big) 
\subset \Omega \cap B(x,r).
$$
\end{lemma}

\begin{proof}
Let $x \in \Omega$ and $r \in (0,\diam(\Omega))$ be arbitrary. Choose $y \in \Omega$ so that
$$
d(x,y) \geq \frac{\diam(\Omega)}{2}.
$$
Then, let $\gamma$ be the $A$-uniform curve connecting $x$ to $y$. By continuity, there is a $t$ such that $d(\gamma(t), x) = r/4$, and thus also $d(\gamma(t),y) \geq r/4$. Therefore,
$$
d(\gamma(t), \Omega^c) \geq \frac{1}{A} \min\{\diam(\gamma|_{[0,t]}), \diam(\gamma|_{[t,1]})\} \geq \frac{r}{4A},
$$
and thus $B(\gamma(t), \frac{r}{4A}) \subset \Omega$ and $$
B\big(\gamma(t), \frac{r}{4A}\big) \subset
B\big(\gamma(t), \frac{r}{2}\big) \subset
B(x,r),
$$which completes the proof.
\end{proof}

From this we conclude useful properties of the restricted measure on $\Omega$, such as Ahlfors regularity 
and a basic volume (or measure) estimate for removed ``
obstacles.''

\begin{proof}[Proof of Lemma \ref{lem:unifarreg}]
  Let $x \in X,r \in (0,\diam(\Omega))$ and let $C_{AR, \Omega}=(4A)^Q C_{AR}$. Firstly, the upper bound in the Ahlfors $Q$-regularity condition is trivial:
  $$
	\mu(B(x,r)\cap\Omega)) \leq \mu(B(x,r)) \leq C_{AR} r^Q \leq C_{AR, \Omega} r^Q.
	$$
  Now, by 
	Lemma \ref{lem:corkscrew} 
	there is a $y \in B(x,r) \cap \Omega$ such that $B(y,\frac{r}{4A}) \subset \Omega$, in which case 
  $$
	\mu(B(x,r) \cap \Omega)) \geq \mu\left(B\left(y,\frac{r}{4A}\right)\right) \geq \frac{r^Q}{(4A)^Q C_{AR}} \geq \frac{r^Q}{C_{AR, \Omega}}
	$$
	and the result follows.
\end{proof}

\begin{proof}[Proof of Lemma \ref{lem:volobstacles}]
Scale the statement so that $\diam(\Omega)=1$. Fix $C_\delta = C_{AR}^3(2L(1+\delta))^Q \delta^{-Q}$
and, for $l > k$, let $\mathcal{R}_{x,r}^l$ be the set of all $R \in \mR_{\nseq,l}$ so that $R \cap B(x,r) \neq \emptyset$. 
 It is sufficient to prove that
$$
\mu\Big(
\bigcup_{R \in \mR_{x,r}^l} R
\Big) \leq \frac{C_\delta}{n_l^Q} r^Q
$$
for every $l > k$;
the desired estimate follows from summation over $l$.

Given $R \in \mR_{x,r}^l$ let 
$x_R \in R \cap B(x,r)$, so $R \subset B(x_R, Ls_l)$ follows from Definition \ref{def:sparsecoll}. 
Since $r \geq s_k> s_l$ we have 
$$
B(x_R, \delta s_{l-1}/2) \subset 
B(x,r+\delta s_k) \subset B(x,(1+\delta)r).
$$
By separation, the balls $B(x_R, \delta s_{l-1})$ are disjoint for distinct $R$.
We then estimate using Ahlfors regularity
\begin{eqnarray*}
 \mu\Big(B(x,r) \cap \bigcup_{R \in \mR_{\nseq,l}} R\Big) ~\leq~
 \sum_{R \in \mR_{x,r}^l} 
 \mu(R) &~\leq~& 
 \sum_{R \in \mR_{x,r}^l} 
 \mu(B(x_R, Ls_l)) \\ &~\leq~&
 \frac{C_{AR}^2 (2Ls_{l})^Q}{\delta^Q s_{l-1}^Q}
 \sum_{R \in \mR_{x,r}^l} 
 \mu(B(x_R, \delta s_{l-1}/2)) \\ &~\leq~&
\frac{C_{AR}^2 2^Q L^Q}{\delta^Q n_l^Q}\mu(B(x,(1+\delta)r)) \\ &~\leq~&
\frac{C_{AR}^3(2L(1+\delta))^Q}{\delta^Q n_l^Q} r^Q ~=~
\frac{C_\delta }{n_l^Q} r^Q.
\end{eqnarray*}
as desired.
\end{proof}

\subsection{Preserving uniformity}
One of the forthcoming technical issues in removing a set $R$ 
is that an arbitrary uniform curve relative to a pair of points in $X \setminus R$ may travel ``too far away'' from $R$.  To resolve this, we verify the following result, in whose proof we use the argument from \cite[Theorem 4.1]{vaisala1988}. 

To fix
notation, for a metric space $X=(X,d)$ and for $\epsilon > 0$ we denote $\epsilon$-neighborhoods of subsets $Y$ of $X$ by
$$
N_\epsilon(Y) := \bigcup_{x \in Y} B(x,\epsilon).
$$

\begin{lemma}\label{lem:collaring}
Fix $D,C,A\geq 1$. Let $X$ be a $C$-quasiconvex, $D$-metric doubling metric space.  If $S$ is a bounded, $A$-co-uniform domain in $X$, then 
for every $\epsilon >0$ there is a constant $L_{\epsilon} = L_\epsilon(C,D,A)$ such that for every $x,y \in N_{\epsilon \diam(S)}(S) \setminus S$, there exists a $L_\epsilon$-uniform curve $\gamma$ with respect to $x$, $y$, and $X \setminus S$ with $\gamma \subset N_{4(C+A^2)\epsilon\diam(S)}(S)$.
\end{lemma}

\begin{proof}
The statement is scale invariant, so
assume $\diam(S)=1$.  Fix $\epsilon>0$. 
Let $x,y \in N_{\epsilon}(S) \setminus S$ be arbitrary. If $d(x,y) \leq \epsilon$, the result follows simply by choosing the $A$-uniform curve with respect to $x$, $y$, and $X \setminus S$. Thus assume $d(x,y) > \epsilon$, 
in which case
$$
d(x,y) \leq 2\epsilon + \diam(S) \leq 2\epsilon + 1.
$$
Let $S_\epsilon$ be a maximally $\epsilon$-separated subset of $N_{
C
\epsilon}(S) \setminus S$, that is for each distinct $a,b \in S_\epsilon$ we have $d(a,b) \geq \epsilon$. The union $ \bigcup_{s \in S_\epsilon} B(s,2\epsilon)$ covers $N_{C\epsilon}(S) \setminus S$, so by quasiconvexity, connectivity of $\partial S$, and doubling, 
there exists $M_0 \in \N$ with dependence $M_0 = M_0(\epsilon, C, D)$ as well as a chain of points
$\{x_i\}_{i=1}^M$ in $S_\epsilon\cup\{x,y\}\subset N_{C\epsilon}(S) \setminus S$ satisfying $x_1=x$, $x_M=y$, $3\leq M \leq M_0$, and
$$
\frac{\epsilon}{2} ~\leq~ d(x_i,x_{i+1}) ~<~ 2\epsilon.
$$
Note, quasiconvexity is used simply to ensure that the points $x,y$ can be connected to $\partial S$.
For $i=1, \dots, M-1$, let $\gamma_i \co [0,1] \to X$ be the $A$-uniform curve with respect to $x_i$, $x_{i+1}$, and $X \setminus S$, so 
$\diam(\gamma_i) \leq 2A\epsilon$.
By continuity, there exists $t_i \in [0,1]$ such that
$$
\frac{\epsilon}{4} \leq
\min\{
\diam(\gamma_i|_{[0,t_i]}), \diam(\gamma_i|_{[t_i,1]})
\}.
$$
Then for $i=1, \dots, M-2$, let 
$\gamma_i'$ 
be the $A$-uniform curve with respect to $\gamma_i(t_i)$, $\gamma_{i+1}(t_{i+1})$, and $X \setminus S$.  Define $\gamma$ to be the concatenation of $\gamma_1|_{[0,t_1]}$ with $\gamma_M|_{[t_{M-1},1]}$ and all the $\gamma_i'$. 
Direct calculation and Definition \ref{def:unifdom} imply that 
\begin{align*}
\diam(\gamma_i') ~\leq&~
A d(\gamma_i(t_i),\gamma_{i+1}(t_{i+1})) \\ ~\leq&~
A\big(
d(\gamma_i(t_i),x_{i+1}) + d(x_{i+1},\gamma_{i+1}(t_{i+1}))
\big) \leq
A\big(
\diam(\gamma_i) + \diam(\gamma_{i+1})
\big) \leq
4A^2\epsilon,
\end{align*}
and 
$d(\gamma_i'(t), S) \geq \frac{\epsilon}{8A^2}$ 
for $t \in [0,1]$.  Now,
$$
\diam(\gamma) \leq 4MA^2\epsilon \leq 4MA^2 d(x,y).
$$
Also, if $\gamma(t)$ intersects with $\gamma_i'$ then
$$
{
d(\gamma(t), S) \geq 
\frac{\epsilon}{8A^2} \geq 
\frac{\diam(\gamma)}{32MA^4} \geq \frac{1}{32MA^2} \min\{\diam(\gamma|_{[0,t]}),\diam(\gamma|_{[t,1]})\}}.
$$
As for the cases when $\gamma(t)$ coincides with a point on $\gamma_1(s)$ or $\gamma_M(s)$, 
the estimate follows from the $A$-uniformity of $\gamma_1$ and $\gamma_M$. To clarify, this involves some case checking. We expand only the case of $\gamma(t)$ coinciding with $\gamma_1(s)$, when we have $d(\gamma(t),\Omega^c) = d(\gamma_1(s),\Omega^c) \geq \frac{1}{A}\min\{\diam(\gamma_1|_{[0,s]}),\diam(\gamma_1|_{[s,1]})\}$. We also have $\diam(\gamma_1|_{[0,s]})=\diam(\gamma|_{[0,t]})$, so if the minimum is attained with $\diam(\gamma_1|_{[0,s]})$ the inequality is immediate. If the minimum is attained by the second option, then  we have $\diam(\gamma_1|_{[s,1]})\geq \epsilon/4 \geq \frac{1}{4A}\diam(\gamma|_{[0,t]}) $ by the choice of $t_1$. In combination, we get that $\gamma$ is an $32MA^4$-uniform curve contained in $N_{2(C+A)\epsilon \diam(S)}(S)$. The containment follows since $\gamma_i' \subset N_{2(C+A)\epsilon \diam(S)}(S)$.
\end{proof}

We will need the following simple lemma on uniform domains.

\begin{lemma} \label{lem:basicuniflemma} Let $\Omega$ be an open domain and let $x,y \in \Omega$.  If $\gamma \co [0,1] \to \Omega$ is an $A$-uniform curve with respect to $x$, $y$, and $\Omega$, then for every $t \in [0,1]$ it holds that
$$
{
d(\gamma(t), \Omega^c) \geq \frac{1}{4A}\min\{d(x, \Omega^c) + \diam(\gamma|_{[0,t]}),d(y, \Omega^c) + \diam(\gamma|_{[t,1]})\}.}
$$
\end{lemma}

\begin{proof} {
Up to symmetry, assume $\diam(\gamma|_{[0,t]}) \leq \diam(\gamma|_{[t,1]})$
If
$$
\diam(\gamma|_{[0,t]}) \geq \frac{d(x, \Omega^c)}{2},
$$
then the claim follows from $A$-uniformity. If on the other hand
$$
\diam(\gamma|_{[0,t]}) \leq
\frac{d(x, \Omega^c)}{2},
$$
then,
by the Triangle inequality,
\begin{align*}
d(\gamma(t), \Omega^c) \geq&~
d(x,\Omega^c) - d(\gamma(t), x) \\ \geq&~
d(x,\Omega^c) - \diam(\gamma|_{[0,t]}) \geq 
\frac{d(x, \Omega^c)}{2} 
\geq \frac{1}{4} d(x,\Omega^c) + \frac{1}{4} \diam(\gamma|_{[0,t]})
\end{align*}
which, with $A \geq 1$, is the desired result.
}
\end{proof}

We are now ready to show that for co-uniform subsets $S$ of uniform domains $\Omega$, their relative complements $\Omega \setminus S$ are also uniform.

\begin{proof}[Proof of Theorem \ref{thm:cutout}]
Let $A_\Omega = A_1 > 0$ and $A_S = A_2 > 0$ be the uniformity constants of $\Omega$ and $X \setminus S$, respectively. Fix $\epsilon = \frac{d(S,\Omega^c)}{\diam(S)}$. Without loss of generality assume $\diam(S)=1$. 
Letting $
\delta_0 \in (0,\min\{1/A_{\Omega},1/A_{S}\})$ to be determined later,
we show that $\Omega \setminus S$ is $A'$-uniform for some $A' \geq 1/\delta_0$, i.e.\ that for each $x,y \in \Omega \setminus S$ there is a curve $\gamma$ so that 
\begin{equation} \label{eq:uniformfirst}
d(\gamma(t), \Omega^c \cup  S) ~\geq~
\frac{1}{A'} \min\{
{\rm diam}(\gamma|_{[0,t]}),
{\rm diam}(\gamma|_{[t,1]})
\}
\end{equation}
and where ${\rm diam}(\gamma) \leq A' d(x,y)$.


Let $x,y \in \Omega \setminus S$ be arbitrary. If $d(x,y)<\frac{\epsilon}{3(A_S + A_{\Omega})}$, the claim follows by either using the uniformity of $X \setminus S$ or the uniformity of $\Omega$, depending on which of $S$ or $\Omega^c$ is closer to $x$ or $y$. Thus, without loss of generality assume $d(x,y) \geq \frac{\epsilon}{3(A_S + A_{\Omega})}$. Also, without loss of generality, assume $x,y \notin \partial S$. The case of either $x,y \in \partial S$ can be obtained by using the uniformity of $\Omega$ to connect points $x',y' \in \Omega \setminus \bar{S}$ to $x$, $y$, respectively, with
$$
\max\{ d(x,x'),d(y,y') \} ~\leq~
\frac{1}{A_{\Omega}^2}d(S, \Omega^c).
$$
By uniformity of $\Omega$, there is an $A_\Omega$-uniform
curve $\gamma_0 \co [0,1] \to X$ 
with respect to $x$, $y$, and $\Omega$, 
so define the set
$$
\mathcal{B} = \big\{
t \in [0,1] ~|~
d(\gamma_0(t), S) < \delta_0 \min\{
{\rm diam}(\gamma_0|_{[0,t]}),
{\rm diam}(\gamma_0|_{[t,1]})
\}
\big\}.
$$
If $\mathcal{B} = \emptyset$ then $\gamma_0$ satisfies \eqref{eq:uniformfirst} with $\delta_0$ in place of $\frac{1}{A'}$, and thus $\gamma = \gamma_0$ would be the desired  
curve. 
Otherwise, $\mathcal{B}$ is open, and 
hence a countable union of disjoint open intervals,
$$
\mathcal{B} = \bigcup_{i \in J} I_i 
$$
for some possibly finite subset $J \subset \N$
with $I_i = (a_i,b_i)$. 
Note that for each $z=a_i, b_i$ we have equality in the above condition, that is
\begin{equation} \label{eq:cutoutendpts}
d(\gamma_0(z), S) ~=~
\delta_0 \min\{
{\rm diam}(\gamma_0|_{[0,z]}),
{\rm diam}(\gamma_0|_{[z,1]})
\}.
\end{equation}
Let $C' = 
4(C+A_S^2+\epsilon)$ and let $L = 
L_{\epsilon/(3C')}
= L_{\epsilon/(3C')}(C,D,A_S)$ be the constant from Lemma \ref{lem:collaring}. We now replace each $\gamma_0|_{ I_i}
$ 
with a new curve $\gamma_i$ so that the concatenation 
satisfies \eqref{eq:uniformfirst}; 
in particular, we claim that we can choose $\gamma_i$ to have
\begin{equation} \label{eq:cutout}
d(\gamma_i, \Omega^c) \geq 
\frac{ \max\{d(\gamma_0(a_i), S), d(\gamma_0(b_i), S)\}}{4} + \frac{\epsilon}{12C'}d(\gamma_0(a_i),\gamma_0(b_i))
\end{equation}
and with the constant $C'' = \max\{L, A_S\}$,
\begin{equation} \label{eq:cutdiam}
\diam(\gamma_i) \leq C'' 
d(\gamma_0(a_i),\gamma_0(b_i))
\end{equation}
holds for each $i \in \mathbb{N}$.%

We proceed by cases, as follows. Suppose first that
\begin{equation} \label{eq:cutoutcases}
\frac{\epsilon}{3 C'
} > \max\{
d(\gamma_0(a_i),S),
d(\gamma_0(b_i),S)
\}
\end{equation}
is true.  So by
Lemma \ref{lem:collaring}
with $\epsilon/(3C')$ in place of $\epsilon$,
there is a curve $\gamma_i$ in $N_{\epsilon/3}(S)$ that joins $\gamma_0(a_i)$ and $\gamma_0(b_i)$ and which is $L$-uniform with respect to $X \setminus S$. In particular, \eqref{eq:cutdiam} holds with
$C''=L$ 
and 
our choice of $\epsilon$ yields
$$
d(\gamma_i(t), \Omega^c) ~\geq~
d(S, \Omega^c) - d(\gamma_i(t), S) ~\geq~
\epsilon - \frac{\epsilon}{3}
$$
so \eqref{eq:cutout} follows from
\eqref{eq:cutoutcases} and 
$$
d(\gamma_0(a_i),\gamma_0(b_i)) \leq \diam(S) + d(\gamma_0(a_i),S) + d(\gamma_0(b_i),S) \leq 
1+ \frac{2\epsilon}{C'}.
$$

If \eqref{eq:cutoutcases} is false 
then instead
by co-uniformity, there is a $A_S$-uniform curve $\gamma_i$ with respect to $\gamma_0(a_i)$, $\gamma_0(b_i)$, and $X \setminus S$. 
We now claim that 
the distance estimates \eqref{eq:cutout}, \eqref{eq:cutdiam}
hold for these curves $\gamma_i$.  To this end,
by symmetry we may assume that 
\begin{eqnarray*}
d(\gamma_0(a_i),S) ~\geq~ 
\max\Big\{
\frac{\epsilon}{3 C'},
d(\gamma_0(b_i), S)
\Big\}.
\end{eqnarray*}
Introduce the short-hand notation $x_i\defeq \gamma_0(a_i), y_i \defeq \gamma_0(b_i)$. Assume now that
$
\delta_0 < \frac{\epsilon}{32A_{\Omega}A_SC'}
$,
which with \eqref{eq:cutoutendpts} implies that
\begin{align*}
d(x_i, \Omega^c) ~\geq&~ 
\frac{1}{A_{\Omega}} \min\{{\rm diam}(\gamma_0|_{[0,a_i]}), {\rm diam}(\gamma_0|_{[a_i,1]}) \} = 
\frac{1}{\delta_0 A_{\Omega}} d(x_i, S) \geq 
\frac{1}{\delta_0 A_{\Omega}} \frac{\epsilon}{3 C'}.
\end{align*}
Then combining the previous estimates and the choice of $\delta_0$ yields
$$
d(x_i,y_i) \leq \diam(S) + 2d(x_i,  S) \leq \frac{6C'\delta_0 A_\Omega}{\epsilon}d(x_i, \Omega^c) \leq \frac{1}{8A_S}d(x_i, \Omega^c),
$$
and
$$
d(x_i,y_i) \leq \diam(S) + 2d(x_i, S) \leq \frac{6C'}{\epsilon}d(x_i, S).
$$
We have \eqref{eq:cutdiam} and therefore
\begin{align*}
d(\gamma_i, \Omega^c) \geq 
d(x_i,\Omega^c) - \diam(\gamma_i) \geq&~
d(x_i,\Omega^c) - A_S d(x_i,y_i) \\ \geq&~
\frac{d(x_i, S)}{2} =
 \frac{ \max\{d(x_i, S), d(y_i, S)\}}{2}.
\end{align*}
In particular, \eqref{eq:cutout} holds in both cases for the $\gamma_i$ as constructed.

In either case,  $C''$-uniformity of $\gamma_i$ with respect to $X \setminus S$ and 
Lemma \ref{lem:basicuniflemma} imply that  for all $t$ in the domain of $\gamma_i$
\begin{equation} \label{eq:cutout2}
d(\gamma_i(t), S) \geq 
\frac{1}{C''}
\frac{ \min\{d(x_i, S) + \diam(\gamma_i|_{[a_i,t]}), d(y_i, S) + \diam(\gamma_i|_{[t,b_i]})\}}{4}.
\end{equation}
Now, similarly to the proof of Lemma \ref{lemma:quasiconvex}  re-parametrize each 
$\gamma_i$ to have domain 
$I_i=[a_i,b_i]$ 
and define 
the concatenation 
$\gamma \co [0,1] \to X$ by $\gamma(t) = \gamma_i(t)$ if $t \in  I_i
$, 
and $\gamma(t)=\gamma_0(t)$ for all other $t \in [0,1]$. This concatenated curve is the desired uniform curve and we will proceed to estimate its diameter and distance to $S \cup \Omega^c$.

The diameter bounds for $\gamma_i$ in \eqref{eq:cutdiam} give rather directly that $\gamma$ is continuous.
%
By \eqref{eq:cutdiam} each $\gamma_i$ has diameter 
at most
$$
\diam(\gamma_i)
\leq
C'' 
d(x_i, y_i) \leq
C'' 
\diam(\gamma_0),
$$
so it follows that the concatenation $\gamma$ has diameter 
at most
\begin{eqnarray*}
\diam(\gamma) \leq \diam(\gamma_0) + 2\max_i \diam(\gamma_i)
 \leq
(1+2C'')
\diam(\gamma_0) \leq
A_\Omega (1+2C'')
d(x,y).
\end{eqnarray*}
 To check the uniformity condition \eqref{eq:unifdom}, 
 we again proceed by cases. Supposing first that $t \notin  I_i
$ for any index $i$, 
put $U_0 = [0,t]$ and $U_1 = [t,1]$.
For $k=0,1$
we have from \eqref{eq:cutdiam}
\begin{equation}\label{eq:lenbound}
\diam(\gamma|_{U_k}) \leq 
\diam(\gamma_0|_{U_k}) + 2\max_{i, I_i\subset U_k}\diam(\gamma_i|_{ I_i}) \leq 
(1+2C'')
\diam(\gamma_0|_{U_k})
\end{equation}

 Then, we get
\begin{equation} \label{eq:goodtext}
  d(\gamma(t), \Omega^c) \geq 
  \frac{1}{A_{\Omega}} \min_{k=0,1}{\rm diam}(\gamma_0|_{U_k}) \geq 
  \frac{
  1 
  }{A_{\Omega}(1+2C'')}
  \min_{k=0,1}{\rm diam}(\gamma_0|_{U_k}),
\end{equation}
and (from the definition of $\mathcal{B}$)
\begin{equation}\label{eq:goodtint}
 d(\gamma(t), S) \geq 
 \delta_0 \min_{k=0,1}{\rm diam}(\gamma_0|_{U_k}) \geq 
 \frac{
 \delta_0
 }{(1+2C'')}
  \min_{k=0,1}{\rm diam}(\gamma_0|_{U_k}).
\end{equation}

Now 
consider the remaining case where $t \in  I_i
$ for some $i \in J$, 
in which case $U_k \cup { I_i}$ and $U_k \cap  I_i$ and $U_k \setminus  I_i$ are all intervals for $k=0,1$. Similarly as above,
\begin{equation}\label{eq:bound-gambeg}
\diam(\gamma|_{U_k}) \leq 
\diam(\gamma_0|_{U_k \setminus I_i}) + \diam(\gamma_i|_{U_k \cap I_i}) \leq 
(1+2C'')
(\diam(\gamma_0|_{U_k \setminus  I_i}) + 
d(x_i,y_i)).
\end{equation}

Taking a minimum over $k=0,1$ in \eqref{eq:bound-gambeg} gives
\begin{align}
\label{eq:minbounds} 
\min_{k=0,1} \diam(\gamma|_{U_k}) \leq 
(1+2C'')(\min_{k=0,1} \diam(\gamma_0|_{U_k \setminus I_i})
+ d(x_i,y_i)). 
\end{align}
Combining our work with $\frac{\epsilon}{12C'}\geq \delta_0$ gives the following.
\begin{align*}
d(\gamma(t), \Omega^c) \stackrel{\eqref{eq:cutout}}{\geq}&~ 
\frac{ \max\{d(x_i, S), d(y_i, S)\}}{4} \ + \  \frac{\epsilon}{12C'}d(x_i,y_i) \\
\stackrel{\tiny\begin{array}{c}\eqref{eq:cutoutendpts} \\ \eqref{eq:minbounds}\end{array}}{\geq}\!\!\!\!&~
\frac{ \delta_0 \displaystyle \min_{k=0,1} \diam(\gamma|_{U_k}) }
{1+2C''}
\\
d(\gamma(t), S) \stackrel{\eqref{eq:cutout2}}{\geq}&~
\frac{1}{C''}
\frac{ \min\{d(x_i, S) + \diam(\gamma_i|_{[a_i,t]}), d(y_i, S) + \diam(\gamma_i|_{[t,b_i]})\}}{4} \\
\stackrel{\eqref{eq:cutoutendpts}}{\geq}&~
\frac{\delta_0}{4C''}
 \min_{k=0,1}\{\diam(\gamma_0|_{U_k \setminus  I_i}) + \diam(\gamma_i|_{U_k \cap  I_i}), \diam(\gamma_0|_{U_k \cup  I_i}) + \diam(\gamma_i|_{I_i \setminus U_k})\} \\
\stackrel{\eqref{eq:bound-gambeg}}{\geq}& \frac{\delta_0}{4C''(1+2C'')} \min_{k=0,1} \diam(\gamma|_{U_k}).
\end{align*}
In the ultimate inequality, we bound each of the terms in the minimum first, and then combine the bound. 
Now, the previous two estimates give for $t \in (a_i,b_i)$ that
\begin{equation}\label{eq:badtext}
d(\gamma(t), S \cup \Omega^c) \geq 
\frac{\delta_0}{4C''(1+2C'')}
\min\Big\{
{\rm diam}(\gamma|_{[0,t]}), 
{\rm diam}(\gamma|_{[t,1]})
\Big\}.
\end{equation}

The estimates \eqref{eq:badtext} together with the diameter bound show that the curve $\gamma$ is $A'$--uniform for
$$
A' = \max\Big\{
\frac{4C''(1+2C'')
}
{\delta_0},
(1+2C'')
A_{\Omega}
\Big\}.
$$
\end{proof}

\noindent \textbf{Acknowledgements:} {
The first author was partly supported by the NSF grant DMS-1704215. He also thanks Mario Bonk for useful conversations on the topic. 

Both authors also thank Nageswari Shanmugalingam for comments on an earlier version of the manuscript. Discussions with her lead to splitting that paper into two parts, in order to enhance readability: this one on the general case allowing for $p>1$ and the technically completely different \cite{sylvesterjasun} discussing more specific results for $p=1$. The authors are also thankful for Qingshan Zhou and Panu Lahti, as well as the referees, for a careful reading and comments.}

\bibliographystyle{amsplain}

\vfill
\end{document}